%% file: Fresse-Nishiyama-revision.tex
\documentclass[12pt]{amsart}

\usepackage[vcentermath,enableskew]{youngtab}
\usepackage{amssymb}
\usepackage[all]{xy}
\usepackage{multirow}
\usepackage{array}

\usepackage{pgf,interval}
\usepackage{ytableau}
\ytableausetup{centertableaux,nobaseline}

\usepackage{mathrsfs,calligra}
\usepackage{accents}

\usepackage[final]{showkeys}

\newcommand{\tableau}[1]{\mbox{\tiny$\ytableaushort{#1}$}}
\newcommand{\diagram}[1]{\mbox{\tiny$\ydiagram{#1}$}}
\newcommand{\tableaul}[1]{\mbox{\tiny $\begin{ytableau} #1 \end{ytableau}$}}

\newcommand{\orbitofX}{\mathfrak{Z}}

\newcommand{\version}{Ver.~0.0}
\newcommand{\setversion}[1]{\renewcommand{\version}{Ver.~{#1}}}
\setversion{0.1 [2019/01/25]}
\setversion{0.2 [2019/01/27 16:19:20 EST]}
\setversion{0.3 [2019/01/31 15:36:36 EST]}
\setversion{0.4 [2019/03/04]}
\setversion{0.5 [2019/03/17 20:09:03 EDT]}
\setversion{0.6 [2019/03/22 22:27:02 EDT]}
\setversion{0.7 [2019/04/17]}
\setversion{0.71 [2019/04/23 22:07:19 JST]}
\setversion{1.0 [date you finalize the draft]}

\newtheorem*{conclusion}{Conclusion}

\newtheorem*{program}{Program}
\newtheorem{complement-}{Complement}[section]
\theoremstyle{definition}
\newtheorem*{notation}{Notation}

\numberwithin{equation}{section}

\newcommand{\RSl}{\mathrm{RS}_1}
\newcommand{\RSr}{\mathrm{RS}_2}
\newcommand{\RS}{\mathrm{RS}}
\newcommand{\rinsert}{\mathrm{RowInsert}}
\newcommand{\cinsert}{\mathrm{ColumnInsert}}

\input def_exotic_nullcone.tex

\begin{document}

\title[A generalization of Steinberg theory and an exotic moment map]{A generalization of Steinberg theory \\ and an exotic moment map}
\author{Lucas Fresse}\thanks{L. F. is supported in part by the ANR project GeoLie ANR-15-CE40-0012.}
\address{Universit\'e de Lorraine, CNRS, Institut \'Elie Cartan de Lorraine, UMR 7502, Vandoeu\-vre-l\`es-Nancy, F-54506, France}
\email{lucas.fresse@univ-lorraine.fr}
\author{Kyo Nishiyama}\thanks{K.~N.~is supported by JSPS KAKENHI Grant Number \#{16K05070}.}
\address{Department of Physics and Mathematics, Aoyama Gakuin University, Fuchinobe 5-10-1, Sagamihara 229-8558, Japan}
\email{kyo@gem.aoyama.ac.jp}

\begin{abstract}
For a reductive group $G$, Steinberg established a map from the Weyl
group to the set of nilpotent $G$-orbits by using moment maps on
double flag varieties. In particular, in the case of the general
linear group, it provides a geometric interpretation
of the Robinson-Schensted correspondence between permutations and
pairs of standard tableaux of the same shape.

We extend Steinberg's approach to the case of a symmetric pair $(G, K)$ to obtain two different maps, 
namely a \emph{generalized Steinberg map} and 
an \emph{exotic moment map}.  

Although the framework is general, in this paper 
we focus on the pair 
$(G,K) = (\mathrm{GL}_{2n}(\C), \mathrm{GL}_n(\C)
\times \mathrm{GL}_n(\C))$. 
Then the generalized Steinberg map is a map from \emph{partial} permutations to 
the pairs of nilpotent orbits in $ \lie{gl}_n(\C) $.  
It involves a generalization of the classical Robinson-Schensted correspondence to the case of partial permutations. 

The other map, the exotic moment map, establishes a combinatorial map from the set of partial permutations
to that of signed Young diagrams, i.e., the set of nilpotent $ K$-orbits in the Cartan space
$(\mathrm{Lie}(G)/\mathrm{Lie}(K))^* $.

We explain the geometric background of the theory and 
combinatorial algorithms which produce the above mentioned maps.
\end{abstract}


\maketitle

\section{Introduction}

\subsection{}

\label{section-1.1}

Let $G$ be a connected reductive algebraic group over $\C$. Let $T\subset G$ be a maximal torus and let $B\subset G$ be a Borel subgroup containing $T$. 
In \cite{Steinberg-1976} (see also \cite{Hinich-Joseph}),  
Steinberg gave a geometric correspondence between the Weyl group $W:=W(G,T)$ and the nilpotent orbits in the Lie algebra $\mathfrak{g}:=\mathrm{Lie}(G)$ in terms of the flag variety $\mathcal{B}:=G/B$ 
(or rather the product $ \mathcal{B} \times \mathcal{B} $). Let us review it shortly.

The action of $G$ on $\mathcal{B}$ gives rise to a Hamiltonian action on the cotangent bundle 
$ T^*\mathcal{B} $, which is a symplectic variety.  
Thus we get a $ G $-equivariant moment map $\mu_\mathcal{B}:T^*\mathcal{B}\to\mathfrak{g}^*\cong\mathfrak{g}$ (see, e.g., \cite{Chriss-Ginzburg}).
Hereafter we identify $\mathfrak{g}$ with $\mathfrak{g}^*$ via a fixed nondegenerate invariant bilinear form on $\mathfrak{g}$. 
In fact, the image of $\mu_\mathcal{B}$ is contained in the nilpotent cone $\mathcal{N}=\mathcal{N}_\mathfrak{g}$.

Let us consider the double flag variety $\mathcal{B}\times\mathcal{B}$ with the diagonal $ G $-action.  Then, again its cotangent bundle $T^*(\mathcal{B}\times\mathcal{B})$ is a Hamiltonian $ G $-variety, and we get a moment map $\mu_{\mathcal{B}\times\mathcal{B}}:T^*(\mathcal{B}\times\mathcal{B})\to\mathfrak{g}$.  
Note that, by functoriality, 
$ \mu_{\mathcal{B}\times\mathcal{B}} $ is just a sum of two moment maps on each $ T^*\mathcal{B}$.
We denote the null fiber of the moment map by $ \mathcal{Y} = \mu_{\mathcal{B}\times\mathcal{B}}^{-1}(0) $, and call it a \emph{conormal variety} (or Steinberg variety).
We summarize the situation in the diagram below.
\[
\xymatrix{T^*(\mathcal{B}\times\mathcal{B}) \ar[r]^{\mathrm{pr}_2}\ar[d]^{\mathrm{pr}_1}\ar@{-->}[rd]^{\mu_{\mathcal{B}\times\mathcal{B}}} & T^*\mathcal{B} \ar[d]^{\mu_\mathcal{B}} \\
T^*\mathcal{B} \ar[r]^{\mu_\mathcal{B}} & \mathcal{N}&\hspace{-1.8cm}\subset\mathfrak{g}}
\qquad
\begin{array}[t]{l}
\mu_{\mathcal{B}\times\mathcal{B}}=\mu_\mathcal{B}\circ\mathrm{pr}_1+\mu_\mathcal{B}\circ\mathrm{pr}_2, \\[5mm]
\mathcal{Y}=\mu_{\mathcal{B}\times\mathcal{B}}^{-1}(0)=T^*\mathcal{B}\times_\mathcal{N}T^*\mathcal{B}.
\end{array}
\]
It is interesting to note
that $\mu_\mathcal{B}:T^*\mathcal{B}\to\mathcal{N}$ is a resolution of singularities (the {\em Springer resolution}), so the Steinberg variety is a fiber product of two copies of this resolution.

From general facts on conormal varieties (see Section \ref{newsection-3}), we conclude that
the conormal variety $\mathcal{Y}$ is equidimensional and its irreducible components are param\-etrized by the $G$-orbits in $\mathcal{B}\times\mathcal{B}$, which are in one-to-one correspondence with the Weyl group $W$ due to the Bruhat decomposition (just note that $(\mathcal{B}\times\mathcal{B})/G\cong B\backslash G/B$).  

The map $\pi_1:=\mu_\mathcal{B}\circ\mathrm{pr}_1$
maps every irreducible component of $\mathcal{Y}$ 
to an irreducible $ G $-stable closed subset in the nilpotent variety $ \nilpotentsg $
(using $\pi_2:=\mu_\mathcal{B}\circ\mathrm{pr}_2$ leads to the same result).  
Since $ \nilpotentsg $ has only finitely many orbits (due to the Jacobson--Morozov Lemma combined with Malcev's Theorem), 
this image is 
the closure of a single nilpotent
$G$-orbit. 
In this way, we finally obtain a map from the Weyl group to the set of nilpotent $G$-orbits:
\[\Stmap :W\to \mathcal{N}_\mathfrak{g}/G , \]
which we call the {\em Steinberg map}.

When $G=\mathrm{GL}_n(\C)$, the Weyl group $ W $ coincides with the symmetric group 
$\mathfrak{S}_n$
and the nilpotent orbits $ \nilpotentsg / G $ are classified in terms of the Jordan normal form, 
i.e., a partition $ \lambda \vdash n $.  
Let us denote the set of partitions of $ n $ by $ \partitionsof{n} $.  
Then the Steinberg map in this case gives 
a map from $\mathfrak{S}_n$ to $\partitionsof{n}$ 
which yields a geometric interpretation of the Robinson-Schensted correspondence \cite{Steinberg2}.
A more detailed review of Steinberg's theory is given in Section \ref{section-Steinberg} below.

\subsection{}\label{section-1.2}

One of our main goals in this paper is to extend Steinberg's approach to the case of a symmetric pair $(G,K)$, where $K$ is the subgroup of fixed points of an involution $\theta\in\mathrm{Aut}(G)$.  
The subgroup $ K $ is called a symmetric subgroup, and it is automatically reductive.  
Such pairs $ (G, K) $ are completely classified 
and the quotient space $ G/K $ is called a symmetric space (see, e.g., 
\cite{Helgason.DG.1978}).
In what follows, we assume for simplicity that $K$ is connected (otherwise one should replace $K$ with its connected component of the identity). 
Note that $K$ is connected, e.g., if $G$ is semisimple and simply connected (see \cite[\S8]{Steinberg-1968}).

Let $P$ and $Q$ be parabolic subgroups of $ G $ and $ K $ respectively.
We define a {\em double flag variety} for the symmetric pair as a product of 
two flag varieties 
\[\mathfrak{X}:=K/Q\times G/P.\]
The group $K$ acts diagonally on $\mathfrak{X}$.  
As in Section \ref{section-1.1}, the Hamiltonian action of $K$ on the cotangent bundle $T^*\mathfrak{X}$ gives rise to a moment map $\mu_\mathfrak{X}:T^*\mathfrak{X}\to\mathfrak{k}^*\cong\mathfrak{k}$, where $\mathfrak{k}:=\mathrm{Lie}(K)=\mathfrak{g}^\theta$, and we define a conormal variety $\mathcal{Y}:=\mu_\mathfrak{X}^{-1}(0)$ as the null fiber of this moment map.  

In general there are infinitely many $ K $-orbits in $ \mathfrak{X} $, 
though in many interesting cases the number of orbits is finite. 
In this case, we say that $ \gerX $ is \emph{of finite type}. 
A number of examples of double flag varieties of finite type are given in \cite{Nishiyama-Ochiai};
in the case where $ P $ or $ Q $ is a Borel subgroup, full
classifications of double flag varieties of finite type are obtained in \cite{HNOO}.

The condition that $ \mathfrak{X} $ is of finite type implies a lot of good properties of 
the conormal variety $\mathcal{Y}$.

\begin{proposition}
Assume that
\begin{equation}
\label{1}
\mbox{$\mathfrak{X}$ has a finite number of $K$-orbits.}
\end{equation}
Then, the conormal variety $\mathcal{Y}$ is equidimensional and its irreducible components are in one-to-one correspondence with the $K$-orbits of $\mathfrak{X}$.\ More precisely, whenever $\mathbb{O}\subset\mathfrak{X}$ is a $K$-orbit, the closure of the conormal bundle $T^*_\mathbb{O}\mathfrak{X}$ is an irreducible component of $\mathcal{Y}$.
\end{proposition}

For the proof, see Proposition \ref{P1}.

\begin{assumption}
Hereafter we assume condition \eqref{1}, 
i.e., we assume that the double flag variety $ \mathfrak{X} = K/Q \times G/P $ is of finite type.
\end{assumption}

A clear difference of the present situation from that of 
Steinberg's theory is
the lack of symmetry in the definition of the double flag variety $\mathfrak{X}$. 

Let us denote the Cartan decomposition by $ \lie{g} = \lie{k} \oplus \lie{s} $, 
where $ \lie{k} $ is the Lie algebra of $ K $ or equivalently the $ (+1) $-eigenspace of the involution $ \theta $ (we denote the differential of $ \theta \in \Aut(G) $ by the same letter) and $ \lie{s} $ is the $ (-1) $-eigenspace.  Note that $ \lie{s} $ is not a Lie algebra, but only a vector subspace isomorphic to the tangent space of $ G/K $ at the base point $ eK $.  
We summarize the situation in the diagram below.
\[
\xymatrix{
T^*\mathfrak{X} \ar[r]^{\mathrm{pr}_2}\ar[dd]^{\mathrm{pr}_1}\ar@{-->}[ddr]^{\mu_\mathfrak{X}} & T^*(G/P) \ar[d]^{\mu_{G/P}} \\
 & \mathfrak{g} \ar[d]^{(\cdot)^\theta}\ar[r]^{(\cdot)^{-\theta}} & \mathfrak{s} \\
T^*(K/Q) \ar[r]^{\mu_{K/Q}} & \mathfrak{k}
}
\qquad
\begin{array}[t]{l}
\mathfrak{g}=\mathfrak{k}\oplus\mathfrak{s}\ \mbox{(where $\mathfrak{s}=\mathfrak{g}^{-\theta}$)} \\[2mm]
\mbox{with projection maps $(\cdot)^\theta$ and $(\cdot)^{-\theta}$,} \\[2mm]
\mu_\mathfrak{X}=\mu_{K/Q}\circ\mathrm{pr}_1+(\cdot)^\theta\circ\mu_{G/P}\circ\mathrm{pr}_2, \\[2mm]
\mathcal{Y}=\mu_\mathfrak{X}^{-1}(0).
\end{array}
\]
In the above diagram, there appear two maps
\[
\pi_\mathfrak{k}:=\mu_{K/Q}\circ\mathrm{pr}_1:\mathcal{Y}\to\mathcal{N}_\mathfrak{k}\subset\mathfrak{k}
\quad
\text{and}
\quad
\pi_\mathfrak{g}:=\mu_{G/P}\circ\mathrm{pr}_2:\mathcal{Y}\to\mathcal{N}_\mathfrak{g}\subset\mathfrak{g}, 
\]
which do not play the same role.  
Actually $\pi_\mathfrak{k}=-(\cdot)^\theta\circ\pi_\mathfrak{g}$.  
We call $\pi_\mathfrak{k}$ the ``symmetrized moment map''.

Next
we define 
\[\pi_\mathfrak{s}:=(\cdot)^{-\theta}\circ \pi_\mathfrak{g}:\mathcal{Y}\to\mathfrak{s},\] 
where $ (\cdot)^{-\theta} $ denotes the projection to $ \lie{s} $ along the Cartan decomposition 
$ \lie{g} = \lie{k} \oplus \lie{s} $.  
We call $ \pi_\mathfrak{s} $ the ``exotic moment map''. 
Note that $\pi_\mathfrak{s}=\pi_\mathfrak{g}+\pi_\mathfrak{k}$.

In general the image of the exotic moment map $\pi_\mathfrak{s}$ is not necessarily 
contained in the nilpotent variety $\mathcal{N}_\mathfrak{s}:=\mathcal{N}_{\lie{g}}\cap\mathfrak{s}$.
\emph{Let us assume}:
\begin{equation}
\label{2} \mbox{the image of $\pi_\mathfrak{s}$ is contained in the nilpotent variety $\mathcal{N}_\mathfrak{s}$.}
\end{equation}

All the maps considered so far are $K$-equivariant.
Note that both nilpotent varieties $\mathcal{N}_\mathfrak{k}$ and $\mathcal{N}_\mathfrak{s}$ consist of finitely many $K$-orbits 
(see, e.g., \cite{Collingwood.McGovern.1993}).
Then for any irreducible component of $\mathcal{Y}$, its image by $\pi_\mathfrak{k}$ (resp. $\pi_\mathfrak{s}$) contains a dense nilpotent $K$-orbit.

\begin{conclusion}
Altogether, under assumptions \eqref{1} and \eqref{2}, we get two maps
\begin{equation}\label{3}
\Xik : \mathfrak{X}/K \xrightarrow{\quad} \mathcal{N}_\mathfrak{k}/K
\quad \text{ and } \quad 
\Xis : \mathfrak{X}/K \xrightarrow{\quad} \mathcal{N}_\mathfrak{s}/K.
\end{equation}
We also call $ \Xik $ a ``symmetrized moment map'' and 
$ \Xis $ an ``exotic moment map'' by abuse of terminology.
\end{conclusion}

These maps generalize the Steinberg map $W\cong (\mathcal{B}\times\mathcal{B})/G\to\mathcal{N}_\mathfrak{g}/G$ of Section \ref{section-1.1}, which is recovered by considering the particular case where $\theta=\mathrm{id}_G$ (so that $K=G$ and $\mathfrak{s}=0$) and $P=Q=B$.

\begin{remark}
It is interesting to note that the situation considered by Steinberg is also recovered if we consider the symmetric pair
$(G,K)=(G_1\times G_1,\Delta G_1)$, where $G_1$ is a connected reductive group and $\Delta G_1$ stands for the diagonal embedding of $G_1$ into $G_1\times G_1$,
and we put $ P = B_1 \times G_1 $ and $ Q = \Delta B_1 $, where $ B_1 $ is a Borel subgroup of $ G_1 $. 
\end{remark}

In practice, for a particular (e.g., classical) symmetric pair $(G,K)$, 
a combinatorial description of the sets of nilpotent orbits $\mathcal{N}_\mathfrak{k}/K$ or $\mathcal{N}_\mathfrak{s}/K$ is known (see \cite{Collingwood.McGovern.1993}). 
Then a natural problem is 
to find a combinatorial parametrization of the orbits $\mathfrak{X}/K$ and to describe the maps of (\ref{3}) in an explicit way. Let us summarize this into a program:

\begin{program} Let $\mathfrak{X}=K/Q\times G/P$ be the double flag variety corresponding to a symmetric pair $(G,K)$ and a pair of parabolic subgroups $(P,Q)$ as above,
and assume the conditions \eqref{1} and \eqref{2}.
\begin{itemize}
\item[\rm (A)] Find a combinatorial parametrization of the orbits $\mathfrak{X}/K$;
\item[\rm (B)] Compute the maps
$ \Xik : \mathfrak{X}/K\to\mathcal{N}_\mathfrak{k}/K$ and 
$ \Xis : \mathfrak{X}/K\to\mathcal{N}_\mathfrak{s}/K$
of \eqref{3} explicitly.
\end{itemize}
\end{program}

\subsection{} \label{section-1.3}
One of the main goals of this paper is to carry out the program of Section \ref{section-1.2} in the case of the symmetric pair
\[
(G,K)=(\mathrm{GL}_{2n}(\C),\mathrm{GL}_n(\C)\times\mathrm{GL}_n(\C)) 
\]
and to establish new combinatorial bijections involving \emph{partial permutations}.
We consider everything over $ \C $, and sometimes we simply write 
$ \GL_n $ or $ \Mat_n $ instead of $ \GL_n(\C) $ or $ \Mat_n(\C) $ respectively.  
In fact, all the results in this paper are still valid if we replace $ \C $ by any algebraically closed field of characteristic zero.  

Let us choose the parabolic subgroups $ P $ and $ Q $ as follows.

\medskip

\begin{itemize}
\item 
$ P = P_\mathrm{S} := 
\left\{\left(\begin{array}{cc} a & c \\ 0 & b \end{array}\right):a,b\in\GLnC,\ c\in\MatnC\right\} 
= \Stab_G(\C^n\oplus\{0\})$ 
is a maximal parabolic subgroup (Siegel parabolic subgroup);

\smallskip

\item $Q=B_K\subset K$ is a Borel subgroup.
\end{itemize}

\medskip

\noindent
Thus $K/B_K$ is the full flag variety of
$K=\GLnC\times\GLnC$ (i.e.,
the double flag variety of $\GLnC$) while
$G/P_\mathrm{S}\cong\mathrm{Gr}_n(\C^{2n})$ can be
identified with the Grassmann variety of $n$-dimensional subspaces
of $\C^{2n}$.  
In this respect, our double flag variety is isomorphic to 
$ \Flag(\C^n) \times \Flag(\C^n) \times \Grass_n(\C^{2n}) $ on 
which $ K = \GLnC \times \GLnC $ acts.

Note that 

\begin{center}
{\em In the situation under consideration, conditions (\ref{1}) and (\ref{2}) are satisfied.}
\end{center}

\noindent
This assertion follows from \cite[Table 3]{Nishiyama-Ochiai} and \cite[Proposition 4.2]{Fresse-Nishiyama}.

\medskip

Let us explain in more detail our strategy to the steps (A) and (B) in the
program for the double flag variety $\mathfrak{X}=K/B_K{\times}G/P_\mathrm{S}$ above.

(A) In Section \ref{section-8}, we give a
parametrization of the $K$-orbits in the double flag variety
$\mathfrak{X}$. It is equivalent to
parametrize the $B_K$-orbits in the Grassmann variety
$G/P_\mathrm{S}\cong\mathrm{Gr}_n(\C^{2n})$. To this end, we
consider the set $(\mathfrak{T}_n^2)'$ of $(2n)\times n$ matrices of
rank $n$ of the form $ \omega = \left(\begin{array}{c} \tau_1
\\ \tau_2 \end{array}\right)$, associated to a pair of
partial permutations $\tau_1,\tau_2\in\mathfrak{T}_n$.   
Here a partial permutation of $ [n] = \{ 1, 2, \dots, n \} $ means 
an injective map from a subset $ J \subset [n] $ to $ [n] $.  
As in the case of permutation, 
we can associate a matrix in $ \MatnC $ with a partial permutation $ \tau $, which we also denote by $ \tau $ by abuse of notation.  
See Definition \ref{D1} for details.  
Let us denote the image of the matrix $ \omega $ by $ \Im \omega $, which is an $ n $-dimensional subspace of $ \C^{2n} $ generated by the column vectors of $ \omega $.  
Thus $ \Im \omega $ represents a point in $ \Grass_n(\C^{2n}) $.

Here is a classification of the set of orbits $ \dblFV/K $ by partial permutations.

\begin{theorem}[see Theorem \ref{T3}]\label{thm:introT3}
Every $B_K$-orbit of
$\mathrm{Gr}_n(\C^{2n})$ is of the form
$B_K\cdot(\Im \omega)$ for some
$\omega\in(\mathfrak{T}_n^2)'$.
Moreover, $\omega$ and $\omega'$ represent the same $B_K$-orbit if and only if they coincide up to column permutation.  
Thus we get 
$ (\mathfrak{T}_n^2)'/\mathfrak{S}_n \simeq \dblFV/K $.
\end{theorem}

(B) Sections \ref{section-9}--\ref{section-10} are devoted
to the calculation of the maps
$\Xik:\mathfrak{X}/K\to\mathcal{N}_\mathfrak{k}/K$ and
$\Xis:\mathfrak{X}/K\to\mathcal{N}_\mathfrak{s}/K$ of (\ref{3}).  
Now by the parametrization of $ \dblFV/K $ (Theorem~\ref{thm:introT3} above), 
we can regard these maps as 
\begin{align}
\Xik: & (\mathfrak{T}_n^2)'/\mathfrak{S}_n \to \partitionsof{n}^2 ,
\\
\Xis: & (\mathfrak{T}_n^2)'/\mathfrak{S}_n \to \signpartitionsof{2n} ,
\end{align}
where $ \signpartitionsof{2n} $ denotes 
the set of signed Young diagrams of size $ 2n $ with signature $ (n,n) $ 
parametrizing the set of nilpotent $ K $-orbits $ \mathcal{N}_\mathfrak{s}/K$.  
We still call $ \Xik $ and $ \Xis $ the \emph{symmetrized} and 
\emph{exotic moment maps} respectively, 
though they are all reduced down to combinatorial way.

The description of these maps are quite involved.  
We give a combinatorial description for the $K$-orbits 
of $\mathfrak{X}$ which correspond to the matrices $\omega$ of the
form 
$
\omega
=\left(\begin{array}{c} \tau
\\ 1_n \end{array}\right)
$ where $\tau$ is a partial
permutation.  
However, the problem for more degenerate orbits remains open.
Our main results regarding step (B) can be summarized as follows.

\begin{theorem}[Theorems \ref{T4} and \ref{T5}]\label{TII}
For $\omega
=\left(\begin{array}{c} \tau
\\ 1_n \end{array}\right)
\in(\mathfrak{T}_n^2)'$, 
there are combinatorial algorithms which describe 
$ \Xik(\omega) $ and $ \Xis(\omega) $.  
The algorithm for $ \Xik(\omega) $ reduces to the classical Robinson-Schensted algorithm 
if $ \tau $ is a permutation. 
\end{theorem}

\subsection{Image of the conormal variety}
\label{section-1.4}
The setting being as in Section \ref{section-1.3},
we also consider the image of the whole conormal variety
by the maps $\pi_{\mathfrak{k}}$ and $\pi_{\mathfrak{s}}$. The
images $\pi_\mathfrak{k}(\mathcal{Y})$ and
$\pi_\mathfrak{s}(\mathcal{Y})$ are $K$-stable subsets of
$\mathcal{N}_\mathfrak{k}$ and $\mathcal{N}_\mathfrak{s}$,
respectively, which are not closed in general (see Remark
\ref{R7.1}). We define the nilpotent varieties
$\mathcal{N}_{\mathfrak{X},\mathfrak{k}}:=\overline{\pi_\mathfrak{k}(\mathcal{Y})}$
and
$\mathcal{N}_{\mathfrak{X},\mathfrak{s}}:=\overline{\pi_\mathfrak{s}(\mathcal{Y})}$.

\begin{theorem}
\label{T1}
We have $\mathcal{N}_{\mathfrak{X},\mathfrak{k}}=\mathcal{N}_\mathfrak{k}$, while the ``exotic nilpotent variety'' $\mathcal{N}_{\mathfrak{X},\mathfrak{s}}$ is not irreducible. Specifically, if $n=1$, then $\mathcal{N}_{\mathfrak{X},\mathfrak{s}}=\mathcal{N}_\mathfrak{s}$ (it has two irreducible components; see Remark \ref{R2.1}).\ If $n\geq 2$, then $\mathcal{N}_{\mathfrak{X},\mathfrak{s}}$
has exactly three irreducible components described as follows.
\begin{thmenumeralph}
\item Assume that $n$ is even. Then the components of $\mathcal{N}_{\mathfrak{X},\mathfrak{s}}$ are the closures of the $K$-orbits parametrized by the signed Young diagrams
\[\Lambda_+:=\tableau{+-+-\cdots-,+-+-\cdots-}\,,\quad \Lambda_0:=\tableau{+-+-\cdots-,-+-+\cdots+}\,,\quad \Lambda_-:=\tableau{-+-+\cdots+,-+-+\cdots+}\]
where the rows have length $n$. In this case the variety $\mathcal{N}_{\mathfrak{X},\mathfrak{s}}$ is equidimensional of dimension $2(n^2-n)$.
\item Assume that $n$ is odd. Then the components of $\mathcal{N}_{\mathfrak{X},\mathfrak{s}}$ are the closures of the $K$-orbits parametrized by
\[\Lambda_+:=\tableau{+-+-\cdots-+-,+-+-\cdots-}\,,\quad \Lambda_0:=\tableau{+-+-\cdots+,-+-+\cdots-}\,,\quad \Lambda_-:=\tableau{-+-+\cdots+-+,-+-+\cdots+}\]
where the rows have lengths $n-1,n,n+1$. The variety
$\mathcal{N}_{\mathfrak{X},\mathfrak{s}}$ is not equidimensional 
and, it has two components of dimension $2(n^2-n)+1$ and one component of dimension $2(n^2-n)$.
\end{thmenumeralph}
\end{theorem}

In the theorem, we use the parametrization of the $K$-orbits of $\mathcal{N}_\mathfrak{s}$ by signed Young diagrams; see \cite{Collingwood.McGovern.1993}, \cite{Ohta}, or Section \ref{section-2.2} below. Theorem \ref{T1} is proved in Section~\ref{section-11}.

\subsection{}\label{section-1.5} 
To determine the maps $\mathfrak{X}/K\to\mathcal{N}_\mathfrak{k}/K$
and $\mathfrak{X}/K\to\mathcal{N}_\mathfrak{s}/K$ in (\ref{3}) 
we consider another version of the Steinberg maps, 
which might be a more straightforward 
generalization of the original Steinberg theory 
and is of independent interest.  Let us explain it.

Note that $B_K$
is of the form $B_K=B_1\times B_2$, for Borel subgroups
$B_1,B_2$ of $ \GLnC $.  
We consider an action of $B_1 \times B_2$ on the space of
$n\times n$ matrices $\MatnC$, one factor acting on the left and the other acting on the right.
Namely the explicit action is given by 
\[(b_1,b_2)\cdot x:=b_1xb_2^{-1} \qquad \forall (b_1,b_2)\in B_K,\ \forall x\in\MatnC.\]
Since the action of $ \GLnC \times \GLnC $ on $ \MatnC $ is clearly spherical, 
there are only finitely many $ B_1 \times B_2 $ orbits on $ \Mat_n $.  
By means of Gaussian eliminations, 
it is easy to obtain a complete system of representatives of the double coset 
space $ B_1 \backslash \MatnC / B_2 $, which turns out to be
the set of partial permutations $\mathfrak{T}_n$ 
(Proposition \ref{P3}). 
The conormal variety for this action is given by 
\[\mathcal{Y}_{\mathrm{M}_n}=\{(x,y)\in\MatnC\times\MatnC:xy\in\mathfrak{n}_1,\ yx\in\mathfrak{n}_2\},\]
where $\mathfrak{n}_i$ stands for the nilradical of
$\mathrm{Lie}(B_i)$; see Proposition \ref{P2}. 
By the general theory (Proposition \ref{P1}), this conormal variety is equidimensional and
its irreducible components are parametrized by partial permutations.
The image of each component of $\mathcal{Y}_{\mathrm{M}_n}$ by the 
moment map
\[\mu_{\Mat_n} : (x,y)\longmapsto (xy,-yx) \in \MatnC^2 \]
determines a pair of nilpotent $ \GLnC $-orbits, 
or in other words, a pair of partitions $ (\lambda, \mu) \in \partitionsof{n}^2 $.
This yields a map from partial permutations to pairs of partitions:
\[
\Phi : \ppermutationsof{n} \longrightarrow \partitionsof{n}^2 , \qquad
\tau \longmapsto (\lambda, \mu) .
\]
We call the map $ \Phi $ a \emph{generalized Steinberg map}.

The result summarized in the next theorem is also one of our main results, 
which generalizes the Robinson-Schensted correspondence to the case of 
partial permutations.  
The theorem is actually independent of the theory of the double flag varieties, 
and we are expecting a further generalization to general reductive groups.  
Theorem~\ref{introthm:T2.thm.3.9} plays a key role in our calculations of the maps $\Xik$ and $\Xis$ of Theorem \ref{TII}.  

Let us prepare one more notation to state the theorem.  
For each $ r \; (0 \leq r \leq n) $, we consider triples of the form $ (T_1, T_2; \nu) $
where $ T_1, T_2 $ are standard tableaux of shapes denoted $ \lambda, \mu \in \partitionsof{n} $ respectively
and $ \nu \in \partitionsof{r} $ is a partition such that
$ \lambda \setminus \nu $ and $ \mu \setminus \nu $ are column strips (or vertical strips).  
In particular, $ \nu $ is in the intersection of $ \lambda $ and $ \mu $.  
Let us denote the set of such triples by $ \Triplets{r} $.

\begin{theorem}[Theorems \ref{T2} and \ref{thm:3.9}]\label{introthm:T2.thm.3.9}
The map $\Phi$ can be described through an explicit combinatorial algorithm.
This algorithm establishes a bijection between 
the set of partial permutations $ \ppermutationsof{n} $ 
and the set of triples $ \bigcup_{r = 0}^n \Triplets{r} $ defined above.  
If the partial permutation is a full permutation, 
the bijection reduces to the classical 
Robinson-Schensted correspondence.  
Namely $ \tau \in \mathfrak{S}_n $ corresponds to $ (T_1, T_2; \lambda) $ 
where $ \shape{T_1} = \shape{T_2} = \lambda $ (so that $ r = n $).
\end{theorem}

\subsection{}

Let us comment on some relevant references.  
In \cite{Yamamoto1,Yamamoto2}, Atsuko Yamamoto gave a combinatorial algorithm for calculating 
the moment map image of the conormal bundles of 
$K$-orbits on $ G/B $ for classical symmetric pairs $(G,K)$ of type A.  
This particular case can be considered as a special case described in 
Section \ref{section-1.2} if
we take $ Q = K $ and $ P = B $.   
Her description of the problem is very close to ours in spirit.  
  
In \cite{Travkin}, based on the theory of mirabolic character sheaves in \cite{Finkelberg.Ginzburg.Travkin.2009}, Travkin gave a generalization of Robinson-Schensted-Knuth algorithm for 
a triple flag variety of the form $\Flag(\C^n) \times \Flag(\C^n) \times \mathbb{P}(\C^{n})$, which can be also recognized as a special case of the setting of Section \ref{section-1.2} 
if we take $ G = \GL_n \times \GL_n$, $K = \Delta \GL_n$ 
(the diagonal embedded $\GL_n$ into $ G $), 
$ P = B_1 \times B_1 $ (a Borel subgroup of $ G $), and $ Q = P_{\mathrm{mir}} $ (the mirabolic maximal parabolic subgroup of $ K \simeq \GL_n $).  
Travkin actually considered the diagonal $ \GL_n $ action on $\Flag(\C^n) \times \Flag(\C^n) \times \C^n$ and, 
in his terminology, the orbits 
are parametrized by a certain set of pairs $(w,\beta)$, where $ w \in \permutationsof{n} $ and $ \beta \subset [n] $ is related to a decreasing sequence in $ w $.  
His parameter set can actually be identified with the set of partial permutations $\ppermutationsof{n}$,
and the mirabolic Robinson-Schensted-Knuth correspondence establishes a bijection between $ \ppermutationsof{n} $ and the set of triples $ \bigcup_{r = 0}^n \Triplets{r} $. 
However, this bijection appears to be totally different from ours.  
It is interesting to compare these two different bijections and to get a geometric interpretation.  
See Remark \ref{R:Travkin} for further discussion.

Rosso generalized the correspondence to the case of partial flag varieties and gave 
a geometric interpretation of the classical Robinson-Schensted-Knuth correspondence \cite{Rosso.2012}.

In \cite{Henderson-Trapa}, Henderson and Trapa also gave a generalization of Robinson-Schensted correspondence 
for the symmetric pair $ (G, K) = (\GL_{2n}, \Sp_{2n}) $, with $ Q\subset K $ mirabolic and $ P = B \subset G$ a Borel subgroup.  
Their approach is very close to ours.  See also Remark~\ref{remark:dim.identity.HT}.

Finally, in our previous paper 
\cite{Fresse-Nishiyama}, we considered the case where 
$ (G, K) = (\GL_n, \GL_p \times \GL_q)$ ($p + q = n$); 
$ P $ is the stabilizer of a $k$-dimensional subspace of $ \C^n $ (a maximal parabolic subgroup of $ G $) and 
$ Q = Q_1 \times \GL_q $, where $ Q_1 $ is a mirabolic subgroup of $ \GL_p $.  
In some sense \cite{Fresse-Nishiyama} can be taken as a first trial, and the present paper largely deepens it.

\subsection{Organization of the paper}
The paper is divided into three parts. In the first part, we review
the background on the Robinson-Schensted algorithm
(Section \ref{section-RS}), moment maps and conormal varieties
(Section \ref{newsection-3}), the Steinberg map (Section
\ref{section-Steinberg}). In Section \ref{section-2.2}, we review the
basic facts and prepare the notation related to the symmetric pair
$(G,K)$ of type AIII studied in this paper. 

In the second part of the paper, we consider the problem outlined in Section \ref{section-1.5}.
In Section \ref{section-orbits}, we consider an action 
of the product of Borel subgroups 
on the space of $n\times n$ matrices given by left and right multiplications.  
We show that the partial permutations serve as a complete set of representatives for the orbits.
In Section \ref{section-Phi}, we define a generalized Steinberg map
and provide an explicit combinatorial algorithm on partial permutations to calculate this map, 
which generalizes the Robinson-Schensted correspondence on permutations.
We establish an identity between the number of orbits (i.e., partial permutations) 
and the dimension of a certain induced representation of $ \mathfrak{S}_n $ 
(see Corollary \ref{C3.2}).   This identity suggests the existence of a geometric interpretation 
of the construction of such representations (see Conjecture~\ref{conjecture:gen.Springer.representation}).

The final part of the paper is devoted to
the main results outlined in Sections \ref{section-1.3}--\ref{section-1.4}.
Precise statements and proofs of these results
are given in Sections \ref{section-8} (parametrization of the orbit set
$\mathfrak{X}/K$), \ref{section-9} and \ref{section-10} (calculation of the maps
$\Xik$ and $\Xis$). Section
\ref{section-11} contains the proof of Theorem \ref{T1}. 

The setting is recalled at the
beginning of each section. An index of notation can be found at the
end of the paper.

\subsection{Acknowledgements}
We thank Anthony Henderson and George Lusztig for useful remarks which improve the manuscript.
L.F. thanks Aoyama Gakuin University for warm hospitality during his visit in February 2017. This work was initiated in this period.
K.N. thanks Institut \'Elie Cartan de Lorraine in Universit\'e de Lorraine for warm hospitality during his visit in June and July, 2018.  Key ingredients of this work were obtained in this period.

\part{Preliminaries on miscellanea and a review of Steinberg theory}

In this part we review known facts, which we need in the rest of the paper.  
The content of each section varies independently, 
and the notation or settings are often different from section to section.

\section{Robinson-Schensted correspondence}

\label{section-RS}
In this first preliminary section we summarize the background on Young diagrams and tableaux; we refer to \cite{Fulton} for more details.

\subsection{Partitions, Young diagrams, tableaux}

\label{section-2.2.1}

We use the following conventions for Young diagrams and tableaux.

Let $\partitionsof{n}$ denote the set of partitions of $n$, i.e.,
nonincreasing sequences of positive integers
$\lambda=(\lambda_1,\ldots,\lambda_k)$ such that
$\lambda_1+\ldots+\lambda_k=n$.

A partition can be represented by a {\em Young diagram}, also
denoted by $\lambda$, which is an array of empty boxes,
left-justified, whose rows have lengths
$\lambda_1,\ldots,\lambda_k$. We do not distinguish between the
notions of partition and Young diagram.

We call {\em Young tableau of shape $\lambda$} a numbering of the
boxes of $\lambda$ by pairwise distinct entries, which are in
increasing order along rows (from left to right) and columns (from
top to bottom). We write $\lambda=\mathrm{shape}(T)$ whenever $T$ is
a Young tableau of shape $\lambda$, e.g., 
\[
T
=\tableau{127,36,59,8}\quad\Rightarrow\quad \mathrm{shape}(T)
=\diagram{3,2,2,1}=(3,2,2,1)\in \partitionsof{8} . 
\]
If the set of entries of a Young tableau $ T $ is $ [n] := \{ 1, 2, \dots, n \} $, we call $ T $ 
a \emph{standard tableaux}.  
The set of standard tableaux of shape $ \lambda \in \partitionsof{n} $ is denoted as $ \STab(\lambda) $.

\subsection{Row-insertion and column-insertion algorithms}

\label{section-2.2.2}

Given a Young tableau $T$ and a number $a$ distinct from
all the entries of $T$, we denote by $(T\leftarrow a)$ (resp.
$(a\to T)$) the Young tableau obtained from $T$ by inserting the
entry $a$ according to the following algorithm:
\begin{itemize}
\item If $a$ is larger than any entry in the first row (resp. column) of $T$, then insert $a$ in a new box at the end of the first row (resp.
column).
\item Otherwise, let $a_1$ be the smallest entry in the first row (resp. column) of $T$ which is $>a$;
substitute $a_1$ by $a$ and insert $a_1$ in the subtableau of $T$ starting with the second row (resp. column), according to the same rule.
\end{itemize}
We call this procedure {\em row insertion} (resp. {\em column
insertion}).

For instance, for $T$ as above, we get
\[(T\leftarrow 4)=\tableau{124,367,59,8}\quad\mbox{and}\quad (4\to T)=\tableau{1267,35,49,8}.\]

Given a list $(a_1,\ldots,a_\ell)$ of pairwise distinct numbers, let
\begin{eqnarray*}
 & \rinsert(a_1,\ldots,a_\ell):=((\cdots((\emptyset\leftarrow
a_1)\leftarrow a_2)\cdots )\leftarrow a_\ell) \\
\mbox{and} & \cinsert(a_1,\ldots,a_\ell):=
(a_\ell\to(\cdots\to(a_2\to(a_1\to \emptyset))\cdots)).
\end{eqnarray*}
We get the following equality (see, e.g., \cite[Theorem
4.1.1]{van-Leeuwen}):
\[
\rinsert(a_1,\ldots,a_\ell)=\cinsert(a_\ell,\ldots,a_1).
\]

\subsection{The Robinson-Schensted correspondence}\label{section-2.2.3}

For the later use, we formulate the correspondence in a slightly general way.

Given a pair of Young tableaux $(T,S)$ of the same shape, a number
$a$ which is distinct from all the entries of $T$, and a number $b$
which is bigger than all the entries of $S$, we denote by
\begin{equation}\label{eq:insertion.in.pair}
(T,S)\leftarrow(a,b)
\end{equation}
the pair of Young tableaux $(T',S')$ such
that
\begin{itemize}
\item
$T'=(T\leftarrow a)$,
\item
$S'$ is obtained from $S$ by putting
the entry $b$ in a new box at the same position as the unique box of
$\mathrm{shape}(T')\setminus\mathrm{shape}(T)$.
\end{itemize}

Let $a_1<\ldots<a_\ell$ and $b_1<\ldots<b_\ell$ be two sequences of
numbers. Given a bijection
$w:\{b_1,\ldots,b_\ell\}\to\{a_1,\ldots,a_\ell\}$, we define a 
pair of Young tableaux using the insertion \eqref{eq:insertion.in.pair} successively 
\[(\RSl(w),\RSr(w)):= (\emptyset,\emptyset) \leftarrow (w(b_1),b_1) \leftarrow (w(b_2),b_2) \leftarrow \cdots \leftarrow (w(b_\ell),b_\ell) . \]
Thus $\RSl(w)=\rinsert(w(b_1),\ldots,w(b_\ell))$ and 
$ \RSr(w) $ encodes the development of the shape. By definition, the tableaux
$\RSl(w)$ and $\RSr(w)$ have the same shape with entries 
$\{a_1,\ldots,a_\ell\}$ and $\{b_1,\ldots,b_\ell\}$ respectively.
Moreover, we have (cf. \cite[\S4.1]{Fulton}):
\[(\RSl(w^{-1}),\RSr(w^{-1}))=(\RSr(w),\RSl(w)).\]

Since an element $ w \in \mathfrak{S}_n $ is a bijection from $ [n] $ to itself, 
we can apply the procedure above.  
Then the map 
\begin{equation}
\mathfrak{S}_n \ni w \mapsto (\RSl(w),\RSr(w)) \in \coprod_{\lambda \in \partitionsof{n}} 
\STab(\lambda) \times \STab(\lambda) 
\end{equation}
establishes a
bijection between the symmetric group and the set of pairs of standard 
tableaux of the same shape. This bijection is
referred to as {\em the Robinson-Schensted correspondence}.

\subsection{Jeu de taquin}

\label{section-2.2.4}

For two partitions $\nu=(\nu_1,\ldots,\nu_\ell)\in\partitionsof{m}$ and $\lambda=(\lambda_1,\ldots,\lambda_k)\in\partitionsof{n}$, we write $\nu\subset\lambda$ if $\ell\leq k$ and $\nu_i\leq \lambda_i$ for all $i\in\{1,\ldots,\ell\}$. Then the set of boxes $\lambda\setminus\nu$ is called a {\em skew diagram}. A {\em skew tableau} of shape $\lambda\setminus\nu$ is a numbering of the boxes of $\lambda\setminus\nu$ by pairwise distinct integers which are in increasing order along rows and columns.
Such a skew tableau $T$ can be transformed into a Young tableau by the procedure of {\em jeu de taquin} (or \emph{slidings}):
\begin{itemize}
\item Choose any inside corner of $T$, i.e., a box $c$ of $\nu$ which is adjacent to a box of $T$ and such that $\{c\}\cup(\lambda\setminus\nu)$ is also a skew diagram.
\item {\it Sliding step:}
one or both of the boxes on the right or below $c$ is contained in $T$. Then choose the smaller entry 
(if there are two) and slide it
into the box $c$.
\item If the box $c_1$ that has just been emptied has also a neighbor on the right or below, then apply the sliding step to $c_1$. Repeat this procedure until the emptied box has no neighbor on the right nor below.
\item Then one obtains a skew tableau whose number of inside boxes is smaller.
Repeat the whole procedure, until one obtains a skew tableau without inside boxes, i.e., a Young tableau.
\end{itemize}
The Young tableau $\mathrm{rect}(T)$ so-obtained is called the {\em rectification of $T$}. 
Note that the result is independent of the choice of inside corners. 

For instance, if we denote by a dot the choice of inside corner, we proceed 
\[
T=\tableaul{\none & \none[\bullet] & 3 \\ \none
& 4 & 7 \\ 1 & 6 & 9 \\ 8} 
\xrightarrow{\text{\;sliding\;}}
\tableaul{\none & 3 & 7
\\ \none[\bullet] & 4 & 9 \\ 1 & 6 \\ 8}
\xrightarrow{\text{\;sliding\;}}
\tableaul{\none[\bullet] & 3 & 7 \\ 1 & 4 & 9 \\ 6 \\ 8}
\xrightarrow{\text{\;sliding\;}}
\tableau{137,49,6,8}=\mathrm{rect}(T).
\]

If $T,S$ are two Young tableaux whose entries are disjoint,
we define $T*S$ as the rectification of the skew tableau obtained by
putting $S$ at the top right to $T$, e.g.,
\[
\tableau{57,9}*\tableau{134,28}=\mathrm{rect}\left(\tableau{\none\none134,\none\none28,57,9}\right)=\tableau{134,278,5,9}.
\]
Note that
\[T*\tableau{a}=(T\leftarrow a)\quad\mbox{and}\quad \tableau{a}*T=(a\to T),\]
and more generally
\begin{equation*}
T*\tableaul{\ell_1 \\ {}^{\vdots} \\ \ell_s}
=T \leftarrow \ell_s \leftarrow \ell_{s - 1} \leftarrow \cdots \leftarrow \ell_1 
\quad\mbox{ and }\quad
\tableaul{m_1 \\ {}^{\vdots} \\ m_s}*T = m_s \to m_{s - 1} \to \cdots \to m_1 \to T
\end{equation*}
%
%
whenever $a$, $\ell_1<\ldots<\ell_s$, $m_1<\ldots<m_s$
are not entries of $T$.

\section{Moment maps and conormal varieties}

\label{newsection-3}

\subsection{Conormal variety}
\label{section-2.1} Let $H$ be a connected algebraic group acting on
a smooth algebraic variety $X$. 
The cotangent bundle $T^*X$ has a
structure of symplectic variety and the $H$-action induces a
Hamiltonian action of $H$ on $T^*X$.   
The corresponding moment map is denoted by 
$\mu_X:T^*X\to\mathfrak{h}^*$, where $\mathfrak{h}^*$ denotes the algebraic dual of the Lie algebra of $H$
(\cite[Proposition 1.4.8]{Chriss-Ginzburg}). 

The null fiber of the moment map
\begin{equation*}
\mathcal{Y}_X:=\mu_X^{-1}(0)
\end{equation*}
is a union of Lagrangian subvarieties of $T^*X$
that we call {\em conormal variety}.  
Let us see some of the beautiful nature of conormal varieties.  
In our situation above, 
the moment map $\mu_X$ can be described explicitly as follows (see \cite[\S1.4]{Chriss-Ginzburg}):
\[ \mu_X : T^*X = \{(x,\xi) : \xi\in (T_xX)^*\} \longrightarrow \mathfrak{h}^*, \qquad (x,\xi) \longmapsto \xi\circ d\rho_x\]
where $\rho_x:H\to X$, $h\mapsto h\cdot x$ is the orbit map and $d\rho_x:\mathfrak{h}\to T_xX$ is its differential.

\begin{proposition}\label{P1}
We consider the cotangent bundle $\pi_X:T^*X\to X$ and its restriction to the conormal variety $\mathcal{Y}_X$.
\begin{thmenumerate}
\item\label{P1:item:1}
For each $H$-orbit $\mathbb{O}\subset X$, the restriction $\pi_X^{-1}(\mathbb{O})\cap \mathcal{Y}_X\to\mathbb{O}$ coincides with the conormal bundle $T^*_\mathbb{O}X\to\mathbb{O}$. In particular
$\pi_X^{-1}(\mathbb{O})\cap \mathcal{Y}_X=T^*_\mathbb{O}X$ is a smooth, irreducible, Lagrangian subvariety of $T^*X$, and hence it has dimension $\dim X$.
\item
The conormal variety is a union of the conormal bundles: 
$ \conormalvar_X = \bigsqcup_{\bborbit \in X/H} T^*_{\bborbit} X $.
\item
Consequently, if $X$ has a finite number of $H$-orbits, then $\mathcal{Y}_X$ is equidimensional of dimension $\dim X$, and each irreducible component of $\mathcal{Y}_X$ is of the form $\overline{T_\mathbb{O}^*X}$ for a unique $H$-orbit $\mathbb{O}\subset X$.
\end{thmenumerate}
\end{proposition}

\begin{proof} It suffices to prove \eqref{P1:item:1}.  The rest of the statements are clear.
For $(x,\xi)\in T^*X$, we have
\begin{eqnarray*}
(x,\xi)\in \pi^{-1}_X(\mathbb{O})\cap\mathcal{Y}_X & \iff & x\in\mathbb{O}\quad \mbox{and}\quad \mu_X(x,\xi)=0 \\
 & \iff & x\in\mathbb{O}\quad\mbox{and}\quad \forall \eta\in\mathfrak{h},\ \xi(d\rho_x(\eta))=0 \\
 & \iff & x\in\mathbb{O}\quad\mbox{and}\quad \xi|_{T_x\mathbb{O}}=0 \\ & \iff & (x,\xi)\in T_\mathbb{O}^*X
\end{eqnarray*}
where we use that the map $d\rho_x:\mathfrak{h}\to T_xX$ is surjective onto $T_x(H\cdot x)$ (see, e.g., \cite[Theorem 4.3.7]{Springer}), hence $\pi_X^{-1}(\mathbb{O})\cap\mathcal{Y}_X=T^*_\mathbb{O}X$. The other assertions made in \eqref{P1:item:1} follow from the properties of the conormal bundle $T^*_\mathbb{O}X\to \mathbb{O}$.
\end{proof}

\subsection{Moment map for double flag variety}
\label{section-2.1.1}

Let $G$ be a connected reductive algebraic group with Lie algebra $\mathfrak{g}$. 
Hereafter we identify $\mathfrak{g}$ with its algebraic dual $\mathfrak{g}^*$ 
via a fixed nondegenerate $G$-invariant bilinear form $\langle\cdot,\cdot\rangle$ on $\mathfrak{g}$.

For a parabolic subgroup $P\subset G$, the partial flag variety
$G/P$ can also be viewed as the set of parabolic subalgebras
$\mathfrak{p}_1\subset\mathfrak{g}$ which are conjugate to
$\mathfrak{p}:=\mathrm{Lie}(P)$. The tangent space
$T_{\mathfrak{p}_1}(G/P)$ coincides with the quotient space
$\mathfrak{g}/\mathfrak{p}_1$. Its dual
$(T_{\mathfrak{p}_1}(G/P))^*=(\mathfrak{g}/\mathfrak{p}_1)^*$
identifies with the space
$\{\xi\in\mathfrak{g}^*:\xi|_{\mathfrak{p}_1}=0\}$, which itself
corresponds (through the invariant form $\langle\cdot,\cdot\rangle$) to the
nilpotent radical $\mathfrak{nil}(\mathfrak{p}_1)$. In this way the
cotangent bundle is given as 
\[
T^*(G/P)=\{(\mathfrak{p}_1,\xi)\in
(G/P)\times\mathfrak{g}^*:\xi|_{\mathfrak{p}_1}=0\} \cong
\{(\mathfrak{p}_1,x)\in
(G/P)\times\mathfrak{g}:x\in\mathfrak{nil}(\mathfrak{p}_1)\}
\]
and the action of $G$ on $G/P$ gives rise to the moment map
\[\mu_{G/P}:T^*(G/P) \longrightarrow \mathfrak{g}^*, \quad (\mathfrak{p}_1,\xi) \longmapsto \xi
\quad\mbox{ or equivalently, } \quad 
(\mathfrak{p}_1,x) \longmapsto x.\]

In this paper, we consider a double flag variety of the form
$\mathfrak{X}=G/P\times K/Q$ for a (connected) symmetric subgroup
$K=G^\theta\subset G$, defined by an involution
$\theta\in\mathrm{Aut}(G)$, and a parabolic subgroup $Q\subset K$.
In this situation, we assume that the bilinear form
$\langle\cdot,\cdot\rangle$ is also $\theta$-invariant
(the bilinear form can always be chosen in this way).

We identify $ \dblFV $ with the collection of 
pairs of parabolic subalgebras 
$(\mathfrak{p}_1,\mathfrak{q}_1)$ where
$\mathfrak{p}_1\subset\mathfrak{g}$ is $G$-conjugate to $\mathfrak{p}$ and
$\mathfrak{q}_1\subset\mathfrak{k}:=\mathrm{Lie}(K)$ is 
$K$-conjugate to $\mathfrak{q}:=\mathrm{Lie}(Q)$.  
Then the cotangent
bundle is described as 
\begin{eqnarray*}
T^*\mathfrak{X} & = &
\{(\mathfrak{p}_1,\mathfrak{q}_1,\xi,\eta)\in(G/P)\times(K/Q)\times\mathfrak{g}^*\times\mathfrak{k}^*:\xi|_{\mathfrak{p}_1}=0,\
\eta|_{\mathfrak{q}_1}=0\} \\
 & \cong &
\{(\mathfrak{p}_1,\mathfrak{q}_1,x,y)\in(G/P)\times(K/Q)\times\mathfrak{g}\times\mathfrak{k}:x\in\mathfrak{nil}(\mathfrak{p}_1),\
y\in\mathfrak{nil}(\mathfrak{q}_1)\}
\end{eqnarray*} and the diagonal action
of $K$ on $\mathfrak{X}$ gives rise to a moment map
\[
\mu_\mathfrak{X}:T^*\mathfrak{X}\to\mathfrak{k}^*,\
(\mathfrak{p}_1,\mathfrak{q}_1,\xi,\eta)\mapsto
\xi|_{\mathfrak{k}}+\eta\]
or, equivalently,
\[
(\mathfrak{p}_1,\mathfrak{q}_1,x,y)\mapsto x^\theta+y
\]
with $x^\theta:=(x+\theta(x))/2$. The conormal variety
$\mathcal{Y}_\mathfrak{X}=\mu_\mathfrak{X}^{-1}(0)$ can be described
as
\[
\mathcal{Y}_\mathfrak{X}=
\{(\mathfrak{p}_1,\mathfrak{q}_1,\xi,\eta)\in(G/P)\times(K/Q)\times\mathfrak{g}^*\times\mathfrak{k}^*:\xi|_{\mathfrak{p}_1}=0,\
\eta|_{\mathfrak{q}_1}=0,\ \eta=-\xi|_{\mathfrak{k}}\}
\]
hence we get an isomorphism 
\begin{eqnarray*}
\mathcal{Y}_\mathfrak{X} & \cong &
\{(\mathfrak{p}_1,\mathfrak{q}_1,\xi)\in(G/P)\times(K/Q)\times\mathfrak{g}^*:\xi|_{\mathfrak{p}_1}=0,\
\xi|_{\mathfrak{q}_1}=0\} \\
 & \cong &
\{(\mathfrak{p}_1,\mathfrak{q}_1,x)\in(G/P)\times(K/Q)\times\mathfrak{g}:x\in\mathfrak{nil}(\mathfrak{p}_1),\
x^\theta\in\mathfrak{nil}(\mathfrak{q}_1)\}.
\end{eqnarray*}
In the following we often identify these isomorphic varieties without mention.

\subsection{Moment maps for rational representations}
\label{section-2.1.2}
Let $ V $ be a finite dimensional $ H $-module.  
We denote by $(\eta,v)\mapsto \eta v$ the action of $\mathfrak{h}$ on $V$ obtained by differentiation. The cotangent space $T^*V=V\times V^*$ is endowed with a Hamiltonian action of $H$ given by $h\cdot(v,\xi)=(hv,\xi\circ h^{-1})$. The corresponding moment map is given by
\[
\mu_V:T^*V \longrightarrow \mathfrak{h}^*, \qquad 
(v,\xi) \longmapsto \big\{\eta\mapsto \xi(\eta v)\big\}.
\]
So the conormal variety $\mathcal{Y}_V$ is expressed as
\begin{equation}\label{4}
\mathcal{Y}_V=\mu_V^{-1}(0)=\{(v,\xi)\in V\times V^*:\forall \eta\in\mathfrak{h},\ \xi(\eta v)=0\}.
\end{equation}

\section{The Steinberg map}

\label{section-Steinberg}

Let us explain in more detail the construction of the Steinberg map, outlined in Section \ref{section-1.1}.
We follow the approach of \cite{Hinich-Joseph}, which is slightly different from Steinberg's original construction \cite{Steinberg-1976}.

The flag variety $\mathcal{B}=G/B$ can be identified with the set of all Borel subalgebras $\mathfrak{b}'\subset\mathfrak{g}$. As explained in Section \ref{section-2.1.1},
the cotangent bundle $T^*(\mathcal{B}\times\mathcal{B})$ and the moment map $\mu_{\mathcal{B}\times\mathcal{B}}$ can be described as
\[
\mu_{\mathcal{B}\times\mathcal{B}} : 
T^*(\mathcal{B}\times\mathcal{B}) = \{(\mathfrak{b}'_1,\mathfrak{b}'_2,x_1,x_2):x_i\in\mathfrak{nil}(\mathfrak{b}'_i)\} \longrightarrow \mathfrak{g}, \quad 
(\mathfrak{b}'_1,\mathfrak{b}'_2,x_1,x_2) \longmapsto x_1+x_2
\]
and the conormal variety is given by
\[
\mathcal{Y}=\{(\mathfrak{b}'_1,\mathfrak{b}'_2,x)\in \mathcal{B}\times\mathcal{B}\times\mathfrak{g}:x\in\mathfrak{nil}(\mathfrak{b}'_1)\cap \mathfrak{nil}(\mathfrak{b}'_2)\}.
\]
In this case, $\mathcal{Y}$ is often referred to as the {\em Steinberg variety}.

On the other hand, every $G$-orbit of $\mathcal{B}\times\mathcal{B}$ takes the form
$\mathcal{Z}_w:=G(\lie{b},\wconjn{w}{b})$ for a unique Weyl group element $w\in W$. Here $B$ is a (fixed) Borel subgroup containing the maximal torus $T$, and
we denote $\lie{b}=\mathrm{Lie}(B)$ and $\lie{n}=\mathfrak{nil}(\lie{b})$.
Hereafter, we use the notation $\wconjn{w}{b}=\mathrm{Ad}(w)(\lie{b})$
and $\wconjn{w}{n}=\mathrm{Ad}(w)(\lie{n})$.
The orbit $\mathcal{Z}_w$ gives rise to the conormal bundle
\begin{equation}
\label{conormal-bundle}
T^*_{\mathcal{Z}_w}(\mathcal{B}\times\mathcal{B})=\{(\mathfrak{b}'_1,\mathfrak{b}'_2,x)\in\mathcal{Y}:(\mathfrak{b}'_1,\mathfrak{b}'_2)\in\mathcal{Z}_w\}=G\cdot\{(\lie{b},\wconjn{w}{b},x):x\in\lie{n}\cap \wconjn{w}{n})\}\end{equation}
whose closure is an irreducible component of $\mathcal{Y}$ according to Proposition \ref{P1}.
Every component of $\mathcal{Y}$ is of this form.

The projection map $\pi:(\mathfrak{b}'_1,\mathfrak{b}'_2,x)\mapsto x$ is $G$-equivariant and closed.
It therefore maps $\overline{T^*_{\mathcal{Z}_w}(\mathcal{B}\times\mathcal{B})}$ onto the closure of a nilpotent orbit $\mathcal{O}_w\in\mathcal{N}/G$.
Note that $\mathcal{O}_w$ is also characterized as the unique nilpotent orbit which intersects the space
$ \lie{n}\cap \wconjn{w}{n} $ along a dense open subset.
The so-obtained map
\[
\mathrm{St}:W\cong
(\mathcal{B}\times\mathcal{B})/G\longrightarrow\mathcal{N}/G, \qquad w\longmapsto\mathcal{O}_w
\]
is the {\em Steinberg map} introduced in
Section \ref{section-1.1}.

Note that the bijection
$W\stackrel{\sim}{\to}(\mathcal{B}\times\mathcal{B})/G$, $w\mapsto
G(\lie{b},\wconjn{w}{b})$ is not canonical, as it depends on the choice of a Borel
subgroup $B\subset G$ which contains the maximal torus $T$
(though it becomes canonical if $W$ is replaced by the abstract Weyl group; see \cite[Proposition 3.1.29]{Chriss-Ginzburg}).

In the case of $G=\GLnC$, we always consider the
Steinberg map corresponding to the Borel subgroup $B=B_n^+$ of upper
triangular matrices. 
In this case, 
the Weyl group $W=\mathfrak{S}_n$ is the symmetric group, whereas the
nilpotent orbits $\mathcal{O}_\lambda\in\mathcal{N}/G$ are encoded
by the partitions $ \lambda \in \partitionsof{n} $.

\begin{notation}
For a partition
$\lambda=(\lambda_1,\ldots,\lambda_k)\in\partitionsof{n}$, we denote
by $\mathcal{O}_\lambda\subset\MatnC$ the subset
consisting of the nilpotent matrices which have $k$ Jordan blocks of
sizes $\lambda_1,\ldots,\lambda_k$. The subsets
$\mathcal{O}_\lambda$ (for $\lambda\in\partitionsof{n}$) are exactly
the nilpotent orbits in $\mathfrak{gl}_n$.
\end{notation}

In this way, we get a combinatorial incarnation of the Steinberg map, which we still denote by $\mathrm{St}:\mathfrak{S}_n\to\partitionsof{n}$ by abuse of notation.
This map can be described in terms of the Robinson-Schensted algorithm (see
Section \ref{section-2.2.3}).  

\begin{theorem}[{\cite{Steinberg2}}]
\label{T-RS}
Assume that $G=\GLnC$ and $B=B_n^+$ is the subgroup of upper triangular matrices.
Then, for any permutation $w\in\mathfrak{S}_n$, we have
\[\mathrm{St}(w)=\mathrm{shape}(\rinsert(w(1),\ldots,w(n)))=\mathrm{shape}(\RS_i(w))\quad(i\in\{1,2\}).\]
For $ \lambda \in \partitionsof{n} $, the fiber $ \mathrm{St}^{-1}(\lambda) $ is the 
set of permutations $ w $ which correspond to a pair of standard tableaux in $ \STab(\lambda) \times \STab(\lambda) $ via the Robinson-Schensted correspondence.
\end{theorem}

\section{The symmetric pair of type AIII}

\label{section-2.2}

\subsection{Symmetric pair of type AIII (tube type)}
\label{section-2.4.1}
Here we denote $G=\GL_{2n}$ and
$\lie{g}=\lie{gl}_{2n}$. We consider an
involution $\theta\in\Aut(G)$ given by
\[\theta(g)=\iota g\iota^{-1}\qquad\mbox{where}\quad\iota=\left(\begin{array}{cc}
1_n & 0 \\ 0 & -1_n \end{array}\right).\] Its differential yields an involution
$\theta:\lie{g}\to\lie{g}$, which can be defined exactly in the same
way in matrix expression.

The symmetric subgroup $K:=G^\theta$ can be described as
\begin{eqnarray*}
K & = & \left\{\left(\begin{array}{cc} a & 0 \\ 0 & b
\end{array}\right):a,b\in\GLnC\right\}=\{g\in\GLnnC:g(V^+)=V^+,\ g(V^-)=V^-\} \\
 & \cong &
 \GLnC\times\GLnC\end{eqnarray*}
where we denote $V^+:=\C^n\times\{0\}^n$ and
$V^-:=\{0\}^n\times\C^n$ so that we get a direct sum decomposition
$\C^{2n}=V^+\oplus V^-$.

The Cartan decomposition 
$\lie{g}=\mathfrak{k}\oplus\mathfrak{s}$ is given by 
\[\mathfrak{k}:=\mathrm{Lie}(K)=\left\{\left(\begin{array}{cc} \alpha & 0 \\ 0 &
\beta
\end{array}\right):\alpha,\beta\in\MatnC\right\},\quad
\mathfrak{s}:=\lie{g}^{-\theta}=\left\{\left(\begin{array}{cc}
0 & \gamma
\\ \delta & 0
\end{array}\right):\gamma,\delta\in\MatnC\right\}.\]
We denote the projections along this direct sum decomposition as
\begin{align*}
& & \lie{g}\to\mathfrak{k},\  
x=\left(\begin{array}{cc} x_1 &
x_2 \\ x_3 & x_4 \end{array}\right)\mapsto
x^\theta:=\left(\begin{array}{cc} x_1 & 0 \\ 0 & x_4
\end{array}\right)\quad\mbox{and}\quad
& \lie{g}\to\mathfrak{s},\  
x\mapsto
x^{-\theta}:=\left(\begin{array}{cc} 0 & x_2 \\ x_3 & 0
\end{array}\right).
\end{align*}

\subsection{The nilpotent varieties $\mathcal{N}_\mathfrak{k}$ and $\mathcal{N}_\mathfrak{s}$}
\label{section-2.4.2}
We denote by $\mathcal{N}_\mathfrak{k}$ and $\mathcal{N}_\mathfrak{s}$ the closed subsets of nilpotent elements in $\mathfrak{k}$ and $\mathfrak{s}$, respectively.
Each one of these nilpotent sets has a finite number of $K$-orbits
that are parametrized as follows. We also indicate the closure relations among orbits.  

The decomposition of $\mathcal{N}_\mathfrak{k}$ into $K$-orbits is given by
\[\mathcal{N}_\mathfrak{k}=\bigcup_{(\lambda,\mu)\in\partitionsof{n}^2}\mathcal{O}_\lambda\times\mathcal{O}_\mu\]
where, by abuse of notation, $\mathcal{O}_\lambda\times\mathcal{O}_\mu$ stands for the set
of elements $\left(\begin{smallmatrix} \alpha & 0 \\ 0 & \beta \end{smallmatrix}\right)$
with $\alpha\in\mathcal{O}_\lambda$ and $\beta\in\mathcal{O}_\mu$.
Thus we get a bijective parametrization $\mathcal{N}_\mathfrak{k}/K\cong \partitionsof{n}^2$ of nilpotent $ K $-orbits.

We recall the definition of the {\em dominance order} on partitions of $n$ (or Young diagrams).
Let $ \numberofboxes{\lambda}{k} $ denote the number of boxes in the first $k$ columns of $\lambda$.
We set $\lambda\preceq\lambda'$ if
\[
\numberofboxes{\lambda}{k} \geq \numberofboxes{\lambda'}{k} \qquad \forall k\geq 1.\]
Then, we have
$\mathcal{O}_\lambda\subset\overline{\mathcal{O}_{\lambda'}}$ if and
only if $\lambda\preceq\lambda'$, and consequently
\[\mathcal{O}_\lambda\times\mathcal{O}_\mu\subset\overline{\mathcal{O}_{\lambda'}\times\mathcal{O}_{\mu'}}
\quad\mbox{if and only if}\quad \mbox{$\lambda\preceq\lambda'$ and $\mu\preceq\mu'$.}\]

For describing the $K$-orbits of $\mathcal{N}_\mathfrak{s}$, we need further notation.

\begin{definition}
\label{defOLambda}
{\rm (a)} A {\em signed Young diagram} (of signature $(n,n)$)
is a Young diagram of size $2n$ whose boxes are filled in with $n$ symbols $+$ and $n$ symbols $-$
so that:
\begin{itemize}
\item two consecutive boxes of the same row have opposite signs, so that each row is a sequence of alternating signs;
\item we identify two such fillings up to permutation of rows, in particular 
we can standardize the filling in such a way that,
among rows which have the same length, the rows starting with a $+$ are above those starting with a $-$ (if there is any).
\end{itemize}
Let $\signpartitionsof{2n}$ denote the set of signed Young diagrams of size $2n$.
For $\Lambda\in\signpartitionsof{2n}$, the shape of $\Lambda$ is an element of $\partitionsof{2n}$, denoted by $\mathrm{shape}(\Lambda)$.

{\rm (b)} We introduce the dominance order on signed Young diagrams. 
Let $ \numberofsigns{\Lambda}{k}{+} $ (resp. $ \numberofsigns{\Lambda}{k}{-} $) denote the number of $+$'s (resp. $-$'s)
contained in the first $k$ columns of $\Lambda$.
Given $\Lambda,\Lambda'\in\signpartitionsof{2n}$, we set $\Lambda\preceq\Lambda'$ if
\[
\numberofsigns{\Lambda}{k}{+} \geq \numberofsigns{\Lambda'}{k}{+}
\quad
\mbox{and}\
\quad
\numberofsigns{\Lambda}{k}{-} \geq \numberofsigns{\Lambda'}{k}{-}
\qquad
\forall k\geq 1.
\]
Note that
$\Lambda\preceq\Lambda' $ implies $ \mathrm{shape}(\Lambda)\preceq\mathrm{shape}(\Lambda')$ where the latter relation is the dominance ordering.
%
For instance,
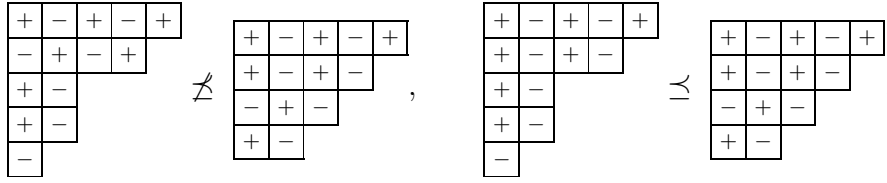
\[\tableau{+-+-+,-+-+,+-,+-,-}\not\preceq\ \tableau{+-+-+,+-+-,-+-,+-},\qquad \tableau{+-+-+,+-+-,+-,+-,-}\preceq\ \tableau{+-+-+,+-+-,-+-,+-}.\]

{\rm (c)}
Given a signed Young diagram $\Lambda$, we denote by $\mathfrak{O}_\Lambda$ the set of nilpotent elements $x\in\mathfrak{s}$ which have a Jordan basis $(\varepsilon_c)$ indexed by the boxes $c\in \Lambda$ such that
\begin{itemize}
\item the vector $\varepsilon_c$ belongs to the subspace $V^+$ (resp. $V^-$) whenever the box $c$ is filled in with a $+$ (resp. a $-$);
\item if $c$ belongs to the first column of $\Lambda$, then $x(\varepsilon_c)=0$; otherwise, then $x(\varepsilon_c)=\varepsilon_{c'}$, where $c'$ is the box on the left of $c$.
\end{itemize}
For instance, if 
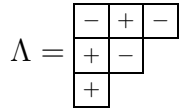
\[
\Lambda=\tableau{-+-,+-,+}
\]
then we obtain
\[
x = \left(\begin{array}{ccc|ccc}
 & & & 0 & 1 & 0 \\
 & 0 & & 0 & 0 & 1 \\
 & & & 0 & 0 & 0 \\ \hline
 1 & 0 & 0 \\
 0 & 0 & 0 & & 0 \\
 0 & 0 & 0
\end{array}\right)\in\mathfrak{O}_\Lambda, 
\qquad
x : \begin{cases}
e^-_1 \to e^+_1 \to e^-_2 \to 0
\\
e^-_3 \to e^+_2 \to 0
\\
e^+_3 \to 0
\end{cases} 
\]
where $ \{ e^{\pm}_i \}_{1 \leq i \leq 3} $ is the standard basis of $ V^+ = \C^3 \oplus \{0\} $, resp.
$ V^- = \{ 0 \} \oplus \C^3 $. 
\end{definition}

\begin{proposition}[\cite{Ohta}]
\label{P2.2}
\begin{thmenumeralph}
\item The subsets $\mathfrak{O}_\Lambda$ (for $\Lambda\in\signpartitionsof{2n}$) are exactly the $K$-orbits of the nilpotent set $\mathcal{N}_\mathfrak{s}$.
\item The signed Young diagram $\Lambda$ such that $x\in\mathfrak{O}_\Lambda$ is also characterized by
\[
\numberofsigns{\Lambda}{k}{\pm} = \dim (\ker x^k) \cap V^\pm\quad \forall k\geq 1.
\]
\item $\mathfrak{O}_\Lambda\subset\overline{\mathfrak{O}_{\Lambda'}}$ if and only if $\Lambda\preceq\Lambda'$.
\end{thmenumeralph}
\end{proposition}

\begin{remark}
\label{R2.1}
{\rm (a)}
It follows from Proposition \ref{P2.2} that the variety $\mathcal{N}_\mathfrak{s}$ is
not irreducible (contrary to $\mathcal{N}_\mathfrak{k}$). Indeed $\mathcal{N}_\mathfrak{s}$ has exactly two irreducible components which are the closures of the $K$-orbits corresponding to the horizontal signed Young diagrams
\[\tableau{+-+-\cdots}\qquad\mbox{and}\qquad\tableau{-+-+\cdots}\qquad\mbox{(of size $2n$)}.\]

{\rm (b)}
If $\Lambda\in\signpartitionsof{2n}$ is a signed Young
diagram with at most $n$ columns, then it is easy to see that we have $\Lambda\preceq
\Lambda_0$, $\Lambda\preceq\Lambda_+$, or $\Lambda\preceq\Lambda_-$,
where $\Lambda_0,\Lambda_+,\Lambda_-$ are the signed Young diagrams
given in Theorem \ref{T1}\,{\rm (a)} when $n$ is even and Theorem
\ref{T1}\,{\rm (b)} when $n$ is odd respectively.
Thus
\[\{x\in\mathcal{N}_\mathfrak{s}:x^n=0\}\subset\overline{\mathfrak{O}_{\Lambda_0}}\cup
\overline{\mathfrak{O}_{\Lambda_+}}\cup\overline{\mathfrak{O}_{\Lambda_-}}\]
if $ n \geq 2 $.  
\end{remark}

\part{A generalized Steinberg map arising from the action of a pair of Borel subgroups on the space of $n\times n$ matrices}

A partial permutation on the set $ [n] := \{ 1, 2, \cdots, n \} $
is an injective map from a (possibly empty) subset $ J \subset [n] $ to $ [n] $;
equivalently it can be viewed as a degenerate permutation matrix.
The partial permutations form a semigroup denoted by $\ppermutationsof{n}$.

In Section \ref{section-orbits}, we consider the simultaneous action of a
pair of Borel subgroups by (left and right) multiplication on the space of $n\times n$ matrices,
and we show that $\mathfrak{T}_n$ is a complete set of representatives
for the orbits.
Note that the considered action extends the Bruhat decomposition of the group
$\GLnC$, and in particular $ \ppermutationsof{n} $ naturally contains the group of permutations $ \permutationsof{n} $.

In Section \ref{section-Phi}, we use this action to give a bijective correspondence between
$ \ppermutationsof{n} $ and a set of pairs of tableaux with additional partition; see Theorems~\ref{T2}--\ref{thm:3.9}.
This correspondence naturally extends the original Robinson-Schensted correspondence for
permutations.

\section{Action of a pair of Borel subgroups on the space of $n\times n$ matrices}

\label{section-orbits}\label{section-3.1}

In this part of the paper, 
we consider two Borel subgroups $B_1,B_2$ of $\GLnC$.\ Let $\mathfrak{b}_1,\mathfrak{b}_2\subset\lie{gl}_n$ be the corresponding Borel subalgebras and let $\mathfrak{n}_1,\mathfrak{n}_2$ be their respective nilradicals. We assume that $B_1,B_2$ contain the standard torus of $\GLnC$.

Let us consider an action of the group $B_1\times B_2$ on the space of
$n\times n$ matrices $\Mat_n$, which
is given by
\[(b_1,b_2)\cdot x:=b_1xb_2^{-1} \qquad\forall (b_1,b_2)\in B_1\times B_2,\ \forall x\in \Mat_n.\]
As explained in Section \ref{section-2.1.2}, this action gives rise to a conormal variety $\mathcal{Y}_{\Mat_n}\subset T^*\Mat_n=\Mat_n\times\Mat_n^*$, which is stable by the Hamiltonian action of the group $B_1\times B_2$ on the cotangent bundle.

\begin{proposition}\label{P2}
Identifying $\Mat_n^*$ with $\Mat_n$ through the trace form $\langle x,y\rangle:=\mathrm{Tr}(xy)$,
the conormal variety $\mathcal{Y}_{\Mat_n}$ is identified with the variety
\[\mathcal{Y}_{\Mat_n}=\{(x,y)\in\Mat_n\times\Mat_n:xy\in\mathfrak{n}_1,\ yx\in\mathfrak{n}_2\}\]
endowed with the action of $B_1\times B_2$ obtained by restriction of the following action on $\Mat_n\times\Mat_n$:
\begin{equation}
\label{5}
(b_1,b_2)\cdot(x,y)=(b_1xb_2^{-1},b_2yb_1^{-1})\quad\forall (b_1,b_2)\in B_1\times B_2,\ \forall (x,y)\in\Mat_n\times\Mat_n.
\end{equation}
\end{proposition}

\begin{proof}
While identifying the cotangent bundle $T^*\Mat_n=\Mat_n\times\Mat_n^*$ with the space $\Mat_n\times\Mat_n$ through the trace form,
the action of $B_1\times B_2$ on $T^*\Mat_n$ translates into the action on $\Mat_n\times\Mat_n$ given in (\ref{5}), because
for $y\in\Mat_n\cong\Mat_n^*$ and $(b_1,b_2)\in B_1\times B_2$ we have
\[\langle y,(b_1,b_2)^{-1}\cdot z\rangle=\mathrm{Tr}(yb_1^{-1}zb_2)=\mathrm{Tr}(b_2yb_1^{-1}z)=\langle b_2yb_1^{-1},z\rangle\quad \forall z\in \Mat_n.\]
By differentiating the action of $B_1\times B_2$ on $\Mat_n$ 
we obtain the infinitesimal action of Lie algebras:
\[(\beta_1,\beta_2)\cdot x=\beta_1x-x\beta_2\quad\forall (\beta_1,\beta_2)\in \mathfrak{b}_1\times\mathfrak{b}_2,\ \forall x\in\Mat_n.\]
Then by (\ref{4}), for any $(x,y)\in\Mat_n\times\Mat_n\cong\Mat_n\times\Mat_n^*$, we have
\begin{eqnarray*}
(x,y)\in\mathcal{Y}_{\Mat_n} & \iff & \forall (\beta_1,\beta_2)\in\mathfrak{b}_1\times\mathfrak{b}_2,\ \langle (\beta_1,\beta_2)\cdot x,y\rangle=0 \\
 & \iff & \forall (\beta_1,\beta_2)\in\mathfrak{b}_1\times\mathfrak{b}_2,\ \mathrm{Tr}(\beta_1xy-x\beta_2y)=0 \\
 & \iff & \forall \beta_1\in\mathfrak{b}_1,\ \forall\beta_2\in\mathfrak{b}_2,\ \mathrm{Tr}(\beta_1xy)=\mathrm{Tr}(yx\beta_2)=0 \\
 & \iff & xy\in\mathfrak{b}_1^\perp(=\mathfrak{n}_1)\quad\mbox{and}\quad yx\in\mathfrak{b}_2^\perp(=\mathfrak{n}_2)
\end{eqnarray*}
where the notation $\perp$ refers to the orthogonal space with respect to the trace form on $\Mat_n$. The proof of the proposition is complete.
\end{proof}

The action of $B_1\times B_2$ on the space
$\Mat_n$ has a finite number of orbits, as shown by the
following statement. 
(This also follows from the general theory of spherical varieties, knowing that 
this action
is the restriction of an action of $\GLnC\times\GLnC$, and it has
an open dense orbit by virtue of the Bruhat decomposition.)
The orbits are parametrized by the set of
so-called partial permutations. In Corollary \ref{C1} we deduce a
description of the irreducible components of the conormal variety.

\begin{definition}
\label{D1} We call a matrix $\tau\in\Mat_n$ a {\em partial permutation}
if each row (resp. column) of $\tau$ has at most one nonzero entry,
equal to 1. (Equivalently $\tau$ is obtained from a permutation
matrix by erasing some $1$'s, replaced by $0$'s.) Let
$\mathfrak{T}_n\subset\Mat_n$ denote the subset of partial
permutations.
\end{definition}

\begin{proposition}
\label{P3} The set $\mathfrak{T}_n$ of partial permutations is a
complete set of representatives of the $B_1\times B_2$-orbits in
$\Mat_n$. In other words every $B_1\times B_2$-orbit of
$\Mat_n$ is of the form
\[\mathbb{O}_\tau:=B_1\tau B_2=\{b_1\tau b_2:b_1\in B_1,\ b_2\in B_2\}\]
for a unique partial permutation $\tau\in\mathfrak{T}_n$.
\end{proposition}

\begin{proof}
By assumption $B_1$ and $B_2$ are Borel subgroups of
$\GLnC$ which contain the standard torus of
diagonal matrices. Thus there are permutations
$\sigma_1,\sigma_2\in\mathfrak{S}_n$ such that
$B_i=\sigma_iB\sigma_i^{-1}$ for $i\in\{1,2\}$, where
$B\subset\GLnC$ is the Borel subgroup of
upper-triangular matrices. Since $\tau\mapsto
\sigma_1\tau\sigma_2^{-1}$ is a bijection on the set of partial
permutations, we may assume without loss of generality that
$B_1=B_2=B$.

A matrix remains in the same $B\times B$-orbit as $a\in\Mat_n$ whenever it is obtained from $a\in\Mat_n$ by adding to a given row (resp. column) a sum of rows (resp. columns) situated below it (resp. on its left) or by multiplying a row (a column) by a nonzero scalar.\ Applying Gauss elimination, it follows that every orbit $BaB$ contains a matrix of the form $\tau\in\mathfrak{T}_n$.

We observe that the maps $a\mapsto d_{i,j}(a):=\mathrm{rank}\,(a_{k,\ell})_{i\leq  k\leq n,\, 1\leq \ell\leq j}$ (for $i,j\in\{1,\ldots,n\}$) are constant on the $B\times B$-orbits. Since we have $d_{i,j}(\tau)=d_{i,j}(\tau')$ for any $(i,j)$ only if $\tau=\tau'$, it follows that every orbit contains exactly one representative of the form $\tau\in\mathfrak{T}_n$.
\end{proof}

\begin{corollary}
\label{C1} The conormal variety $\mathcal{Y}_{\Mat_n}$ of
Proposition \ref{P2} is equidimensional and its irreducible
components are parametrized by the partial permutations.\ More
precisely every irreducible component is of the form
\[\mathcal{Y}_\tau:=\overline{T_{\mathbb{O}_\tau}^*\Mat_n}=\overline{(B_1\times B_2)\cdot\{(\tau,y):y\in\Mat_n,\ \tau y\in\mathfrak{n}_1,\ y\tau\in\mathfrak{n}_2\}}\]
for a unique $\tau\in\mathfrak{T}_n$.
\end{corollary}

\begin{proof}
This follows from Propositions \ref{P1}, \ref{P2}, and \ref{P3}.
\end{proof}

\section{Generalized Steinberg map and Robinson-Schensted correspondence for partial permutations}

\label{section-Phi}

The setting and the notation are the same as in Section \ref{section-orbits}, except that we assume for simplicity
\[
B_1=B_2=B\qquad\mbox{and}\qquad
\lie{n}_1=\lie{n}_2=\lie{n}
\]
where $B\subset \GL_n$ is the Borel subgroup of upper triangular matrices and $\mathfrak{n}\subset\Mat_n$ is the subspace of strictly upper triangular matrices.
The goal of this section is to define and calculate a generalized Steinberg map on the set of partial permutations and, to this end, it is indeed preferable to standardize the notation.

\subsection{The map $\Phi$ on partial permutations}\label{section-3.3}

We consider the map
\[
\varphi=(\varphi_1,\varphi_2):\mathcal{Y}_{\Mat_n} \longrightarrow \lie{n}{\times}\lie{n}, \qquad 
(x,y) \longmapsto (xy, yx).
\]
Note that  $\varphi$ is $B\times B$-equivariant, where $B\times B$ acts on $\lie{n}{\times}\lie{n}$ by the adjoint action on each factor.

For any irreducible component $\mathcal{Y}_\tau\subset\mathcal{Y}_{\Mat_n}$, the set $\varphi(\mathcal{Y}_\tau)\subset\lie{n}{\times}\lie{n}$ is irreducible.
Therefore there exist a pair of nilpotent orbits $\mathcal{O}_\lambda$ and $\mathcal{O}_\mu$ of $\Mat_n$ which intersect $\varphi(\mathcal{Y}_\tau)$ densely.
In other words, we have
\[\overline{(\GLnC\times \GLnC)\cdot\varphi(\mathcal{Y}_\tau)}=\overline{\mathcal{O}_\lambda}\times\overline{\mathcal{O}_\mu}.\]
This yields a map
\[
\Phi : \mathfrak{T}_n \cong \mathrm{Irr}(\mathcal{Y}_{\Mat_n}) \longrightarrow \partitionsof{n}{\times}\partitionsof{n}, \qquad 
\tau \longmapsto (\lambda,\mu).
\]
Let us denote $\Phi_1(\tau)=\lambda$ and $\Phi_2(\tau)=\mu$ for the first and the second component of $\Phi$, respectively.  
We call maps $ \Phi, \Phi_1, \Phi_2 $ \emph{generalized Steinberg maps}.

The next lemma immediately follows from the definition.

\begin{lemma}\label{L1}
With the above notation, $\mathcal{O}_\lambda$ (resp. $\mathcal{O}_\mu$) is characterized as the nilpotent orbit which meets any of the sets
\begin{eqnarray*}
 & \varphi_1(\mathcal{Y}_\tau),\quad \varphi_1(T^*_{\mathbb{O}_\tau}\Mat_n),\quad \{\tau y:y\in\Mat_n\mbox{\rm \ such that }(\tau y,y\tau)\in\lie{n}{\times}\lie{n}\} \\
\mbox{(resp.} & \varphi_2(\mathcal{Y}_\tau),\quad \varphi_2(T^*_{\mathbb{O}_\tau}\Mat_n),\quad \{y\tau:y\in\Mat_n\mbox{\rm \ such that }(\tau y,y\tau)\in\lie{n}{\times}\lie{n}\} & \mbox{)}
\end{eqnarray*}
along dense open subsets.
\end{lemma}

\subsection{The map $\Phi$ on permutations}\label{section-R1}

The decomposition of $\Mat_n$ into $B\times B$-orbits is evidently an extension of the Bruhat decomposition of $G=\GLnC$:
\[G=\bigsqcup_{\sigma\in\mathfrak{S}_n}B\sigma B\subset\Mat_n=\bigsqcup_{\tau\in\mathfrak{T}_n}B\tau B.\]
Here the permutations $\sigma\in\mathfrak{S}_n$ are viewed as
permutation matrices, in particular they are invertible matrices.
Thus, for $(x,y)\in T^*_{B\sigma B}\Mat_n$, 
the element $x\in B\sigma B\subset G$ is an invertible matrix, hence the matrices $xy$ and $yx=x^{-1}(xy)x$ are $ G $-conjugate.  
This readily implies that they generate the same nilpotent orbit, and we get
$\Phi_1(\sigma)=\Phi_2(\sigma)$ whenever $\sigma$ is a permutation.  
In contrast, for a partial permutation $ \tau $, we have 
$\Phi_1(\tau)\not=\Phi_2(\tau)$ in general.  See Example \ref{example-1} below.

Let us compare the conormal bundles for 
the diagonal action of $G$ on $ \mathcal{B} \times \mathcal{B} $ 
and for the action of $ B \times B $ on $ G $ via the left and right multiplications.

For $ \sigma \in \mathfrak{S}_n $, if we put 
$\mathbb{O}_\sigma=B\sigma B\subset
G \subset\Mat_n$ and 
$\mathcal{Z}_\sigma:=G\cdot(\lie{b},\wconjn{\sigma}{b})\subset
\mathcal{B}\times\mathcal{B}$, the corresponding conormal bundles are 
\begin{eqnarray*}
T_{\mathbb{O}_\sigma}^*G=T_{\mathbb{O}_\sigma}^*\Mat_n & = & (B\times B)\cdot \{(\sigma,y):\sigma y\in\lie{n},\ y\sigma\in\lie{n}\} \\
 & = & (B\times B)\cdot \{(\sigma,y):\sigma y\in\lie{n}\cap\wconjn{\sigma}{n}\}
\end{eqnarray*}
and 
\[T_{\mathcal{Z}_\sigma}^*(\mathcal{B}\times\mathcal{B})=G\cdot \{(\lie{b},\wconjn{\sigma}{b},z):z\in\lie{n}\cap\wconjn{\sigma}{n}\}\]
(see (\ref{conormal-bundle})).  
The fibers are the very same and we get 
\[
\closure{G\cdot(\varphi_1(T^*_{\mathbb{O}_\sigma}G))} 
= \closure{G\cdot(\varphi_2(T^*_{\mathbb{O}_\sigma}G))} 
= \closure{G\cdot(\mathfrak{n}\cap\wconjn{\sigma}{n})} 
= \closure{\pi(T_{\mathcal{Z}_\sigma}^*(\mathcal{B}\times\mathcal{B}))}
\]
with $\pi$ the projection given in Section \ref{section-Steinberg}.  
Thus these sets have the same dense nilpotent orbit $ \calorbit_{\lambda} $.  
We summarize this into the following proposition.

\begin{proposition}
\label{proposition-oldequation6}
For any permutation $\sigma\in\mathfrak{S}_n$, we have
\[
\Phi(\sigma)=(\mathrm{St}(\sigma),\mathrm{St}(\sigma)), 
\]
where $ \mathrm{St} : \mathfrak{S}_n \to \partitionsof{n} $ is the Steinberg map for $ \GLnC $.
\end{proposition}


\begin{remark}
\label{R7.3}
Let $\sigma$ be a permutation and set $\lambda=\Phi_1(\sigma)=\Phi_2(\sigma)$. Then we have
\[\overline{\varphi_1(T^*_{\mathbb{O}_\sigma}G)}=\overline{B\cdot(\mathfrak{n}\cap\wconjn{\sigma}{n})}=:\overline{\mathcal{V}}_\sigma.\]
The set
$\mathcal{V}_\sigma:=\overline{\mathcal{V}}_\sigma\cap\mathcal{O}_\lambda$
is called an {\em orbital variety}
(and $\overline{\mathcal{V}}_\sigma$ is its closure). It is an irreducible component
of the variety $\mathcal{O}_\lambda\cap\mathfrak{n}$ and every
component of $\mathcal{O}_\lambda\cap\mathfrak{n}$ is of the form
$\mathcal{V}_{\sigma'}$ for some $\sigma'\in\mathfrak{S}_n$ such
that $\Phi_1(\sigma')=\Phi_2(\sigma')=\lambda$; see \cite{Joseph}.

Note also that
$\overline{\varphi_2(T^*_{\mathbb{O}_\sigma}G)}=\overline{\mathcal{V}}_{\sigma^{-1}}$.

In fact, even for a {\em partial} permutation $ \tau\in\ppermutationsof{n} $, 
we obtain that the (closures of the) images of 
the conormal bundle $ T^*_{\mathbb{O}_\tau}\Mat_n $
by $\varphi_1$ and $\varphi_2$
are closures of orbital varieties. See Corollary \ref{C7.5} below.
\end{remark}

\subsection{Calculation of the map $\Phi$ for partial permutations}

\label{section-3.4}

\begin{notation} {\rm (a)} As in the previous subsections, $B\subset\GLnC$ denotes the Borel subgroup of upper triangular matrices
and $\mathfrak{n}\subset\Mat_n$ is the subspace of strictly upper triangular matrices.

{\rm (b)} We can view a partial permutation $\tau\in\mathfrak{T}_n$
as a map $\tau:\{1,\ldots,n\}\to\{0,1,\ldots,n\}$ such that
$\#\tau^{-1}(i)\leq 1$ whenever $i\not=0$; the corresponding matrix
has 1 as an entry in the position $(\tau(j),j)$ whenever $\tau(j)\not=0$ and
$0$'s elsewhere. The map $\tau$ can also be written in the form
\begin{equation}\label{eq:expression.pp.tau}
\tau=\left(\begin{array}{ccccccc}
j_1 & \cdots & j_r & m_1 & \cdots & m_{n-r} \\
i_1 & \cdots & i_r & 0 & \cdots & 0
\end{array}\right) ,
\end{equation}
which means that $\tau(j_k)=i_k$ for $k\in\{1,\ldots,r\}$ and $\tau(m_t)=0$ for $t\in\{1,\ldots,n-r\}$,
and we will call
\[\left(\begin{array}{ccccccc}
j_1 & \cdots & j_r  \\
i_1 & \cdots & i_r
\end{array}\right)
\]
the {\em nondegenerate part of $\tau$}, 
which is a bijection between the sets $J(\tau):=\{j_1,\ldots,j_r\}$ and $I(\tau):=\{i_1,\ldots,i_r\}$.  
We also write $M(\tau):=\{m_1,\ldots,m_{n-r}\}$ (the ``kernel'' of $\tau$) and $L(\tau):=\{1,\ldots,n\}\setminus I(\tau)$ (the ``coimage'' of $\tau$).

{\rm (c)} Observe that the transpose of a partial permutation is a
partial permutation, namely for $\tau$ in \eqref{eq:expression.pp.tau}, 
the transpose is given by 
\[
{}^t\tau=\left(\begin{array}{ccccccc}
i_1 & \cdots & i_r & \ell_1 & \cdots & \ell_{n-r} \\
j_1 & \cdots & j_r & 0 & \cdots & 0
\end{array}\right) , 
\]
where $\{\ell_1,\ldots,\ell_{n-r}\} = L(\tau)$. In particular the nondegenerate part of ${}^t\tau$ is the inverse of the nondegenerate part of $\tau$.
\end{notation}

In the following statement we use the notation in Section \ref{section-RS} related to the Robinson-Schensted algorithm and the operation $*$ on Young tableaux.

\begin{theorem}\label{T2} 
Consider a partial permutation 
\[
\tau=
\left(\begin{array}{cccccccc} j_1 & \cdots & j_r & m_1 & \cdots & m_s \\ i_1 & \cdots & i_r & 0 & \cdots & 0 \end{array}\right)
\in\mathfrak{T}_n
\quad (\mbox{$r=\mathrm{rank}\,\tau$ and $r+s=n$}) . \] 
Let
$\sigma=\left(\begin{array}{cccccccc} j_1 & \cdots & j_r
\\ i_1 & \cdots & i_r \end{array}\right)$
be the nondegenerate part
of $\tau$; it yields a pair of Young tableaux
$(\RSl(\sigma),\RSr(\sigma))$ defined through the
Robinson-Schensted algorithm. Assume that $m_1<\cdots<m_s$. Let
$\{\ell_1,\ldots,\ell_s\}:=\{1,\ldots,n\}\setminus\{i_1,\ldots,i_r\}$
and assume that $\ell_1<\cdots<\ell_s$.

Then, the image of the generalized Steinberg map 
$\Phi(\tau)=(\Phi_1(\tau),\Phi_2(\tau)) \in \partitionsof{n}^2 $ is given by
\[
\Phi_1(\tau)=\mathrm{shape}\Big(\RSl(\sigma)*\tableaul{\ell_1 \\ {}^{\vdots} \\
\ell_s}\,\Big),\qquad 
\Phi_2(\tau)=\mathrm{shape}\Big(\,\tableaul{m_1
\\ {}^{\vdots} \\ m_s}*\RSr(\sigma)\Big).
\]
\end{theorem}

\begin{proof}
We denote
\[V_1(\tau)
:=\{\tau y : y\in\Mat_n\ \mbox{s.t.}\ (\tau y,y\tau)\in\lie{n}{\times}\lie{n}\},\]
and similarly,
\[
V_2(\tau)
:=\{ y\tau : y\in\Mat_n\ \mbox{s.t.}\ (\tau y,y\tau)\in\lie{n}{\times}\lie{n}\}.\]
By Lemma \ref{L1}, the nilpotent orbit corresponding to  $\Phi_i(\tau)$ intersects 
$V_i(\tau)$ along a dense open subset and this serves as characterization of the orbit (for $i\in\{1,2\}$).

Let us compute the spaces
$V_1(\tau)$ and $V_2(\tau)$.
Let $e_{i,j}$ stand for the elementary matrix with $1$ at the position $(i,j)$ and $0$'s elsewhere. It is straightforward to see that
\[\tau e_{i,j}\in\mathfrak{n}\iff \left\{\begin{array}{ll}
i\in\{m_1,\ldots,m_{s}\} & \mbox{(in which case $\tau e_{i,j}=0$), or} \\[1mm]
i=j_k\mbox{ and $i_k<j$} & \mbox{(in which case $\tau e_{i,j}=e_{i_k,j}$)}
\end{array}\right.\]
and
\[e_{i,j}\tau \in\mathfrak{n}\iff \left\{\begin{array}{ll}
j\in\{\ell_1,\ldots,\ell_{s}\} & \mbox{(in which case $e_{i,j}\tau=0$), or} \\[1mm]
j=i_k\mbox{ and $j_k>i$} & \mbox{(in which case $e_{i,j}\tau=e_{i,j_k}$).}
\end{array}\right.\]
This yields
\[V_1(\tau)=\bigoplus_{(i,j)\in D_1} \C e_{i,j}
\quad\mbox{and}\quad
V_2(\tau)=\bigoplus_{(i,j)\in D_2} \C e_{i,j}\]
where
\begin{eqnarray*}
D_1 & = & \{(i_k,\ell_t):i_k<\ell_t\}\cup\{(i_k,i_t):i_k<i_t,\ j_k<j_t\} \quad \mbox{and}\quad 
\\[1mm]
D_2 & = & \{(m_k,j_t):m_k<j_t\}\cup\{(j_k,j_t):j_k<j_t,\ i_k<i_t\}.
\end{eqnarray*}

For $w\in\mathfrak{S}_n$, we put
\begin{equation*}
D(w):=\{(i,j):1\leq i<j\leq n,\ w^{-1}(i)<w^{-1}(j)\}. 
\end{equation*}
Note that the equality
\begin{equation}
\bigoplus_{(i,j)\in D(w)} \C e_{i,j} = \lie{n} \cap \wconjn{w}{n} 
\end{equation}
holds. By the classical Steinberg theory (Section \ref{section-Steinberg}), we already know the nilpotent orbit which intersects $ \lie{n}\cap\wconjn{w}{n} $ densely.
Thus let us see that $ D_1 $ and $ D_2 $ are exactly of the form $ D(w) $ above for some $ w $.  

We may assume that $j_1<\ldots<j_r$. Moreover let $\{i'_1,\ldots,i'_r\}:=\{i_1,\ldots,i_r\}$ with $i'_1<\ldots<i'_r$ and set $j'_k=\sigma^{-1}(i'_k)$ for all $k\in\{1,\ldots,r\}$.  
Using these indices, 
we define permutations $w_1,w_2$ by
\begin{equation}
\label{w1-w2}
w_1:=\left(\begin{array}{cccccccc} 1 & \cdots & r & r{+}1 & \cdots & n \\ i_1 & \cdots & i_r & \ell_{s} & \cdots & \ell_1 \end{array}\right),\ \ 
w_2:=\left(\begin{array}{cccccccc} 1 & \cdots & s & s{+}1 & \cdots & n \\ m_{s} & \cdots & m_1 & j'_1 & \cdots & j'_r
\end{array}\right) . 
\end{equation}
Then it is easy to see that $ D_i = D(w_i) \; (i = 1, 2) $.  
Whence
\begin{equation}
\label{10}
V_i(\tau)=\bigoplus_{(k,\ell)\in D(w_i)} \C e_{k,\ell}=\mathfrak{n}\cap \wconjn{w_i}{n}\quad\mbox{for $i\in\{1,2\}$.}
\end{equation}

By definition, the partition $\lambda=\mathrm{St}(w)$ encodes the nilpotent orbit which intersects the space $\mathfrak{n}\cap \wconjn{w}{n}$ along a dense open subset (see Section \ref{section-Steinberg}).
Therefore (\ref{10}) implies 
\[(\Phi_1(\tau),\Phi_2(\tau))=(\mathrm{St}(w_1),\mathrm{St}(w_2)).\]
By Theorem \ref{T-RS}, the Steinberg map
$\mathrm{St}$ can be computed by means of the Robinson-Schensted
algorithm. Namely for $w_1,w_2$ we deduce that
\begin{eqnarray*}
\Phi_1(\tau)=\mathrm{St}(w_1) & = & \mathrm{shape}(\RSl(w_1))=\mathrm{shape}(\rinsert(i_1,\ldots,i_r,\ell_{s},\ldots,\ell_1)) \\
 & = & \mathrm{shape}(\rinsert(i_1,\ldots,i_r)\leftarrow \ell_{s}\leftarrow\cdots\leftarrow \ell_1) \\
 & = & \mathrm{shape}\Big(\RSl(\sigma)*\tableaul{\ell_1 \\ { }^{\vdots} \\ \ell_s}\,\Big)
\end{eqnarray*}
and
\begin{eqnarray*}
\Phi_2(\tau)=\mathrm{St}(w_2) & = & \mathrm{shape}(\RSl(w_2))=\mathrm{shape}(\rinsert(m_{s},\ldots,m_1,j'_1,\ldots,j'_r)) \\
 & = & \mathrm{shape}(\cinsert(j'_r,\ldots,j'_1,m_1,\ldots,m_s)) \\
 & = & \mathrm{shape}(m_s\to\cdots\to m_1\to\cinsert(j'_r,\ldots,j'_1)) \\
 & = & \mathrm{shape}\Big(\,\tableaul{m_1 \\ { }^{\vdots} \\ m_s}*\RSr(\sigma)\,\Big)
\end{eqnarray*}
where we use that $\cinsert(j'_r,\ldots,j'_{1})=\rinsert(j'_1,\ldots,j'_{r})$ and
\[\rinsert(j'_1,\ldots,j'_{r})=\RSl\big(\left(\begin{array}{ccc} i'_1 & \cdots & i'_r \\ j'_1 & \cdots & j'_r \end{array}\right)\big)=\RSl(\sigma^{-1})=\RSr(\sigma).\] The proof is complete.
\end{proof}

Recall from Remark \ref{R7.3} the notion of orbital variety.

\begin{corollary}
\label{C7.5}
Let $\tau$ be a partial permutation and let
$(\lambda,\mu)=(\Phi_1(\tau),\Phi_2(\tau))$. Then the varieties
$\overline{\varphi_1(T^*_{\mathbb{O}_\tau}\Mat_n)}$
and
$\overline{\varphi_2(T^*_{\mathbb{O}_\tau}\Mat_n)}$
are closures of orbital varieties of $\mathcal{O}_\lambda$ and $\mathcal{O}_\mu$, respectively.
Namely we have
\[
\overline{\varphi_1(T^*_{\mathbb{O}_\tau}\Mat_n)}\cap\mathcal{O}_\lambda=\mathcal{V}_{w_1}
\quad\mbox{and}\quad
\overline{\varphi_2(T^*_{\mathbb{O}_\tau}\Mat_n)}\cap\mathcal{O}_\mu=\mathcal{V}_{w_2}
\]
where $w_1,w_2$ are the permutations defined in (\ref{w1-w2}).
\end{corollary}

\begin{proof}
By definition we have
$\varphi_i(T^*_{\mathbb{O}_\tau}\Mat_n)=B\cdot V_i(\tau)$ (for $i\in\{1,2\}$),
where $V_i(\tau)$ is as in the proof of Theorem \ref{T2}.
Then the claim follows from (\ref{10}).
\end{proof}


The proof of Theorem~\ref{T2} actually involves 
a generalization of the Robinson-Schensted correspondence for partial permutations.  
We state it as a theorem below.

Let $\mathrm{STab}(\lambda)$
denote the set of standard tableaux of shape $ \lambda\in\partitionsof{n} $,
i.e., Young tableaux of shape $\lambda$ with entries $1,\ldots,n$.
Also, for a subdiagram $\nu\subset\lambda$,
we say that  $\lambda\setminus\nu$ is a {\em column strip} (or a vertical strip)
if the skew diagram $\lambda\setminus\nu$
contains at most one box in each row.

\begin{theorem}\label{thm:3.9}
There is a bijective correspondence
between the set of partial permutations $\mathfrak{T}_n$ and the set of triples
\[
\bigsqcup_{\substack{\lambda,\mu\in\partitionsof{n}\\r\in\{0,\ldots,n\}}}
\begin{array}[t]{r}
\big\{
(T_1,T_2,\nu)\in\mathrm{STab}(\lambda)\times\mathrm{STab}(\mu)\times\partitionsof{r}:\nu\subset\lambda,\ \nu\subset\mu, \\[1mm]
\mbox{and $\lambda\setminus\nu$, $\mu\setminus\nu$ are column strips}
\big\},
\end{array}\]
which is given explicitly as
\[\tau\mapsto\Big(\RSl(\sigma)*\tableaul{\ell_1 \\ { }^{\vdots} \\ \ell_s}\,,\ \tableaul{m_1 \\ { }^{\vdots} \\ m_s}*\RSr(\sigma),\ \mathrm{shape}(\RSl(\sigma))\Big)  \]
where $ \sigma $ denotes the nondegenerate part of $\tau$
(with the same notation as in Theorem \ref{T2}).
\end{theorem}

\begin{proof}
By ${}^tT$ we denote the transpose of a Young tableau $T$. 
For a partial permutation $ \tau \in \ppermutationsof{n} $, 
the corresponding triple $(T_1,T_2,\nu)$ 
is constructed in the following way.  
Let us denote the pair of intermediate tableaux by $S_1=\RSl(\sigma)$ and $S_2=\RSr(\sigma)$.  
Then the triple is obtained by 
\begin{equation}
\label{9}
\left\{
\begin{array}{lll}
T_1 & = & S_1\leftarrow \ell_{s}\leftarrow\cdots\leftarrow \ell_1 \\[1mm]
{}^tT_2 & = & {}^tS_2\leftarrow m_1\leftarrow\cdots\leftarrow m_{s} \\[1mm]
\nu & = & \mathrm{shape}(S_1)=\mathrm{shape}(S_2)
\end{array}
\right.
\end{equation}
From \cite[Proposition in \S 1.1]{Fulton}, we see the skew diagrams
$\mathrm{shape}(T_i)\setminus\nu$ ($i\in\{1,2\}$) are column strips,
i.e., the image of the map in the statement is contained in the right-hand side. 
Conversely the same proposition in \cite{Fulton} 
tells that for any triple $(T_1,T_2,\nu)$ in the right-hand side 
we can reverse the procedure to get 
a pair of Young tableaux
$(S_1,S_2)$ and lists of entries $\ell_1<\cdots<\ell_{s}$ and
$m_1<\cdots<m_{s}$ such that (\ref{9}) holds. We denote by
$\{i_1,\ldots,i_r\}$ (resp. $\{j_1,\ldots,j_r\}$) the set of entries
of $S_1$ (resp. $S_2$). By the Robinson-Schensted correspondence,
there is a unique bijection
\[w:\{j_1,\ldots,j_r\}\to\{i_1,\ldots,i_r\}\]
such that $S_1=\RSl(w)$ and $S_2=\RSr(w)$. 
In this way, we can recover the original partial permutation $ \tau $.  
This shows that the map under consideration is bijective,
which completes the proof.
\end{proof}

\begin{remark}
Through the correspondence described in Theorem \ref{thm:3.9}, any {\em permutation} $\sigma\in\mathfrak{S}_n$ corresponds to a triple $(T_1,T_2,\nu)$ such that $\nu=\mathrm{shape}(T_1)=\mathrm{shape}(T_2)$; moreover in this case we have $(T_1,T_2)=(\RSl(\sigma),\RSr(\sigma))$. Thus, in this way, the correspondence in Theorem \ref{thm:3.9} reduces to 
the classical Robinson-Schensted correspondence on the set of permutations.
\end{remark}

\begin{example}
\label{example-1} In Figures \ref{F1}--\ref{F3}, we describe the
correspondence $\tau\mapsto(T_1,T_2,\nu)$ of Theorem \ref{thm:3.9} for
$n=3$. The set $\mathfrak{T}_3$ contains 34 elements. For each partial permutation
$\tau$, we indicate the nondegenerate part $\sigma$, the pair
$(\RSl(\sigma),\RSr(\sigma))$, and the corresponding triple
$(T_1,T_2,\nu)$. We display the list in three parts, according to
the rank $r$ of~$\tau$. \ytableausetup{smalltableaux}
\begin{figure}[p]
{\footnotesize
\begin{tabular}{|c|c|c|c|c|c|}
\hline 
$\tau$ & $\sigma$ & $(\RSl(\sigma),\RSr(\sigma))$ & $(T_1,T_2,\nu)$ & $\Phi(\tau)=\Xik(\left(\begin{smallmatrix}\tau \\ 1_3\end{smallmatrix}\right))$ & $\Xis(\left(\begin{smallmatrix}\tau \\ 1_3\end{smallmatrix}\right))$
\\
\hline\hline
\vphantom{$\begin{smallmatrix} 1 \\ 2 \\ 3 \\ 4 \\ 5 \end{smallmatrix} $}
$\left(\begin{smallmatrix} 1 & 2 & 3 \\ 1 & 2 & 3 \end{smallmatrix}\right)$ & $=\tau$ & $\ytableaushort{123},\ytableaushort{123}$ & $\ytableaushort{123},\ytableaushort{123},\ydiagram{3}$ & $\ydiagram{3},\ydiagram{3}$ & $\ytableaushort{+-+,-+-}$ \\
\hline
\vphantom{$\begin{smallmatrix} 1 \\ 2 \\ 3 \\ 4 \\ 5 \end{smallmatrix} $}
$\left(\begin{smallmatrix} 1 & 2 & 3 \\ 1 & 3 & 2 \end{smallmatrix}\right)$ & $=\tau$ & $\ytableaushort{12,3},\ytableaushort{12,3}$ & $\ytableaushort{12,3},\ytableaushort{12,3},\ydiagram{2,1}$ & \multirow{4}{*}{$\ydiagram{2,1},\ydiagram{2,1}$} & \multirow{  4}{*}{$\ytableaushort{+-,-+,+,-}$} \\
\cline{1-4}
\vphantom{$\begin{smallmatrix} 1 \\ 2 \\ 3 \\ 4 \\ 5 \end{smallmatrix} $}
$\left(\begin{smallmatrix} 1 & 2 & 3 \\ 2 & 1 & 3 \end{smallmatrix}\right)$ & $=\tau$ & $\ytableaushort{13,2},\ytableaushort{13,2}$ & $\ytableaushort{13,2},\ytableaushort{13,2},\ydiagram{2,1}$ & & \\
\cline{1-4} 
\vphantom{$\begin{smallmatrix} 1 \\ 2 \\ 3 \\ 4 \\ 5 \end{smallmatrix} $}
$\left(\begin{smallmatrix} 1 & 2 & 3 \\ 2 & 3 & 1
\end{smallmatrix}\right)$ & $=\tau$ & $\ytableaushort{13,2},\ytableaushort{12,3}$ &
$\ytableaushort{13,2},\ytableaushort{12,3},\ydiagram{2,1}$ & & \\ 
\cline{1-4}
\vphantom{$\begin{smallmatrix} 1 \\ 2 \\ 3 \\ 4 \\ 5 \end{smallmatrix} $}
$\left(\begin{smallmatrix} 1 & 2 & 3 \\ 3 & 1 & 2 \end{smallmatrix}\right)$ & $=\tau$ & $\ytableaushort{12,3},\ytableaushort{13,2}$ & $\ytableaushort{12,3},\ytableaushort{13,2},\ydiagram{2,1}$ & & \\
\hline
\vphantom{$\begin{smallmatrix} 1 \\ 2 \\ 3 \\ 4 \\ 5 \end{smallmatrix} $}
$\left(\begin{smallmatrix} 1 & 2 & 3 \\ 3 & 2 & 1 \end{smallmatrix}\right)$ & $=\tau$ & $\ytableaushort{1,2,3},\ytableaushort{1,2,3}$ & $\ytableaushort{1,2,3},\ytableaushort{1,2,3},\ydiagram{1,1,1}$ & $\ydiagram{1,1,1},\ydiagram{1,1,1}$ & $\ytableaushort{+,+,+,-,-,-}$ \\
\hline
\end{tabular}}
\caption{The correspondence $\tau\mapsto(T_1,T_2,\nu)$  for
$\mathfrak{T}_3$ ($\mathrm{rank}\,\tau=3$); the maps $\Phi$,
$\Xi_{\mathfrak{k}}$, and $\Xi_{\mathfrak{s}}$} \label{F1}
\end{figure}

\begin{figure}[p]
{\footnotesize
\begin{tabular}{|c|c|c|c|c|c|}
\hline
$\tau$ & $\sigma$ & $(\RSl(\sigma),\RSr(\sigma))$ & $(T_1,T_2,\nu)$ & $\Phi(\tau)=\Xik(\left(\begin{smallmatrix}\tau \\ 1_3\end{smallmatrix}\right))$ & $\Xis(\left(\begin{smallmatrix}\tau \\ 1_3\end{smallmatrix}\right))$ \\
\hline\hline 
$\left(\begin{smallmatrix} 1 & 2 & 3 \\ 0 & 1 & 2
\end{smallmatrix}\right)$ & $\left(\begin{smallmatrix} 2 & 3 \\ 1 &
2 \end{smallmatrix}\right)$ & $\ytableaushort{12},\ytableaushort{23}$ &
$\ytableaushort{123},\ytableaushort{123},\ydiagram{2}$
 & $\ydiagram{3},\ydiagram{3}$ & $\ytableaushort{-+-+,-+}$ \\
\hline
\vphantom{$\begin{smallmatrix} 1 \\ 2 \\ 3 \\ 4 \end{smallmatrix} $}
$\left(\begin{smallmatrix} 1 & 2 & 3 \\ 1 & 2 & 0 \end{smallmatrix}\right)$ & $\left(\begin{smallmatrix} 1 & 2 \\ 1 & 2 \end{smallmatrix}\right)$ & $\ytableaushort{12},\ytableaushort{12}$ & $\ytableaushort{123},\ytableaushort{12,3},\ydiagram{2}$ & \multirow{ 2}{*}{$\ydiagram{3},\ydiagram{2,1}$} & \multirow{ 2}{*}{$\ytableaushort{+-+,-+,-}$} \\
\cline{1-4}
\vphantom{$\begin{smallmatrix} 1 \\ 2 \\ 3 \\ 4 \\ 5 \end{smallmatrix} $}
$\left(\begin{smallmatrix} 1 & 2 & 3 \\ 1 & 0 & 2 \end{smallmatrix}\right)$ & $\left(\begin{smallmatrix} 1 & 3 \\ 1 & 2 \end{smallmatrix}\right)$ & $\ytableaushort{12},\ytableaushort{13}$ & $\ytableaushort{123},\ytableaushort{13,2},\ydiagram{2}$ &  &  \\
\hline 
\vphantom{$\begin{smallmatrix} 1 \\ 2 \\ 3 \\ 4 \\ 5 \end{smallmatrix} $}
$\left(\begin{smallmatrix} 1 & 2 & 3 \\ 0 & 2 & 3
\end{smallmatrix}\right)$ & $\left(\begin{smallmatrix} 2 & 3 \\ 2 &
3 \end{smallmatrix}\right)$ & $\ytableaushort{23},\ytableaushort{23}$ &
$\ytableaushort{13,2},\ytableaushort{123},\ydiagram{2}$ &
\multirow{2}{*}{$\ydiagram{2,1},\ydiagram{3}$} &
\multirow{ 2}{*}{$\ytableaushort{-+-,-+,+}$} \\
\cline{1-4} 
\vphantom{$\begin{smallmatrix} 1 \\ 2 \\ 3 \\ 4 \\ 5 \end{smallmatrix} $}
$\left(\begin{smallmatrix} 1 & 2 & 3 \\ 0 & 1 & 3
\end{smallmatrix}\right)$ & $\left(\begin{smallmatrix} 2 & 3 \\ 1 &
3 \end{smallmatrix}\right)$ & $\ytableaushort{13},\ytableaushort{23}$ &
$\ytableaushort{12,3},\ytableaushort{123},\ydiagram{2}$ & &
 \\
\hline
\vphantom{$\begin{smallmatrix} 1 \\ 2 \\ 3 \\ 4 \\ 5 \end{smallmatrix} $}
$\left(\begin{smallmatrix} 1 & 2 & 3 \\ 1 & 0 & 3 \end{smallmatrix}\right)$ & $\left(\begin{smallmatrix} 1 & 3 \\ 1 & 3 \end{smallmatrix}\right)$ & $\ytableaushort{13},\ytableaushort{13}$ & $\ytableaushort{12,3},\ytableaushort{13,2},\ydiagram{2}$ & \multirow{ 8}{*}{$\ydiagram{2,1},\ydiagram{2,1}$} & \multirow{4}{*}{$\ytableaushort{+-,-+,-+}$} \\
\cline{1-4}
\vphantom{$\begin{smallmatrix} 1 \\ 2 \\ 3 \\ 4 \\ 5 \end{smallmatrix} $}
$\left(\begin{smallmatrix} 1 & 2 & 3 \\ 1 & 3 & 0 \end{smallmatrix}\right)$ & $\left(\begin{smallmatrix} 1 & 2 \\ 1 & 3 \end{smallmatrix}\right)$ & $\ytableaushort{13},\ytableaushort{12}$ & $\ytableaushort{12,3},\ytableaushort{12,3},\ydiagram{2}$ & & \\
\cline{1-4} 
\vphantom{$\begin{smallmatrix} 1 \\ 2 \\ 3 \\ 4 \\ 5 \end{smallmatrix} $}
$\left(\begin{smallmatrix} 1 & 2 & 3 \\ 2 & 0 & 3
\end{smallmatrix}\right)$ & $\left(\begin{smallmatrix} 1 & 3 \\ 2 &
3 \end{smallmatrix}\right)$ & $\ytableaushort{23},\ytableaushort{13}$ &
$\ytableaushort{13,2},\ytableaushort{13,2},\ydiagram{2}$ & & \\
\cline{1-4}
\vphantom{$\begin{smallmatrix} 1 \\ 2 \\ 3 \\ 4 \\ 5 \end{smallmatrix} $}
$\left(\begin{smallmatrix} 1 & 2 & 3 \\ 2 & 3 & 0 \end{smallmatrix}\right)$ & $\left(\begin{smallmatrix} 1 & 2 \\ 2 & 3 \end{smallmatrix}\right)$ & $\ytableaushort{23},\ytableaushort{12}$ & $\ytableaushort{13,2},\ytableaushort{12,3},\ydiagram{2}$ & & \\
\cline{1-4} \cline{6-6}
\vphantom{$\begin{smallmatrix} 1 \\ 2 \\ 3 \\ 4 \\ 5 \end{smallmatrix} $}
$\left(\begin{smallmatrix} 1 & 2 & 3 \\ 2 & 0 & 1 \end{smallmatrix}\right)$ & $\left(\begin{smallmatrix} 1 & 3 \\ 2 & 1 \end{smallmatrix}\right)$ & $\ytableaushort{1,2},\ytableaushort{1,3}$ & $\ytableaushort{13,2},\ytableaushort{13,2},\ydiagram{1,1}$ & & \multirow{4}{*}{$\ytableaushort{-+,-+,+,-}$} \\
\cline{1-4}
\vphantom{$\begin{smallmatrix} 1 \\ 2 \\ 3 \\ 4 \\ 5 \end{smallmatrix} $}
$\left(\begin{smallmatrix} 1 & 2 & 3 \\ 0 & 3 & 1 \end{smallmatrix}\right)$ & $\left(\begin{smallmatrix} 2 & 3 \\ 3 & 1 \end{smallmatrix}\right)$ & $\ytableaushort{1,3},\ytableaushort{2,3}$ & $\ytableaushort{12,3},\ytableaushort{12,3},\ydiagram{1,1}$ & & \\
\cline{1-4} 
\vphantom{$\begin{smallmatrix} 1 \\ 2 \\ 3 \\ 4 \\ 5 \end{smallmatrix} $}
$\left(\begin{smallmatrix} 1 & 2 & 3 \\ 3 & 0 & 1
\end{smallmatrix}\right)$ & $\left(\begin{smallmatrix} 1  & 3 \\ 3
& 1 \end{smallmatrix}\right)$ & $\ytableaushort{1,3},\ytableaushort{1,3}$ &
$\ytableaushort{12,3},\ytableaushort{13,2},\ydiagram{1,1}$
 & & \\
\cline{1-4} 
\vphantom{$\begin{smallmatrix} 1 \\ 2 \\ 3 \\ 4 \\ 5 \end{smallmatrix} $}
$\left(\begin{smallmatrix} 1 & 2 & 3 \\ 0 & 2 & 1
\end{smallmatrix}\right)$ & $\left(\begin{smallmatrix} 2 & 3 \\ 2 &
1 \end{smallmatrix}\right)$ & $\ytableaushort{1,2},\ytableaushort{2,3}$ &
$\ytableaushort{13,2},\ytableaushort{12,3},\ydiagram{1,1}$
 & & \\
\hline 
\vphantom{$\begin{smallmatrix} 1 \\ 2 \\ 3 \\ 4 \\ 5 \\ 6 \end{smallmatrix} $}
$\left(\begin{smallmatrix} 1 & 2 & 3 \\ 0 & 3 & 2
\end{smallmatrix}\right)$ & $\left(\begin{smallmatrix} 2 & 3 \\ 3 &
2 \end{smallmatrix}\right)$ & $\ytableaushort{2,3},\ytableaushort{2,3}$ &
$\ytableaushort{1,2,3},\ytableaushort{12,3},\ydiagram{1,1}$ &
\multirow{2}{*}{$\ydiagram{1,1,1},\ydiagram{2,1}$} &
\multirow{5}{*}{$\ytableaushort{-+,+,+,-,-}$} \\
\cline{1-4}
\vphantom{$\begin{smallmatrix} 1 \\ 2 \\ 3 \\ 4 \\ 5 \end{smallmatrix} $}
$\left(\begin{smallmatrix} 1 & 2 & 3 \\ 3 & 0 & 2 \end{smallmatrix}\right)$ & $\left(\begin{smallmatrix} 1 & 3 \\ 3 & 2 \end{smallmatrix}\right)$ & $\ytableaushort{2,3},\ytableaushort{1,3}$ & $\ytableaushort{1,2,3},\ytableaushort{13,2},\ydiagram{1,1}$ & &  \\
\cline{1-5}
\vphantom{$\begin{smallmatrix} 1 \\ 2 \\ 3 \\ 4 \\ 5 \end{smallmatrix} $}
$\left(\begin{smallmatrix} 1 & 2 & 3 \\ 2 & 1 & 0 \end{smallmatrix}\right)$ & $\left(\begin{smallmatrix} 1 & 2 \\ 2 & 1 \end{smallmatrix}\right)$ & $\ytableaushort{1,2},\ytableaushort{1,2}$ & $\ytableaushort{13,2},\ytableaushort{1,2,3},\ydiagram{1,1}$ & \multirow{2}{*}{$\ydiagram{2,1},\ydiagram{1,1,1}$} &  \\
\cline{1-4}
\vphantom{$\begin{smallmatrix} 1 \\ 2 \\ 3 \\ 4 \\ 5 \end{smallmatrix} $}
$\left(\begin{smallmatrix} 1 & 2 & 3 \\ 3 & 1 & 0 \end{smallmatrix}\right)$ & $\left(\begin{smallmatrix} 1 & 2 \\ 3 & 1 \end{smallmatrix}\right)$ & $\ytableaushort{1,3},\ytableaushort{1,2}$ & $\ytableaushort{12,3},\ytableaushort{1,2,3},\ydiagram{1,1}$ & & \\
\cline{1-5} 
\vphantom{$\begin{smallmatrix} 1 \\ 2 \\ 3 \\ 4 \\ 5 \end{smallmatrix} $}
$\left(\begin{smallmatrix} 1 & 2 & 3 \\ 3 & 2 & 0
\end{smallmatrix}\right)$ & $\left(\begin{smallmatrix} 1 & 2 \\ 3 &
2 \end{smallmatrix}\right)$ & $\ytableaushort{2,3},\ytableaushort{1,2}$ &
$\ytableaushort{1,2,3},\ytableaushort{1,2,3},\ydiagram{1,1}$
 & $\ydiagram{1,1,1},\ydiagram{1,1,1}$ &  \\
\hline
\end{tabular}
} \caption{The correspondence $\tau\mapsto(T_1,T_2,\nu)$ for
$\mathfrak{T}_3$ ($\mathrm{rank}\,\tau=2$); the maps $\Phi$,
$\Xi_{\mathfrak{k}}$, and $\Xi_{\mathfrak{s}}$} \label{F2}
\end{figure}

\begin{figure}[p]
{\footnotesize
\begin{tabular}{|c|c|c|c|c|c|}
\hline $\tau$ & $\sigma$ & $(\RSl(\sigma),\RSr(\sigma))$ & $(T_1,T_2,\nu)$ & $\Phi(\tau)=\Xik(\left(\begin{smallmatrix}\tau \\ 1_3\end{smallmatrix}\right))$ & $\Xis(\left(\begin{smallmatrix}\tau \\ 1_3\end{smallmatrix}\right))$ \\
\hline\hline
\vphantom{$\begin{smallmatrix} 1 \\ 2 \\ 3 \\ 4 \\ 5 \end{smallmatrix} $}
$\left(\begin{smallmatrix} 1 & 2 & 3 \\ 0 & 1 & 0 \end{smallmatrix}\right)$ & $\left(\begin{smallmatrix} 2 \\ 1 \end{smallmatrix}\right)$ & $\ytableaushort{1},\ytableaushort{2}$ & $\ytableaushort{12,3},\ytableaushort{12,3},\ydiagram{1}$ & \multirow{4}{*}{$\ydiagram{2,1},\ydiagram{2,1}$} & \multirow{4}{*}{$\ytableaushort{-+,-+,-+}$} \\
\cline{1-4}
\vphantom{$\begin{smallmatrix} 1 \\ 2 \\ 3 \\ 4 \\ 5 \end{smallmatrix} $}
$\left(\begin{smallmatrix} 1 & 2 & 3 \\ 0 & 2 & 0 \end{smallmatrix}\right)$ & $\left(\begin{smallmatrix} 2 \\ 2 \end{smallmatrix}\right)$ & $\ytableaushort{2},\ytableaushort{2}$ & $\ytableaushort{13,2},\ytableaushort{12,3},\ydiagram{1}$ & & \\
\cline{1-4}
\vphantom{$\begin{smallmatrix} 1 \\ 2 \\ 3 \\ 4 \\ 5 \end{smallmatrix} $}
$\left(\begin{smallmatrix} 1 & 2 & 3 \\ 0 & 0 & 1 \end{smallmatrix}\right)$ & $\left(\begin{smallmatrix} 3 \\ 1 \end{smallmatrix}\right)$ & $\ytableaushort{1},\ytableaushort{3}$ & $\ytableaushort{12,3},\ytableaushort{13,2},\ydiagram{1}$ & & \\
\cline{1-4}
\vphantom{$\begin{smallmatrix} 1 \\ 2 \\ 3 \\ 4 \\ 5 \end{smallmatrix} $}
$\left(\begin{smallmatrix} 1 & 2 & 3 \\ 0 & 0 & 2 \end{smallmatrix}\right)$ & $\left(\begin{smallmatrix} 3 \\ 2 \end{smallmatrix}\right)$ & $\ytableaushort{2},\ytableaushort{3}$ & $\ytableaushort{13,2},\ytableaushort{13,2},\ydiagram{1}$ & & \\
\hline
\vphantom{$\begin{smallmatrix} 1 \\ 2 \\ 3 \\ 4 \\ 5 \\ 6 \end{smallmatrix} $}
$\left(\begin{smallmatrix} 1 & 2 & 3 \\ 1 & 0 & 0 \end{smallmatrix}\right)$ & $\left(\begin{smallmatrix} 1 \\ 1 \end{smallmatrix}\right)$ & $\ytableaushort{1},\ytableaushort{1}$ & $\ytableaushort{12,3},\ytableaushort{1,2,3},\ydiagram{1}$ & \multirow{2}{*}{$\ydiagram{2,1},\ydiagram{1,1,1}$} & \multirow{5}{*}{$\ytableaushort{-+,-+,+,-}$} \\
\cline{1-4}
\vphantom{$\begin{smallmatrix} 1 \\ 2 \\ 3 \\ 4 \\ 5 \\ 6 \end{smallmatrix} $}
$\left(\begin{smallmatrix} 1 & 2 & 3 \\ 2 & 0 & 0 \end{smallmatrix}\right)$ & $\left(\begin{smallmatrix} 1 \\ 2 \end{smallmatrix}\right)$ & $\ytableaushort{2},\ytableaushort{1}$ & $\ytableaushort{13,2},\ytableaushort{1,2,3},\ydiagram{1}$ & & \\
\cline{1-5}
\vphantom{$\begin{smallmatrix} 1 \\ 2 \\ 3 \\ 4 \\ 5 \\ 6 \end{smallmatrix} $}
$\left(\begin{smallmatrix} 1 & 2 & 3 \\ 0 & 3 & 0 \end{smallmatrix}\right)$ & $\left(\begin{smallmatrix} 2 \\ 3 \end{smallmatrix}\right)$ & $\ytableaushort{3},\ytableaushort{2}$ & $\ytableaushort{1,2,3},\ytableaushort{12,3},\ydiagram{1}$ & \multirow{2}{*}{$\ydiagram{1,1,1},\ydiagram{2,1}$} & \\
\cline{1-4}
\vphantom{$\begin{smallmatrix} 1 \\ 2 \\ 3 \\ 4 \\ 5 \\ 6 \end{smallmatrix} $}
$\left(\begin{smallmatrix} 1 & 2 & 3 \\ 0 & 0 & 3
\end{smallmatrix}\right)$ & $\left(\begin{smallmatrix} 3 \\ 3
\end{smallmatrix}\right)$ & $\ytableaushort{3},\ytableaushort{3}$ &
$\ytableaushort{1,2,3},\ytableaushort{13,2},\ydiagram{1}$  & & \\
\cline{1-5}
\vphantom{$\begin{smallmatrix} 1 \\ 2 \\ 3 \\ 4 \\ 5 \\ 6 \end{smallmatrix} $}
$\left(\begin{smallmatrix} 1 & 2 & 3 \\ 3 & 0 & 0 \end{smallmatrix}\right)$ & $\left(\begin{smallmatrix} 1 \\ 3 \end{smallmatrix}\right)$ & $\ytableaushort{3},\ytableaushort{1}$ & $\ytableaushort{1,2,3},\ytableaushort{1,2,3},\ydiagram{1}$ & $\ydiagram{1,1,1},\ydiagram{1,1,1}$ & \\
\hline \hline
\vphantom{$\begin{smallmatrix} 1 \\ 2 \\ 3 \\ 4 \\ 5 \\ 6 \end{smallmatrix} $}
$\left(\begin{smallmatrix} 1 & 2 & 3 \\ 0 & 0 & 0 \end{smallmatrix}\right)$ & $\emptyset$ & $\emptyset,\emptyset$ & $\ytableaushort{1,2,3},\ytableaushort{1,2,3},\emptyset$ & $\ydiagram{1,1,1},\ydiagram{1,1,1}$ & $\ytableaushort{-+,-+,-+}$ \\
\hline
\end{tabular}}
\caption{The correspondence $\tau\mapsto(T_1,T_2,\nu)$ for
$\mathfrak{T}_3$ ($\mathrm{rank}\,\tau\leq 1$); the maps $\Phi$,
$\Xi_{\mathfrak{k}}$, and $\Xi_{\mathfrak{s}}$} \label{F3}
\end{figure}
\end{example}
\ytableausetup{nosmalltableaux}

\begin{remark}
\label{R:Travkin}
There might be a close relationship between the correspondence $\tau\mapsto(T_1,T_2,\nu)$ of Theorem \ref{thm:3.9} and Travkin's mirabolic Robinson-Schensted-Knuth correspondence \cite{Travkin}.
Let us explain this in more detail.

In \cite{Travkin} the diagonal action of $\GL_n$ on the variety $\GL_n/B \times \GL_n/B \times \C^n$
is considered. The orbits are parametrized by the so-called marked permutations, i.e., the pairs $(w,\beta)$ formed by a permutation
$w\in\permutationsof{n}$ and a subset $\beta\subset [n]$ satisfying
$$
\forall i,j\in[n],\ (i\notin\beta\ \mbox{and}\ j\in\beta)\ \Longrightarrow\ (i>j\ \mbox{or}\ w(i)>w(j));
$$
namely $(w,\beta)$ is mapped to the $\GL_n$-orbit of the triple $(B,wB,\sum_{i\in\beta}e_i)$, where $(e_1,\ldots,e_n)$ is the standard basis of $\C^n$.
This is a variant of the parametrization of the $\GL_n$-orbits on $\GL_n/B\times\GL_n/B\times\mathbb{P}(\C^n)$ 
due to Magyar-Weyman-Zelevinsky \cite{Magyar-Weyman-Zelevinsky}.

By using the conormal variety for $\GL_n/B \times \GL_n/B \times \C^n$ and 
the enhanced nilpotent cone $\mathcal{N}_{\gl_n}\times\C^n$ (due to Achar and Henderson \cite{Achar-Henderson}),
Travkin establishes a bijection between the set of marked permutations and the same set of triples as in Theorem \ref{thm:3.9}. He describes this bijection explicitly and call it the mirabolic Robinson-Schensted-Knuth correspondence.

There is a natural bijection between marked permutations $(w,\beta)$ and partial permutations
$\tau\in\ppermutationsof{n}$,
which is defined by $(w,\beta)\mapsto \tau:=w|_{[n]\setminus\beta_\mathrm{max}}$ where
$$
\beta_\mathrm{max}:=\{i\in \beta:(j>i\ \mbox{and}\ j\in\beta)\ \Rightarrow\ w(j)<w(i) \}.
$$
Therefore, Travkin's correspondence can also be expressed as an explicit bijection between partial permutations $\tau\in\ppermutationsof{n}$ and triples $(T_1,T_2,\nu)$ as in Theorem \ref{thm:3.9}.
This bijection seems to be quite different from ours, 
thus giving rise to a non-identical map $\ppermutationsof{n}\to\ppermutationsof{n}$ on partial permutations.
We have no indication whatsoever on a direct, combinatorial description nor on a possible geometric interpretation of this map.
\end{remark}

We mention the following representation theoretic interpretation of the fibers of the map $\Phi$ described in Theorem \ref{T2}.
Let $\rho_\lambda^{(n)}$ denote the irreducible representation of $\mathfrak{S}_n$ corresponding to the partition $\lambda\in\partitionsof{n}$. In the next statement, the notation $\boxtimes$ stands for the outer tensor product.

\begin{corollary}
\label{C3.2}
For every pair of partitions $(\lambda,\mu)\in\partitionsof{n}{\times}\partitionsof{n}$, the number of elements in the fiber $\Phi^{-1}((\lambda,\mu))$ is equal to
\[\sum_{r=0}^n\Big(\mbox{\rm multiplicity of $\rho_\lambda^{(n)}\boxtimes\rho_\mu^{(n)}$ in $\mathrm{Ind}_{\mathfrak{S}_r^2\times\mathfrak{S}_{n-r}^2}^{\mathfrak{S}_n^2}(\C\mathfrak{S}_r\boxtimes\sgn^{\otimes 2})$}\Big)\cdot\dim(\rho_\lambda^{(n)}\boxtimes\rho_\mu^{(n)})\]
where $\C\mathfrak{S}_r$ denotes the regular representation
of $\mathfrak{S}_r^2:=\mathfrak{S}_r\times\mathfrak{S}_r$ (defined
by left and right multiplication) and $\sgn$ denotes the
signature representation of
$\mathfrak{S}_{n-r}$.
\end{corollary}

\begin{proof} We have $\C\mathfrak{S}_r=\bigoplus_{\nu\in\partitionsof{r}}\rho_\nu^{(r)}\boxtimes(\rho_\nu^{(r)})^*$
and
$\sgn^{\otimes 2}=\rho^{(n-r)}_{(1^{n-r})}\boxtimes(\rho^{(n-r)}_{(1^{n-r})})^*$.
By the Pieri formula 
(a special case of the Littlewood-Richardson rule; see the formula (5) in \S 2.2 and Corollary 2 in \S7.3 of Fulton's book \cite{Fulton}), 
we have
\[\mathrm{Ind}_{\mathfrak{S}_r\times\mathfrak{S}_{n-r}}^{\mathfrak{S}_n}(\rho_\nu^{(r)}\boxtimes\rho^{(n-r)}_{(1^{n-r})})
=\bigoplus_{\substack{\lambda\in\partitionsof{n} \mbox{\scriptsize \
s.t.} \\ \mbox{\scriptsize $\lambda\setminus\nu$ is column
strip}}}\rho_\lambda^{(n)}\] hence the multiplicity
$m_r(\lambda,\mu)$ in the decomposition
\[\mathrm{Ind}_{\mathfrak{S}_r^2\times\mathfrak{S}_{n-r}^2}^{\mathfrak{S}_n^2}(\C\mathfrak{S}_r\boxtimes\sgn^{\otimes 2})
=\bigoplus_{(\lambda,\mu)\in\partitionsof{n}^2}(\rho_\lambda^{(n)}\boxtimes\rho_\mu^{(n)})^{\oplus m_r(\lambda,\mu)}\]
coincides with the number of partitions $\nu\in\partitionsof{r}$
(subdiagrams of $\lambda$ and $\mu$)
such that $\lambda\setminus\nu$ and $\mu\setminus\nu$ are column
strips. We also know that
$\dim\rho_\lambda^{(n)}=|\mathrm{STab}(\lambda)|$ 
since $ \STab(\lambda) $ gives a basis of $ \rho_\lambda^{(n)} $ via the construction of Specht module.  
Hence, 
if we put 
\begin{equation*}
\Triplets{r}(\lambda, \mu) := 
\{ (T_1,T_2,\nu)\in\mathrm{STab}(\lambda){\times}\mathrm{STab}(\mu){\times}\partitionsof{r}:
\lambda\setminus\nu,\ \mu\setminus\nu\mbox{ are column strips}\} , 
\end{equation*}
we obtain 
\begin{equation*}
\sum_{r=0}^nm_r(\lambda,\mu)\dim\rho_\lambda^{(n)}\boxtimes\rho_\mu^{(n)}
= \Bigl|\bigcup_{r=0}^n \Triplets{r}(\lambda, \mu) \Bigr| \\
= |\Phi^{-1}((\lambda,\mu))|
\end{equation*}
where Theorems \ref{T2} and \ref{thm:3.9} are used.\ The proof is complete.
\end{proof}

The corollary only tells that 
the number of partial permutations and the dimension of the $ \mathfrak{S}_n^2 $-module specified above exactly match.  
However, in the classical Steinberg theory, the corresponding representation is 
the regular representation of the symmetric group, and 
it coincides with the Springer representation on each fiber.
Thus we propose the following conjecture.

\begin{conjecture}[Generalization of the {Springer representation}]\label{conjecture:gen.Springer.representation}
There exists a 
geometric way to construct an action of the group $\permutationsof{n}\times\permutationsof{n}$
on the top Borel-Moore homology space of 
$ H_{\mathrm{top}}^{\mathrm{BM}}(\conormalvar_{\Mat_n}) $, 
identifying the irreducible components of $ \conormalvar_{\Mat_n} $ with a basis of the 
induced representation 
\begin{equation*}
\bigoplus\nolimits_{r = 0}^n \; \Ind_{\mathfrak{S}_r^2 \times \mathfrak{S}_{n -r}^2}^{\mathfrak{S}_n^2} \!
\Bigl(\C\mathfrak{S}_r \boxtimes \sgn^{\otimes 2} \Bigr).
\end{equation*}
Moreover the components of $\conormalvar_{\Mat_n}$ parametrized by the elements in the fibers $ \Triplets{r}(\lambda, \mu) $ of $ \calorbit_{\lambda} \times \calorbit_{\mu} $ 
give a natural basis of $ \rho_{\lambda}^{(n)} \boxtimes \rho_{\mu}^{(n)} $ with multiplicities.
\end{conjecture}

\begin{remark}\label{remark:dim.identity.HT}
In \cite{Henderson-Trapa}, Henderson and Trapa used a similar equality between 
the number of orbits and the dimension of certain induced representation 
(see Proposition 3.1 in \cite{Henderson-Trapa}).  
However, it is still open to realize the representations geometrically.  
\end{remark}

\part{The image of the symmetrized and exotic moment maps for the double flag variety of type~AIII}

In this final part, we establish
our main results outlined in Sections \ref{section-1.3}--\ref{section-1.4}.

\section{Parametrization of the $K$-orbits in the double flag variety $\mathfrak{X}$}\label{section-8}

In the rest of the paper, we consider the setting given in Sections \ref{section-1.3} and \ref{section-2.2}.
In particular we have a polarized vector space 
\[V:=\C^{2n}=V^+\oplus V^-\quad \mbox{where}\quad V^+:=\C^n\times\{0\}^n,\ V^-:=\{0\}^n\times\C^n, \]
and consider the group $G=\GL_{2n}=\GL(V)$.   
The symmetric subgroup
\[K=\{g\in\GLnnC:g(V^+)=V^+,\ g(V^-)=V^-\}\cong\GLnC\times\GLnC\]
is the stabilizer of the polarization above.  
Finally, we consider the double flag variety
\[\mathfrak{X}=K/B_K\times G/P_\mathrm{S}\]
where
\begin{itemize}
\item
$B_K\subset K$ is a Borel subgroup, and there is no loss of generality in assuming that
$B_K=\left\{\left(\begin{array}{cc} b_1 & 0 \\ 0 & b_2
\end{array}\right):b_1,b_2\in B\right\}\cong
B\times B$ with $B=B_n^+\subset\GL_n$ the Borel subgroup of upper triangular matrices;
\item $P_\mathrm{S}:=\left\{\left(\begin{array}{cc} a & c \\ 0 & b
\end{array}\right):a,b\in \GLnC,\ c\in \MatnC\right\}\subset G$
is a Siegel parabolic subgroup.
\end{itemize}
Hence $\mathfrak{X}$ can be identified with the
direct product
\[\mathfrak{X}=K/B_K\times\mathrm{Gr}_n(\C^{2n})= \Flag(V^+)\times \Flag(V^-)\times\mathrm{Gr}_n(\C^{2n})\]
where $\Flag(V^\pm)$ denotes the variety of complete flags of
the subspace $V^\pm$ while $\mathrm{Gr}_n(\C^{2n})$ stands
for the Grassmann variety of $n$-dimensional subspaces of
$\C^{2n}$.

In the present section, we describe the parametrization of the $K$-orbits in the double flag variety $\mathfrak{X}$. 
It is easy to see that there is a bijection
\[
\begin{array}{c}
\mathrm{Gr}_n(\C^{2n})/B_K \xrightarrow{\;\;\sim\;\;} \mathfrak{X}/K\cong(K/B_K\times\mathrm{Gr}_n(\C^{2n}))/K,\ 
\\[1ex]
V\,\mathrm{mod}\,B_K\mapsto(B_K,V)\,\mathrm{mod}\,K .
\end{array}
\]
Thus we are reduced to parametrize the set of $B_K$-orbits in 
$\mathrm{Gr}_n(\C^{2n})$, which is achieved 
in the following theorem.  
Here again the partial permutations $ \ppermutationsof{n} $ studied in Section \ref{section-orbits} appear.  
We identify $ \tau \in \ppermutationsof{n} $ with the corresponding matrix in $ \Mat_n $.

\begin{theorem}\label{T3}
Let $(\mathfrak{T}_{n}^2)'$ denote the set of $(2n)\times n$ matrices of the form $\left(\begin{array}{c} \tau_1 \\ \tau_2 \end{array}\right)$,
with $\tau_1,\tau_2\in\mathfrak{T}_n$, and which are of rank $n$ 
(i.e., $\tau_1^{-1}(0)\cap\tau_2^{-1}(0)=\emptyset$). 
The group $\mathfrak{S}_n$ acts on the set $(\mathfrak{T}_n^2)'$ from the right by the multiplication of the permutation matrices. 
For $ \omega \in (\ppermutationsof{n})' $, 
let us denote by $ V_{\omega} := \Im \omega $ the $ n $-dimensional subspace generated by the columns of the matrix $ \omega $.
Then, the map
\[
(\mathfrak{T}_n^2)' \xrightarrow{\quad} \mathrm{Gr}_n(\C^{2n}), \qquad 
\omega=\left(\begin{array}{c} \tau_1 \\ \tau_2 \end{array}\right) 
\longmapsto V_\omega \]
induces a bijection
\[(\mathfrak{T}_n^2)'/\mathfrak{S}_n \xrightarrow{\;\;\sim\;\;} \mathrm{Gr}_n(\C^{2n})/B_K\cong\mathfrak{X}/K.\]
\end{theorem}

To prove the theorem, 
we use the following lemma.

\begin{lemma}
\label{L2} Let $B'\subset\GLnC$ be a Borel
subgroup which contains the standard torus. For any partial
permutation $\tau\in\mathfrak{T}_n$, there is a permutation
$w\in\mathfrak{S}_n$ such that
\[\tau wB'\subset B'\tau w.\]
\end{lemma}

\begin{proof}
There is no loss of generality in assuming that
$B'\subset\GLnC$ is the subgroup of \emph{lower triangular} matrices. 
Given $\tau \in \ppermutationsof{n}$, by
permuting the columns of $\tau$, we can find 
$w\in\mathfrak{S}_n$ such that
\[\tau w=\left(\begin{array}{cccccc}
1 & \cdots & r & r+1 & \cdots & n \\
i_1 & \cdots & i_r & 0 & \cdots & 0
\end{array}\right)\quad\mbox{with}\quad i_1<\cdots<i_r,\]
where $r=\mathrm{rank}\,\tau$. The group $B'$ is generated by the torus $T$ of diagonal matrices and the elementary matrices of the form $u_j(\alpha):=1_n+\alpha e_{j+1,j}$ for $j\in\{1,\ldots,n-1\}$. When $t=\mathrm{diag}(t_1,\ldots,t_n)$ is a diagonal matrix, we have
\[\tau wt=\mathrm{diag}(t'_1,\ldots,t'_n)\tau w\in B'\tau w\quad\mbox{where $t'_{i_k}=t_k$ and $t'_\ell=1$ for $\ell\notin\{i_1,\ldots,i_r\}$.}\]
When $u=u_j(\alpha)$ we have
\[\tau wu=\left\{\begin{array}{ll}
\tau w & \mbox{if $r\leq j\leq n-1$,} \\
\tau w+\alpha e_{i_{j+1},j}=(1_n+\alpha e_{i_{j+1},i_j})\tau w & \mbox{if $1\leq j\leq r-1$,}
\end{array}\right.\]
hence $\tau wu\in B'\tau w$ in both cases. Altogether we conclude that $\tau wB'\subset B'\tau w$.\ The proof of the lemma is complete.
\end{proof}

\begin{proof}[Proof of Theorem \ref{T3}]
First we show that $\mathrm{Gr}_n(\C^{2n})$ is the union of the
$B_K$-orbits through the points $V_\omega \in\mathrm{Gr}_n(\C^{2n})$ 
for $\omega\in(\mathfrak{T}_n^2)'$. 
Any point in $\mathrm{Gr}_n(\C^{2n})$ is of the form
$ V_a:=\Im a$ for a certain $(2n)\times n$ matrix $a$ of rank $n$.
We write $a\sim a'$ whenever $V_a$ and $ V_{a'}$ belong to the same
$B_K$-orbit. Let us consider $a=\left(\begin{array}{c} a_1 \\
a_2\end{array}\right)\in\Mat_{2n,n}(\C)$ of rank $n$.
Recall that $B_K=B\times B$ where
$B\subset\GLnC$ is the Borel subgroup of upper triangular matrices. 
By Proposition \ref{P3} we can write
$a_1=b_1\tau_1 b_2$ for some $(b_1,b_2)\in B\times B$ and some
partial permutation $\tau_1\in\mathfrak{T}_n$. Furthermore, by Lemma
\ref{L2}, there is a permutation $w\in\mathfrak{S}_n$ such that
$\tau_1wB\subset B\tau_1w$. Setting $a_2':=a_2b_2^{-1}w$ we have
\[a=\left(\begin{array}{c} a_1 \\ a_2\end{array}\right)=\left(\begin{array}{c} b_1\tau_1w \\ a'_2\end{array}\right)w^{-1}b_2\sim\left(\begin{array}{c} \tau_1w \\ a'_2\end{array}\right).\]
Invoking again Proposition \ref{P3}, we find a pair $(b'_1,b'_2)\in
B\times B$ and a partial permutation $\tau_2\in\mathfrak{T}_n$
such that $a'_2=b'_2\tau_2b'_1$. Moreover since $\tau_1wB\subset
B\tau_1w$, we can find $b''_1\in B$ such that
$\tau_1wb_1'^{-1}=b''_1\tau_1w$. Whence
\[a\sim\left(\begin{array}{c} \tau_1w \\ a'_2\end{array}\right)=\left(\begin{array}{c} \tau_1w \\ b'_2\tau_2b'_1\end{array}\right)\sim\left(\begin{array}{c} \tau_1wb_1'^{-1} \\ b'_2\tau_2\end{array}\right)=\left(\begin{array}{c} b''_1\tau_1w \\ b'_2\tau_2\end{array}\right)\sim\left(\begin{array}{c} \tau_1w \\ \tau_2\end{array}\right)\in(\mathfrak{T}_n^2)'.\]
Hence all the $B_K$-orbits have representatives 
of the form $V_\omega$ for some $ \omega\in (\mathfrak{T}_n^2)'$.

Clearly, if $\omega=\omega'w$ with $w\in\mathfrak{S}_n$, then we have $V_\omega=V_{\omega'}$. Hence the map 
\begin{equation*}
(\mathfrak{T}_n^2)'/\mathfrak{S}_n\to\mathrm{Gr}_n(\C^{2n})/B_K, \qquad\omega\mapsto B_K\cdot V_\omega
\end{equation*}
is well defined and surjective. It remains to show that this map is injective. 

Let $(\varepsilon^\pm_1,\ldots,\varepsilon^\pm_n)$ be the standard basis of $V^\pm$ and 
set $V_i^{\pm}=\langle\varepsilon_1^{\pm},\ldots,\varepsilon_i^{\pm}\rangle_\C$ 
for $ i = 0, 1, \dots, n $.  
Thus $ B_K $ is the stabilizer of the pair of complete flags 
$(V_0^{\pm},V_1^{\pm},\ldots,V_n^{\pm})\in\Flag(V^{\pm})$.
Note that, for $\omega,\omega'\in(\mathfrak{T}_n^2)'$, we have
\begin{eqnarray}\label{10new} 
V_\omega=V_{\omega'}\ \mathrm{mod}\,B_K & \implies & \dim (V_i^++V_j^-)\cap V_\omega=\dim (V_i^++V_j^-)\cap V_{\omega'}, \quad \forall i,j.
\end{eqnarray}
Since $\mathrm{rank}\,\omega=n$ by assumption, 
every column of $\omega$ is nonzero.\ In fact every column of $\omega$
has at most two nonzero coefficients and is of the form $\varepsilon_i^+$, $\varepsilon_i^++\varepsilon_j^-$, or $\varepsilon_j^-$ for some $i,j$ 
(regarding $\varepsilon_i^+,\varepsilon_j^-$ as $2n$-sized column matrices). 
According to these three cases, we have
\begin{itemize}
\item $\varepsilon_i^+$ occurs as a column of $\omega$ if and only if $\dim V_i^+\cap V_\omega>\dim V_{i-1}^+\cap V_\omega$;
\item $\varepsilon_j^-$ occurs as a column of $\omega$ if and only if $\dim V_j^-\cap V_\omega>\dim V_{j-1}^-\cap V_\omega$;
\item $\varepsilon_i^++\varepsilon_j^-$ occurs as a column of $\omega$ if and only if 
\\
\hfil
$
\begin{aligned}
\dim (V_i^++V_j^-)\cap V_\omega >\dim (V_{i-1}^++V_j^-)\cap V_\omega 
&=\dim (V_i^++V_{j-1}^-)\cap V_\omega\\ &=\dim (V_{i-1}^++V_{j-1}^-)\cap V_\omega .
\end{aligned}
$
\hfil
\end{itemize}
This observation combined with (\ref{10new}) tells us that 
\begin{eqnarray*}
V_\omega=V_{\omega'}\ \mathrm{mod}\,B_K & \implies & \mbox{$\omega$ and $\omega'$ have the same columns (up to the order)} \\
 & \implies & \omega=\omega'w\ \mbox{ for some $w\in\mathfrak{S}_n$.}
\end{eqnarray*}
The proof of the theorem is now complete.\end{proof}

\section{Symmetrized moment map $\mathfrak{X}/K\to\mathcal{N}_\mathfrak{k}/K$}\label{section-9}

The purpose of this section is to compute the symmetrized moment map $\mathfrak{X}/K\to\mathcal{N}_\mathfrak{k}/K$ 
of (\ref{3}) in the case of the symmetric pair $ (G, K) = (\GLnnC,\GLnC\times\GLnC)$ under consideration 
(see Section \ref{section-1.3}).
We use the same notation as in Sections \ref{section-2.2} and \ref{section-8}, in particular 
$G=\GL_{2n}$ and
\[K=\Big\{\left(\begin{array}{cc} a & 0 \\ 0 & b \end{array}\right):a,b\in\GLnC\Big\}=\{g\in G:g(V^\pm)=V^\pm\}\cong\GLnC\times\GLnC\]
where $V^+=\C^n\times\{0\}^n$ and $V^-=\{0\}^n\times\C^n$, so that $V^+\oplus V^-=\C^{2n}$. 
Their Lie algebras are denoted by $\lie{g}=\lie{gl}_{2n}$ and 
\[\mathfrak{k}=\mathrm{Lie}(K)=\Big\{\left(\begin{array}{cc} \alpha & 0 \\ 0 & \beta \end{array}\right):\alpha,\beta\in\MatnC\Big\}\cong\lie{gl}_n\times\lie{gl}_n.\]
Recall the projection 
\begin{equation*}
(\cdot)^\theta : \lie{g} \longrightarrow \mathfrak{k}, \qquad
x=\begin{pmatrix}x_1 & x_2 \\ x_3 & x_4\end{pmatrix} 
\longmapsto x^\theta=\begin{pmatrix}x_1 & 0 \\ 0 & x_4\end{pmatrix}
\end{equation*}
along the Cartan decomposition.  
Since the nilpotent cone $\mathcal{N}_\mathfrak{k}$ is the direct product $\mathcal{N}_{\lie{gl}_n}\times\mathcal{N}_{\lie{gl}_n}$, 
the nilpotent $K$-orbits in $\mathcal{N}_\mathfrak{k}$ are of the form $\mathcal{O}_\lambda\times\mathcal{O}_\mu$ 
with a pair of partitions $(\lambda,\mu) \in \partitionsof{n}^2 $:
\[\mathcal{N}_\mathfrak{k}/K=\{\mathcal{O}_\lambda\times\mathcal{O}_\mu:\lambda,\mu\in\partitionsof{n}\}\cong \partitionsof{n}\times \partitionsof{n}.\]
Our double flag variety is identified with 
\[\mathfrak{X}=K/B_K\times \mathrm{Gr}_n(\C^{2n})\cong\Flag(V^+)\times\Flag(V^-)\times\mathrm{Gr}_n(\C^{2n}),\]
where $B_K=B\times B$ is the Borel subgroup of $K$
corresponding to the subgroup $B\subset \GL_n$ of upper-triangular matrices.
From Theorem \ref{T3} we have a parametrization of the set of $ K $-orbits $\mathfrak{X}/K$:
\begin{equation}
\label{recall-orbitsofX}
(\mathfrak{T}_n^2)'/\mathfrak{S}_n 
\xrightarrow{\;\;\sim\;\;} \mathfrak{X}/K, \qquad 
\omega \longmapsto \orbitofX_\omega:=K\cdot(B_K,V_\omega)
\end{equation}
where $(\mathfrak{T}_n^2)'$ is the set of full rank matrices of 
the form $\omega=\left(\begin{array}{c} \tau_1 \\ \tau_2
\end{array}\right) \; (\tau_1,\tau_2\in\mathfrak{T}_n) $, and
$V_\omega = \Im \omega\in\mathrm{Gr}_n(\C^{2n})$.

As before, we identify the flag variety $K/B_K$ (resp. $\mathrm{Gr}_n(\C^{2n})\cong G/P_\mathrm{S}$) with the collection of all Borel subalgebras $\mathfrak{b}'_K\subset \mathfrak{k}$ (resp. parabolic subalgebras $\mathfrak{p}'\subset\lie{g}$ conjugate to $\mathrm{Lie}(P_\mathrm{S})$). With this identification, $\mathfrak{Z}_\omega$ is regarded as the $K$-orbit through the pair $(\mathfrak{b}\times\mathfrak{b},\mathfrak{p}_\omega)$ where $\mathfrak{b}:=\mathrm{Lie}(B)\subset\lie{gl}_n$ is the subalgebra of upper-triangular matrices, and $\mathfrak{p}_\omega:=\{x\in\lie{gl}_{2n}:x(V_\omega)\subset V_\omega\}$. We further denote $\mathfrak{n}:=\mathfrak{nil}(\mathfrak{b})$ and $\mathfrak{u}_\omega:=\mathfrak{nil}(\mathfrak{p}_\omega)=\{x\in\lie{gl}_{2n}:\Im x\subset V_\omega\subset\ker x\}$, 
the nilradicals of $ \lie{b} $ and $ \lie{p}_{\omega} $ respectively.

Recall the conormal variety 
%
\begin{equation}
\label{5.1} \mathcal{Y}=\{(\mathfrak{b}_K',\mathfrak{p}',x)\in
K/B_K\times G/P_\mathrm{S}\times
\lie{g}:x^\theta\in\mathfrak{nil}(\mathfrak{b}_{K}'),\
x\in\mathfrak{nil}(\mathfrak{p}')\} , 
\end{equation} 
which is a union of conormal bundles 
\begin{eqnarray}\label{11}
T_{\orbitofX_\omega}^*\mathfrak{X} & = & \{(\mathfrak{b}'_K,\mathfrak{p}',x)\in\mathcal{Y}:(\mathfrak{b}'_K,\mathfrak{p}')\in\orbitofX_\omega\} \\
 & = & K\cdot\{(\mathfrak{b}\times\mathfrak{b},\mathfrak{p}_\omega,x)\in\mathfrak{X}\times\MatnC:x\in\mathfrak{u}_\omega,\ x^\theta\in\lie{n}{\times}\lie{n}\} \nonumber
\end{eqnarray}
over the $ K $-orbits $\orbitofX_\omega$ (see Section \ref{section-2.1.1}), 
and the map
$\pi_\mathfrak{k}:\mathcal{Y}\to\mathfrak{k} $ defined by 
$ \pi_\mathfrak{k}(\mathfrak{b}_K',\mathfrak{p}',x) = x^\theta$.
Then the symmetrized moment map $\mathfrak{X}/K\to\mathcal{N}_\mathfrak{k}/K$ in (\ref{3}) induces a map 
between parameter sets of orbits:
\[\Xik:(\mathfrak{T}_n^2)' \xrightarrow{\quad} \partitionsof{n}{\times}\partitionsof{n},\qquad 
\omega\longmapsto (\lambda,\mu) \]
where $(\lambda,\mu)$ is the pair of partitions such that
$\overline{\pi_\mathfrak{k}(T_{\orbitofX_\omega}^*\mathfrak{X})}=\overline{\mathcal{O}_\lambda}\times\overline{\mathcal{O}_\mu}$. 
Though it is difficult to give a combinatorial algorithm to describe $ \Xik $ for all $ \omega $'s, 
we have an efficient algorithm for ``generic'' ones using 
Theorem \ref{T2}.
Namely we prove the following theorem.

\begin{theorem}\label{T4}
Whenever $\omega$ is of the form $\omega=\left(\begin{array}{c} \tau \\ 1_n \end{array}\right)$ with a partial permutation $\tau\in\mathfrak{T}_n$,
we have
$\Xik(\omega)=\Phi(\tau)$ where $\Phi:\mathfrak{T}_n\to\partitionsof{n}{\times}\partitionsof{n}$ is 
the generalized Steinberg map described in Theorem \ref{T2}.
\end{theorem}

To prove the theorem, we use 
the following characterization of the map $\Xik$
(which immediately follows from (\ref{11}) and the definition of the map $\Xik$).

\begin{lemma} \label{L3}
If $\Xik(\omega)=(\lambda,\mu)$, then $\mathcal{O}_\lambda\times\mathcal{O}_\mu$ is characterized as the nilpotent orbit of $\mathcal{N}_\mathfrak{k}$ which intersects the subspace
\[\conormbImk{\omega}:=\{x^\theta:x\in\mathfrak{u}_\omega\}\cap(\lie{n}\times\lie{n})\]
along a dense open subset.
\end{lemma}

\begin{proof}[Proof of Theorem \ref{T4}]
Let us characterize the elements of the space $\conormbImk{\omega}$.  
Put $ x = \mattwo{x_1}{x_2}{x_3}{x_4} \in\Mat_{2n} $, 
and note that $ x^\theta = \mattwo{x_1}{0}{0}{x_4} $.
Then $ x $ belongs to the nilpotent radical $ \mathfrak{u}_\omega$ if and only if $\Im x\subset\Im \omega\subset \ker x$, hence
\begin{align}\label{12}
x\in\mathfrak{u}_\omega 
&
\iff \left\{\begin{array}{l}
x\left(\begin{array}{cc}
\tau \\ 1_n
\end{array}\right)=0 \\[3ex]
\left(\begin{array}{cc}
1_n & -\tau
\end{array}\right)x=0
\end{array}\right.
\iff
\left\{\begin{array}{l}
x_1\tau + x_2 =0 \\ x_3\tau+x_4=0 \\ x_1-\tau x_3=0 \\ x_2-\tau x_4=0
\end{array}\right.
\\
& \iff
x=\left(\begin{array}{cc}
\tau x_3 & -\tau x_3\tau \\ x_3 & -x_3\tau
\end{array}\right).
\notag
\end{align}
This yields
\[
\conormbImk{\omega}=\Big\{\left(\begin{array}{cc}
\tau y & 0 \\ 0 & -y\tau
\end{array}\right):y\in\Mat_n\ \mbox{such that}\ (\tau y,y\tau)\in\lie{n} \times \lie{n}\Big\}.
\]
By Lemma \ref{L1}, the $K$-orbit
\[\mathcal{O}_\lambda\times\mathcal{O}_\mu\cong\Big\{\left(\begin{array}{cc}
y_1 & 0 \\ 0 & y_2
\end{array}\right):y_1\in\mathcal{O}_\lambda,\ y_2\in\mathcal{O}_\mu\Big\},\quad\mbox{where}\ (\lambda,\mu)=(\Phi_1(\tau),\Phi_2(\tau)),\]
intersects $\conormbImk{\omega}$ along a dense open subset.  
According to Lemma \ref{L3}, this implies $\Xik(\omega)=(\Phi_1(\tau),\Phi_2(\tau))=\Phi(\tau)$.
\end{proof}


\section{Exotic moment map $\mathfrak{X}/K\to\mathcal{N}_\mathfrak{s}/K$}\label{section-10}

We keep the notation and setting of Section \ref{section-9}.
In this section the Cartan space
\[\mathfrak{s}=\Big\{\left(\begin{array}{cc} 0 & \gamma \\ \delta & 0 \end{array}\right):\gamma,\delta\in\MatnC\Big\}\]
plays an important role.  
We define the projection $(\cdot)^{-\theta}:\lie{g}\to\mathfrak{s}$ along the Cartan decomposition 
$ \lie{g} = \lie{k} \oplus  \lie{s} $, i.e., 
\begin{equation*}
x=\begin{pmatrix}x_1 & x_2 \\ x_3 &
x_4\end{pmatrix}
\longmapsto
x^{-\theta}=\begin{pmatrix} 0 & x_2 \\ x_3 & 0\end{pmatrix} .
\end{equation*}
We define the map 
\begin{equation*}
\pi_\mathfrak{s}:\mathcal{Y} \xrightarrow{\quad}\mathfrak{s}, \qquad
(\mathfrak{b}'_K,\mathfrak{p}',x) \longmapsto x^{-\theta}
\end{equation*}
for the conormal variety $\mathcal{Y}$ in (\ref{5.1}).
As already pointed out in Section \ref{section-1.3}, we know that the image of
this map is contained in the nilpotent variety $\mathcal{N}_\mathfrak{s}$, 
which consists of finitely many
$K$-orbits parametrized by $\signpartitionsof{2n}$, 
the set of signed Young diagrams of size $2n$ (see Section \ref{section-2.2} for definition).

For $\omega\in (\mathfrak{T}_n^2)'$,
we have the $K$-orbit $\orbitofX_\omega $ in $ \mathfrak{X}$
(see (\ref{recall-orbitsofX})).
As before, we take the conormal bundle
$T^*_{\orbitofX_\omega}\mathfrak{X}\subset\mathcal{Y}$ over $\orbitofX_\omega$ and 
conclude that the set
$\overline{\pi_\mathfrak{s}(T^*_{\orbitofX_\omega}\mathfrak{X})}$
is an irreducible, $K$-stable, closed subvariety of
$\mathcal{N}_\mathfrak{s}$. Hence it coincides with the closure of a
unique $K$-orbit $ \frakorbit_{\Lambda} $. In this way, we get a map
\[
\Xis:(\mathfrak{T}_n^2)' \xrightarrow{\quad} \signpartitionsof{2n},
\qquad
\omega \mapsto \Lambda, 
\]
which is the combinatorial incarnation of the exotic moment map $\mathfrak{X}/K\to\mathcal{N}_\mathfrak{s}/K$ in (\ref{3}).
In the theorem below, we compute $\Xis(\omega)$ for a ``generic'' 
$ \omega = \begin{pmatrix} \tau\, \\ \,1_n \end{pmatrix} $ as 
in the case of Theorem~\ref{T4}.
We use the following combinatorial definition.

\begin{definition}\label{D6.1}
Let $(T_1,T_2)$ be a pair of Young tableaux of the same shape $\lambda$ with entries from $\{1,\ldots,n\}$. Let $\ell_1<\cdots<\ell_s$ (resp. $m_1<\cdots<m_s$) be a list of entries in $\{1,\ldots,n\}$ which do not appear in the tableau $T_1$ (resp. $T_2$).
We define a skew tableau 
\[
S = 
\tableaul{m_1 \\ { }^{\vdots} \\ m_s}*T_2\bigtriangleup
T_1*\tableaul{\ell_1 \\ { }^{\vdots} \\
\ell_{s}}
\]
by the following algorithm:
\begin{itemize}
\item Define a tableau
$\widehat{T}_1:=T_1\leftarrow \ell_s\leftarrow\cdots\leftarrow \ell_1$ by row insertion.
Then its shape $\widehat{\lambda}$ contains at most one extra box in each row comparing to $\lambda$.
Let $\widehat{T}_2$ be the tableau of the same shape $\widehat{\lambda} = \shape{\widehat{T}_1} $,  
 obtained as follows.  Place $ T_2 $ in the subshape $ \lambda \subset \widehat{\lambda} $.  
Then add to $T_2$ extra boxes with the entries
$n+1,\ldots,n+s$ from top to bottom.
\item Let 
$\overline{T}_2:=m_s\to\cdots\to m_1\to \widehat{T}_2$ be a tableau obtained by column insertion.
Its shape $\overline{\lambda}$ contains at most one extra box in each row comparing to $\widehat{\lambda}$.
Let $\mu^{(s)}$ denote the vertical Young diagram of size $s$.
Then, there is a unique skew tableau of shape $\overline{\lambda}\setminus\mu^{(s)}$
whose rectification by jeu de taquin is the tableau $\widehat{T}_1$.
Define $S$ as this skew tableau.
\end{itemize}
\end{definition}

It follows from \cite[Proposition in \S1.1]{Fulton} that the procedure in the definition is well defined. 

\begin{lemma}
Assume that $(T_1,T_2)=(\RSl(w),\RSr(w))$ is the pair of Young
tableaux corresponding to the bijection $w:j_k\mapsto i_k$. 
Let $ S $ be the skew tableau obtained in Definition \ref{D6.1}.
Then $S$ is obtained from the Robinson-Schensted tableau
\[\RSl\left(
\begin{array}{ccccccccccccccccccccccc}
m_s & \cdots & m_1 & j_1 & \cdots & j_r & n+1 & \cdots & n+s
\\
-s & \cdots & -1 & i_1 & \cdots & i_r & \ell_s & \cdots & \ell_1
\end{array}
\right)\]
by erasing the boxes of entries $-1,\ldots,-s$. 
\end{lemma}

\begin{proof}
This follows from \cite[Proposition 1 in \S 5.1]{Fulton};
see also Lemma \ref{L7} below.
\end{proof}

\begin{example}
For instance, for $T_1=\tableau{13,46,5}$, $T_2=\tableau{24,36,7}$,
$(\ell_1,\ell_2)=(2,7)$, $(m_1,m_2)=(1,5)$, $n=7$, we~get
\[\widehat{T}_1=\tableau{127,36,4,5},\ \widehat{T}_2=\tableau{248,36,7,9},\ \overline{T}_2=\tableau{1248,36,57,9}\
\mbox{hence}\quad \tableau{1,5}*T_2\bigtriangleup
T_1*\tableau{2,7}=\tableau{\none 127,\none 3,46,5}.
\]
\end{example}

\begin{theorem}\label{T5}
Let $\omega\in(\mathfrak{T}_n^2)'$ be of the form
$\omega=\left(\begin{array}{c} \tau \\ 1_n \end{array}\right)$ for a
partial permutation $\tau\in\mathfrak{T}_n$. We write
\[ \tau =\left(\begin{array}{cccccc}
j_1 & \cdots & j_r & m_1 & \cdots & m_{s} \\ i_1 & \cdots & i_r & 0 & \cdots & 0
\end{array}\right)\quad\mbox{with}\ m_1<\cdots<m_{s}, \]
where $r=\mathrm{rank}\,\tau$ and  $s=n-r$.  
Let $\{\ell_1<\cdots<\ell_{s}\}:=\{1,\ldots,n\}\setminus\{i_1,\ldots,i_r\}$.
Let us denote the nondegenerate part of $\tau$ by 
\[
\sigma:=\left(\begin{array}{ccc}
j_1 & \cdots & j_r \\
i_1 & \cdots & i_r
\end{array}\right).
\]
Then the image $\Xis(\omega) \in \signpartitionsof{2n} $ of the exotic moment map is 
characterized as follows. 
\begin{thmenumerate}
\item
For $ k > 0 $ even, 
the number of $+$'s (resp. $-$'s) contained in the first $k$ columns of $\Xis(\omega)$ coincides with the number of boxes in the first $k$ columns of the tableau
\[\RSl(\sigma)*\tableaul{\ell_1 \\ { }^{\vdots} \\ \ell_{s}}\qquad\big(\mbox{resp.}\quad \tableaul{m_1 \\ { }^{\vdots} \\ m_{s}}*\RSr(\sigma)\big).\]
\item
For $ k > 0 $ odd, the number of $+$'s contained in the first $k$ columns of $\Xis(\omega)$ coincides with the number of boxes contained in the first $k$ columns of the skew tableau 
(see Definition \ref{D6.1} for notation)
\[\tableaul{m_1 \\ { }^{\vdots} \\ m_{s}}*\RSr(\sigma)\bigtriangleup
\RSl(\sigma)*\tableaul{\ell_1 \\ { }^{\vdots} \\ \ell_{s}}.
\]
\item
For $ k > 0 $ odd, the number of $-$'s in the first $k$ columns of $\Xis(\omega)$ is equal to 
\[s+(\mbox{\rm number of boxes in the first $k$ columns of $\RS_i(\sigma)$})\quad(i\in\{1,2\}).\]
\end{thmenumerate}
\end{theorem}

\begin{example}
\label{E6.1} {\rm (a)} If $\tau=\sigma$ is a permutation, then $\RSl(\tau),\RSr(\tau)$ are standard Young tableaux of
the same shape, and the signed Young diagram
$\Xis(\omega)$ is obtained by duplicating this common shape and
filling in the rows and columns with alternated $+$'s and $-$'s. For
instance,
\[\tau=1_n\ \implies\ \mathrm{shape}(\RS_i(\tau))=(n)=\tableau{\ \ \ \ \cdots}\ \implies\ \Xis(\omega)=\tableau{+-+-\cdots,-+-+\cdots}\,.\]
See Section \ref{section-6.2} for more details.

\smallskip
{\rm (b)} Assume that
$\tau=\left(\begin{array}{cccccccccccccccc}
1 & 2 & 3 & \cdots & n \\
0 & 1 & 2 & \cdots & n-1
\end{array}\right)$.
In this case we get
\begin{align*}
& 
(\RSl(\sigma),\RSr(\sigma))=(\tableaul{1 & 2 & \cdots & n\mbox{--}1},\tableaul{2 & 3 & \cdots & n}), 
\\[1ex]
& 
\RSl(\sigma)*\tableau{n}=\tableau{1}*\RSr(\sigma)=\tableau{12\cdots n}\,, \quad \text{ and } 
\\[1ex]
&
\tableau{1}*\RSr(\sigma)\bigtriangleup\RSl(\sigma)*\tableau{n}=\tableaul{\none[\cdot] & 1 & 2 & \cdots & n}
\end{align*}
(the latter tableau is a skew tableau whose first column contains no
box), hence
\[\begin{array}[t]{c}
\Xis(\omega)=\tableau{-+-\cdots+,-+-\cdots+}\ \mbox{ if
$n$ is even} \\[4mm] \mbox{(two rows of length $n$)} \end{array}
\quad\mbox{or}\quad
\begin{array}[t]{c}
\Xis(\omega)=\tableau{-+\cdots+-+,-+\cdots+}\ \mbox{ if
$n$ is odd.} \\[4mm] \mbox{(rows of lengths $n+1,n-1$)} \end{array}\]

{\rm (c)} For $n=3$, we have computed
the signed Young diagram $\Xis(\omega)$
for each matrix of the form $\omega=\left(\begin{array}{c} \tau \\
1_3
\end{array}\right)$ with $\tau\in\mathfrak{T}_3$. These signed Young
diagrams are listed below in Figures~\ref{F1}--\ref{F3} at the end of this article.
\end{example}

The rest of this section is devoted to the proof of Theorem \ref{T5}.  
Let us begin with some preliminary lemmas.

\subsection{A characterization of the image of the exotic moment map $\Xis$}
As in Section \ref{section-9}, we prepare a lemma, which characterizes the $K$-orbit corresponding to the signed Young diagram $\Xis(\omega) \in \signpartitionsof{2n} $.
Recall that $ \lie{n}\subset\lie{gl}_n $ stands for the subalgebra of strictly upper triangular matrices.

\begin{lemma}\label{L4}
\begin{thmenumeralph}
\item
For $\omega\in(\mathfrak{T}_n^2)'$, put $ \Lambda = \Xis(\omega) \in \signpartitionsof{2n} $.  
Then the nilpotent $K$-orbit 
$ \frakorbit_{\Lambda} \subset \mathcal{N}_\mathfrak{s}$ is characterized as the $K$-orbit which intersects
\[\conormbIms{\omega} := \{x^{-\theta}:x\in\mathfrak{u}_\omega\mbox{ such that }x^\theta\in\lie{n}\times\lie{n}\} \]
along a dense open subset.
\item
Moreover, if $\omega$ is of the form $\omega=\left(\begin{array}{c} \tau \\ 1_n \end{array}\right)$ with $\tau\in\mathfrak{T}_n$, then we have
\[\conormbIms{\omega}=\Big\{\left(\begin{array}{cc}
0 & -\tau y \tau \\y & 0
\end{array}\right):y\in\MatnC\ \mbox{such that}\ (\tau y,y\tau)\in\lie{n}\times\lie{n}\Big\}.\]
\end{thmenumeralph}
\end{lemma}

\begin{proof}
By (\ref{11}), we have $T^*_{\orbitofX_\omega}\mathfrak{X}=K\cdot\{(\lie{b}\times\lie{b},\mathfrak{p}_\omega,x):x\in\mathfrak{u}_\omega,\ x^\theta\in\lie{n}\times\lie{n}\}$, hence we get
\begin{equation*}
\closure{\pi_\mathfrak{s}(T_{\orbitofX_\omega}^*\mathfrak{X})} 
= \closure{K\cdot \conormbIms{\omega}}
= \closure{\frakorbit_{\Lambda}}
\end{equation*}
by the definition of the map $\Xis$.  
This shows part {\rm (a)} of the statement. 
Part {\rm (b)} follows from the calculation made in the proof of Theorem \ref{T4} (see (\ref{12})).
\end{proof}

\subsection{The permutation case}
\label{section-6.2} We first determine $\Xis(\omega)$ in
the case where $\omega$ involves a permutation $\tau$.

\begin{notation}
Given a partition $\lambda=(\lambda_1,\ldots,\lambda_s)\in\partitionsof{n}$, we denote by $\Lambda[2\lambda]$
the signed Young diagram of size $2n$ obtained by duplicating $\lambda$, i.e., each row of $\lambda$ of length $\lambda_i$ gives rise to two rows of $\Lambda[2\lambda]$ of length $\lambda_i$, one starting with $+$ and the other starting with $-$.
For instance:
\[\lambda=\mbox{\tiny $\ydiagram{4,2,2,1}$}\quad\implies\quad\Lambda[2\lambda]=\tableau{+-+-,-+-+,+-,+-,-+,-+,+,-}.\]
In other words, with the notation of Section \ref{section-2.4.2}, the signed Young diagram $\Lambda[2\lambda]$ is characterized as follows:
\[ 
\numberofsigns{(\Lambda[2\lambda])}{k}{+} = \numberofsigns{(\Lambda[2\lambda])}{k}{-} = \numberofboxes{\lambda}{k}, 
\qquad
\forall k\geq 1 .
\]
%
\end{notation}

\begin{lemma}\label{L5}
Assume that $\omega=\left(\begin{array}{c}\tau \\ 1_n\end{array}\right)$
where
$\tau$ is a permutation.
Put 
$ \lambda = \shape{\RS_i(\tau)} $ (for $i\in\{1,2\}$).
Then we have $ \Xis(\omega)=\Lambda[2\lambda] $.
\end{lemma}

\begin{proof}
By Lemma \ref{L4},
we have to determine the signed Young diagram $\Lambda$ parametrizing the $K$-orbit of $\mathcal{N}_\mathfrak{s}$ containing 
$\left(\begin{array}{cc} 0 & -\tau y\tau \\ y & 0 \end{array}\right)$ whenever $y\in\MatnC$ is a generic element of the space
\[\{y\in\MatnC:(\tau y,y\tau)\in\lie{n} \times \lie{n}\}.\]
For such an element $y$, we know from Theorem \ref{T2} that the Jordan form of $a:=y\tau$ is given by 
$ \lambda $.  
Using that the matrix $\tau$ is a permutation (thus invertible), we can write
\[\left(\begin{array}{cc}
\tau^{-1} & 0 \\0 & 1_n
\end{array}\right)\left(\begin{array}{cc}
0 & -\tau y \tau \\y & 0
\end{array}\right)\left(\begin{array}{cc}
\tau & 0 \\0 & 1_n
\end{array}\right)=\left(\begin{array}{cc}
0 & -y\tau \\ y\tau & 0
\end{array}\right)=\left(\begin{array}{cc}
0 & -a \\a & 0
\end{array}\right)=:x\]
hence $x$ also belongs to the $K$-orbit $\mathfrak{O}_\Lambda$.
Note also that, for all $k\geq 0$, we have
\[x^{2k}=(-1)^k\left(\begin{array}{cc}
a^{2k} & 0 \\0 & a^{2k}
\end{array}\right)\quad\mbox{and}\quad
x^{2k+1}=(-1)^k\left(\begin{array}{cc}
0 & -a^{2k+1} \\ a^{2k+1} & 0
\end{array}\right).\]
Hence for all $k\geq 1$,
\[ 
\numberofsigns{\Lambda}{k}{\pm}
=\dim \ker x^k\cap V^\pm 
=\dim \ker a^k=
\numberofboxes{\lambda}{k}
\]
and therefore $\Lambda=\Lambda[2\lambda]$ as asserted.
\end{proof}

\subsection{A lemma for a partial permutation}
Let $\omega=\left(\begin{array}{c}\tau \\ 1_n\end{array}\right)$
where $\tau$ is a partial permutation. We consider an element
$y\in\MatnC$ which is generic in the space
\[\{y\in\MatnC:(\tau y,y\tau)\in\lie{n}\times\lie{n}\}\]
so that the signed Young diagram $\Lambda:=\Xis(\omega)$ parametrizes the $K$-orbit $\mathfrak{O}_\Lambda\subset\mathcal{N}_\mathfrak{s}$ which contains the element
\[x:=\left(\begin{array}{cc}0 & -\tau y \tau \\ y & 0\end{array}\right)\]
(see Lemma \ref{L4}). The proof of the following lemma is straightforward
(the second part of each claim {\rm (a)} and {\rm (b)} follows from the definition of the orbit $\mathfrak{O}_\Lambda$; see Definition \ref{defOLambda}\,{\rm (c)}).

\begin{lemma}
\label{L6}
\begin{thmenumeralph}
\item
$x^{2k}=(-1)^k\left(\begin{array}{cc} (\tau y)^{2k} & 0 \\ 0 & (y\tau)^{2k} \end{array}\right)$ for all $k\geq 0$, hence
\[
\numberofsigns{\Lambda}{2k}{+} = \dim\ker (\tau y)^{2k}
\quad\mbox{and}\quad
\numberofsigns{\Lambda}{2k}{-} = \dim\ker (y \tau)^{2k}.
\]
\item
Similarly 
$x^{2k+1}=(-1)^k\left(\begin{array}{cc}0 & -(\tau y)^{2k+1} \tau \\ (y\tau)^{2k}y & 0\end{array}\right)$
for all $k\geq 0$, hence 
\[
\numberofsigns{\Lambda}{2k+1}{+} = \dim\ker \bigl((y\tau )^{2k}y\bigr)
\quad\mbox{and}\quad
\numberofsigns{\Lambda}{2k+1}{-} = \dim\ker \bigl((\tau y)^{2k+1}\tau\bigr).
\]
\end{thmenumeralph}
\end{lemma}

\subsection{Proof of Theorem \ref{T5}\,{\rm (1)}}

We can take $y\in\{y\in\MatnC:(\tau y,y\tau)\in\lie{n} \times \lie{n}\}$ generic so that
\[\left(\begin{array}{cc}
0 & -\tau y\tau \\ y & 0
\end{array}\right)\mbox{ belongs to the $K$-orbit $\mathfrak{O}_\Lambda\subset\mathcal{N}_\mathfrak{s}$ for $\Lambda:=\Xis(\omega)$}\]
(see Lemma \ref{L4}) and
\[\mbox{$\tau y$ (resp. $y\tau$) belongs to the nilpotent orbit $\mathcal{O}_\lambda$ (resp. $\mathcal{O}_\mu$) for $(\lambda,\mu):=\Phi(\tau)$}\]
(see Lemma \ref{L1}). In view of Lemma \ref{L6}\,{\rm (a)}, it follows that the number of $+$'s (resp. $-$'s) in the first $k$ columns of $\Lambda$, for $k$ even, coincides with the number of boxes in the first $k$ columns of $\lambda$ (resp. $\mu$). 
Thus Theorem \ref{T5}\,{\rm (1)} follows from 
Theorem \ref{T2}.

\subsection{Proof of Theorem \ref{T5}\,{\rm (2)}}
We consider the map
\[\psi:\MatnC\longrightarrow \Mat_{n+s},\quad y\longmapsto \left(\begin{array}{cc}
0 & y \\ 0 & 0
\end{array}\right)\]
and the permutation
\[\widehat{\tau}:=\left(
\begin{array}{cccccccccc}
m_1 & \cdots & m_s & j_1 & \cdots & j_r & n+1 & \cdots & n+s \\
s & \cdots & 1 & i'_1 & \cdots & i'_r & \ell'_s & \cdots & \ell'_1
\end{array}
\right)\in\mathfrak{S}_{n+s}\]
where we set $i'_k:=i_k+s$ and $\ell'_k:=\ell_k+s$.

\begin{lemma}\label{L7}
Let $\lambda:=\mathrm{shape}(\RS_i(\widehat{\tau}))$ ($i\in\{1,2\}$) for $\widehat{\tau}$ as above and let $\nu$ be the one-column Young diagram with $s$ boxes. Then the skew tableau
\[ S := \tableaul{m_1 \\ { }^{\vdots} \\ m_{s}}*\RSr(\sigma)\bigtriangleup
\RSl(\sigma)*\tableaul{\ell_1 \\ { }^{\vdots} \\ \ell_{s}}\]
in Theorem \ref{T5}\,{\rm (2)} is of shape $\lambda\setminus\nu$.
\end{lemma}

\begin{proof}
Let $T_i=\RS_i(\sigma)$ for $i\in\{1,2\}$.
Let $(\widehat{T}_1,\widehat{T}_2)$ be the pair of tableaux obtained after the first step of the definition of the skew tableau $S$.
By definition of the operation $\bigtriangleup$, the shape of $S$ is obtained as
\begin{equation}
\label{13}
\mathrm{shape}\big(
m_s\to\cdots\to m_1\to \widehat{T}_2
\big)\setminus\nu.\end{equation}
By definition of the algorithm, for $i\in\{1,2\}$ we have
\[\widehat{T}_i=\RS_i(w)\quad\mbox{where}\quad w:=\left(
\begin{array}{cccccccccc}
j_1 & \cdots & j_r & n+1 & \cdots & n+s \\
i_1 & \cdots & i_r & \ell_s & \cdots & \ell_1
\end{array}
\right).\]
Let $j'_1,\ldots,j'_n$ be the elements of the set $\{j_1,\ldots,j_r,n+1,\ldots,n+s\}$ ordered in such a way that $w(j'_1)<\cdots<w(j'_n)$; equivalently $\widehat{\tau}(j'_1)(=s+1)<\cdots<\widehat{\tau}(j'_n)$.
Then we have
\begin{eqnarray*}
\RSr(\widehat{\tau}) & = & \RSl(\widehat{\tau}^{-1}) = \rinsert(m_s,\ldots,m_1,j'_1,\ldots,j'_n) \\
 & = & (m_s\to\cdots\to m_1\to\rinsert(j'_1,\ldots,j'_n)) \\
 & = & (m_s\to\cdots\to m_1\to\widehat{T}_2).
\end{eqnarray*}
Comparing this equality with (\ref{13}) completes the proof of the lemma.
\end{proof}

\begin{lemma}
\label{L8}
\begin{thmenumeralph}
\item The map $\psi$ restricts to a bijection
\[\{y\in\MatnC:(\tau y,y\tau)\in\lie{n} \times \lie{n}\}
\xrightarrow{\;\;\sim\;\;}
\{z\in\Mat_{n+s}:(\widehat{\tau} z,z\widehat{\tau})\in\widehat{\mathfrak{n}}\times\widehat{\mathfrak{n}}\}\]
where $\widehat{\mathfrak{n}}\subset\Mat_{n+s}$ stands for the subspace of strictly upper triangular matrices.
\item For any $y\in\MatnC$ such that $(\tau y,y\tau)\in\lie{n} \times \lie{n} $, we have
\[\psi((y\tau)^{k}y)=(\psi(y)\widehat{\tau})^{k}\psi(y)\quad\mbox{for all $k\geq 0$}.\]
\end{thmenumeralph}
\end{lemma}

\begin{proof}
To show part {\rm (a)}, it suffices to prove the following claims:
\begin{eqnarray}
\label{Claim1}
& \{z\in\Mat_{n+s}:(\widehat{\tau} z,z\widehat{\tau})\in\widehat{\mathfrak{n}}\times\widehat{\mathfrak{n}}\}\subset\Im \psi; \\
\label{Claim2}
& \mbox{for $y\in\MatnC$ and $\widehat{y}:=\psi(y)$, \quad $(\tau y,y\tau)\in\lie{n} \times \lie{n}\iff(\widehat{\tau} \widehat{y},\widehat{y}\widehat{\tau})\in\widehat{\mathfrak{n}}\times\widehat{\mathfrak{n}}$.}
\end{eqnarray}

Let us show (\ref{Claim1}). Let $z\in\Mat_{n+s}$ such that $(\widehat{\tau} z,z\widehat{\tau})\in\widehat{\mathfrak{n}}\times\widehat{\mathfrak{n}}$.
For $i\in\{1,\ldots, s\}$, using that $\widehat{\tau} z\in\widehat{\mathfrak{n}}$, we have
\[z_{n+i,j}=(\widehat{\tau} z)_{\ell'_{s-i+1},j}=0\quad\mbox{whenever $1\leq j\leq \ell'_{s-i+1}$}.\]
Note that $\{\ell'_{s-i+1}+1,\ldots,n+s\}\subset\{\ell'_{s-k+1}:1\leq k<i\}\cup\{i'_1,\ldots,i'_r\}$. Since $z\widehat{\tau}\in\widehat{\mathfrak{n}}$, we also have 
\[z_{n+i,i'_k}=(z\widehat{\tau})_{n+i,j_k}=0\quad\mbox{for all $k\in\{1,\ldots,r\}$\ \ (since $j_k\leq n$)}\]
and
\[z_{n+i,\ell'_{s-k+1}}=(z\widehat{\tau})_{n+i,n+k}=0\quad\mbox{whenever $1\leq k<i$}.\]
Altogether this implies that
\begin{equation}
\label{14}
z_{n+i,j}=0\quad\mbox{for all $i\in\{1,\ldots,s\}$, all $j\in\{1,\ldots,n+s\}$}.
\end{equation}
Let $j\in\{1,\ldots,s\}$. Since $\widehat{\tau} z\in\widehat{\mathfrak{n}}$, we have
\[z_{j_k,j}=(\widehat{\tau} z)_{i'_k,j}=0\quad\mbox{for all $k\in\{1,\ldots,r\}$\ \ (since $i'_k\geq s+1$)}\]
and
\[z_{m_{s-k+1},j}=(\widehat{\tau} z)_{k,j}=0\quad\mbox{whenever $j\leq k\leq s$.}\]
For $k\in\{1,\ldots,j-1\}$, using that $z\widehat{\tau}\in\widehat{\mathfrak{n}}$, we get
\[z_{m_{s-k+1},j}=(z\widehat{\tau})_{m_{s-k+1},m_{s-j+1}}=0\quad\mbox{(since $m_{s-k+1}>m_{s-j+1}$)}.\]
Since $\{1,\ldots,n\}=\{j_1,\ldots,j_r\}\cup\{m_1,\ldots,m_s\}$, altogether we obtain
\begin{equation}
\label{15}
z_{i,j}=0\quad\mbox{for all $i\in\{1,\ldots,n\}$, all $j\in\{1,\ldots,s\}$.}
\end{equation}
From (\ref{14}) and (\ref{15}), we conclude that (\ref{Claim1}) holds true.

Next, let us show (\ref{Claim2}). Note that the matrix corresponding to $\widehat{\tau}$ is of the form
\[\widehat{\tau}=\left(\begin{array}{cc}
\alpha & 0 \\ \tau & \beta
\end{array}\right)\quad\mbox{for some matrices $\alpha\in\Mat_{s,n}$, $\beta\in\Mat_{n,s}$}.\]
This yields
\begin{equation}
\label{18}
\widehat{\tau}\widehat{y}=\left(\begin{array}{cc}
0 & \gamma \\ 0 & \tau y
\end{array}\right)\quad\mbox{and}\quad
\widehat{y}\widehat{\tau}=\left(\begin{array}{cc}
y\tau & \delta \\ 0 & 0
\end{array}\right)
\end{equation}
for some $\gamma\in\Mat_{s,n}$ and $\delta\in\Mat_{n,s}$. Whence the equivalence
\[(\widehat{\tau} \widehat{y},\widehat{y}\widehat{\tau})\in\widehat{\mathfrak{n}}\times\widehat{\mathfrak{n}}\iff (\tau y,y\tau)\in\mathfrak{n}\times\mathfrak{n}\]
which establishes (\ref{Claim2}).

Let us prove part {\rm (b)} by induction in $k\geq 0$.  
The case $k=0$ is trivial. So assume that
$\psi((y\tau)^{k}y)=(\psi(y)\widehat{\tau})^{k}\psi(y)$. Set $\widehat{y}=\psi(y)$.
Using the second equality in (\ref{18}), we get 
\[
(\widehat{y}\widehat{\tau})^{k+1}\widehat{y}=\left(\begin{array}{cc}
y\tau & \delta \\ 0 & 0
\end{array}\right)\left(\begin{array}{cc} 0 & (y\tau)^{k}y \\ 0 & 0 \end{array}\right)=
\left(\begin{array}{cc} 0 & (y\tau)^{k+1}y \\ 0 & 0 \end{array}\right)
\]
whence the equality $(\widehat{y}\widehat{\tau})^{k+1}\widehat{y}=\psi((y\tau)^{k+1}y)$.
\end{proof}

\begin{proof}[Proof of Theorem \ref{T5}\,{\rm (2)}]
We set $\widehat{\omega}:=\left(\begin{array}{c}\widehat{\tau} \\ 1_{n+s}\end{array}\right)$, where $\widehat{\tau}\in\mathfrak{S}_{n+s}$ is the permutation given above. 
By Lemmas \ref{L5} and \ref{L6}\,{\rm (b)} applied to $\widehat{\omega}$ and $\widehat{\mathfrak{n}}\times\widehat{\mathfrak{n}}$, 
for a generic $\widehat{y}$ in the space $\{\widehat{y}\in\Mat_{n+s}:(\widehat{\tau}\widehat{y},\widehat{y}\widehat{\tau})\in\widehat{\mathfrak{n}}\times\widehat{\mathfrak{n}}\}$, we have
\[
\dim \ker \bigl((\widehat{y}\widehat{\tau})^{2k}\widehat{y}\bigr)
=
\numberofboxes{\bigl(\shape{\RSl(\widehat{\tau})}\bigr)}{2k+1}
\quad\forall k\geq 0.
\]
By Lemma \ref{L8}\,{\rm (a)}, we have $\widehat{y}=\psi(y)$ for $y\in\MatnC$ generic in the space $\{y\in\MatnC:(\tau y,y\tau)\in\lie{n} \times \lie{n}\}$, and by Lemma \ref{L6}\,{\rm (b)} we may suppose that
\[\dim \ker \bigl( (y\tau)^{2k}y \bigr) 
=
\numberofsigns{\bigl(\Xis(\omega)\bigr)}{2k+1}{+}
\quad\forall k\geq 0.\]
Moreover in view of Lemma \ref{L8}\,{\rm (b)} we have 
$\dim\ker \bigl( (y\tau)^{2k}y \bigr) 
=\dim\ker \bigl( (\widehat{y}\widehat{\tau})^{2k}\widehat{y} \bigr) - s $ for all $k$.
By Lemma \ref{L7}, this yields the equality stated in Theorem \ref{T5}\,{\rm (2)}.
\end{proof}

\subsection{Proof of Theorem \ref{T5}\,{\rm (3)}}

We consider the sets $I:=\{i_1,\ldots,i_r\}$ and
$J:=\{j_1,\ldots,j_r\}$, and the increasing bijections
$w_I:I\to\{1,\ldots,r\}$ and $w_J:J\to\{1,\ldots,r\}$. The bijection
$\sigma:J\to I$ gives rise to a permutation
\[\tau':=w_I\sigma w_J^{-1}\in\mathfrak{S}_r.\]
Let us consider linear maps
\[\xi:\Mat_r \to \MatnC,\ z\mapsto \overline{z}\quad
\mbox{with $\overline{z}_{j,i}=\left\{
\begin{array}{ll}
z_{w_J(j),w_I(i)} & \mbox{if $(j,i)\in J\times I$,} \\
0 & \mbox{if $(j,i)\notin J\times I$}
\end{array}
\right.$}\] and
\[
\phi:\MatnC\to\Mat_r,\ y\mapsto
y'\quad \mbox{with $y'_{w_J(j),w_I(i)}=y_{j,i}$ for all $(j,i)\in
J\times I$.}
\]
Let $\mathfrak{n}'\subset\Mat_r$ denote the
subspace of strictly upper triangular matrices.

\begin{lemma}\label{L9}
\begin{thmenumeralph}
\item 
$\phi\circ\xi=\mathrm{id}_{\Mat_r}$;
\item
For $z\in\Mat_r$ and $\overline{z}=\xi(z)$, we have:
\[(\tau' z,z\tau')\in\mathfrak{n}'\times\mathfrak{n}'\implies
(\tau \overline{z},\overline{z}\tau)\in\mathfrak{n}\times\mathfrak{n};\]
\item
For $y\in\MatnC$ and $y'=\phi(y)$, we have:
\[(\tau y,y \tau)\in\mathfrak{n}\times\mathfrak{n}\implies
(\tau' y',y'\tau')\in\mathfrak{n}'\times\mathfrak{n}'\] and
\[
{}^t( \tau (y \tau )^ky \tau )=\xi(\,{}^t(\tau'(y'\tau')^ky'\tau')\,)\quad\mbox{for
all $k\geq 0$.}
\]
\end{thmenumeralph}
\end{lemma}

\begin{proof}
First note that
\[(\phi\circ\xi(z))_{i,j}=(\xi(z))_{w_J^{-1}(i),w_I^{-1}(j)}=z_{w_J(w_J^{-1}(i)),w_I(w_I^{-1}(j))}=z_{i,j}\quad \mbox{for all $i,j\in\{1,\ldots,r\}$,}\]
whence {\rm (a)}. Before showing parts {\rm (b)} and {\rm (c)}, we note that
if $z=\phi(y)$ then
\begin{align}
\label{19}
& (\tau y)_{i_k,i_\ell}=y_{j_k,i_\ell}=z_{w_J(j_k),w_I(i_\ell)}=(\tau' z)_{w_I(i_k),w_I(i_\ell)}\,,\\
\label{20}
& (y \tau)_{j_k,j_\ell}=y_{j_k,i_\ell}=z_{w_J(j_k),w_I(i_\ell)}=(z\tau')_{w_J(j_k),w_J(j_\ell)}\,, \quad \mbox{ and }\\
\label{21}
& (\tau y \tau)_{i_k,j_\ell}=y_{j_k,i_\ell}=z_{w_J(j_k),w_I(i_\ell)}=(\tau'z\tau')_{w_I(i_k),w_J(j_\ell)}
\end{align}
for all $k,\ell\in\{1,\ldots,r\}$.

Let us show part {\rm (b)}. Assume that $(\tau'z,z\tau')\in\mathfrak{n}'\times\mathfrak{n}'$. Note that $(\tau\overline{z})_{i,j}=0$ if $i\notin I$ (due to the definition of $I$) or if $j\notin I$ (due to the definition of $\overline{z}$).
Similarly $(\overline{z}\tau)_{i,j}=0$ if $(i,j)\notin J\times J$.
By {\rm (a)} we have $z=\phi(\overline{z})$. By (\ref{19}),\ for $i,j\in I$ such that $i\geq j$, we get $(\tau\overline{z})_{i,j}=(\tau'z)_{w_I(i),w_I(j)}=0$ since $w_I(i)\geq w_I(j)$ (because $w_I$ is increasing) and $\tau'z\in\mathfrak{n}'$. By (\ref{20}) we get similarly $(\overline{z}\tau)_{i,j}=(z\tau')_{w_J(i),w_J(j)}=0$
whenever $i,j\in J$ satisfy $i\geq j$. Altogether we have shown that $(\tau\overline{z},\overline{z}\tau)\in\lie{n}\times\lie{n}$.

Let us show part {\rm (c)}. Assume that $(\tau y,y \tau)\in\lie{n}\times\lie{n}$. Then, (\ref{19}) and (\ref{20}) yield
\[
(\tau'y')_{i,j}=(\tau y)_{w_I^{-1}(i),w_I^{-1}(j)}=0
\quad\mbox{and}\quad
(y'\tau')_{i,j}=(y \tau)_{w_J^{-1}(i),w_J^{-1}(j)}=0
\]
whenever $i,j\in\{1,\ldots,r\}$ are such that $i\geq j$. Whence $(\tau'y',y'\tau')\in\lie{n}'\times\lie{n}'$. 
It remains to show the second assertion in part {\rm (c)}. 
By definition of $I,J$, we have
\begin{equation}
\label{22}
(\tau(y\tau)^ky\tau)_{i,j}=0\quad\mbox{for all $k\geq 0$},\quad\mbox{if}\quad(i,j)\notin I\times J.\end{equation}
Next fix $(i,j)\in I\times J$ and let us show the formula
\begin{equation}
\label{23}
(\tau(y\tau)^ky\tau)_{i,j}=(\tau'(y'\tau')^ky'\tau')_{w_I(i),w_J(j)}
\end{equation}
by induction on $k\geq 0$. The case $ k = 0 $ follows from (\ref{21}). Assuming that formula (\ref{23}) holds for $k$, by using (\ref{22}) and (\ref{19}), we see that
\begin{eqnarray*}
(\tau(y\tau)^{k+1}y\tau)_{i,j}=\sum_{\ell=1}^n(\tau y)_{i,\ell} (\tau(y\tau)^{k}y\tau)_{\ell,j}=
\sum_{\ell\in I}(\tau y)_{i,\ell} (\tau(y\tau)^{k}y\tau)_{\ell,j} \\
=\sum_{\ell=1}^r(\tau'y')_{w_I(i),\ell}(\tau'(y'\tau')^ky'\tau')_{\ell,w_J(j)}=(\tau'(y'\tau')^{k+1}y'\tau')_{w_I(i),w_J(j)}.
\end{eqnarray*}
This establishes (\ref{23}). Finally relations (\ref{22}) and (\ref{23}) yield the desired equality ${}^t(\tau(y\tau)^ky\tau)=\xi(\,{}^t(\tau'(y'\tau')^ky'\tau')\,)$
for all $k\geq 0$.
\end{proof}

\begin{proof}[Proof of Theorem \ref{T5}\,{\rm (3)}]
Let $y\in\MatnC$ be an element which is generic in the space $\{y\in\MatnC:(\tau y,y\tau)\in\lie{n} \times \lie{n} \}$, so that
\begin{equation}\label{24}
\numberofsigns{\bigl(\Xis(\omega)\bigr)}{2k+1}{-}
= \dim \ker \bigl( \tau(y\tau)^{2k}y\tau \bigr)
\end{equation}
(by Lemma \ref{L6}\,{\rm (b)}). 
By Lemma \ref{L9} we may assume that $y':=\phi(y)$ is generic in the space
$\{z\in\Mat_r: (\tau' z,z\tau')\in\lie{n}'\times\lie{n}'\}$ 
and, by Lemmas \ref{L5} and \ref{L6}\,{\rm (b)}, we may assume that
\begin{equation}\label{25}
\dim \ker \bigl( \tau'(y'\tau')^{2k}y'\tau' \bigr) 
= 
\numberofboxes{ \bigl( \shape{\RSl(\tau')} \bigr) }{2k+1}.
\end{equation}
In addition, by Lemma \ref{L9}\,{\rm (c)}, we have
\begin{equation}\label{26}
\dim\ker \bigl( \tau(y\tau)^{2k}y\tau \bigr) 
=\dim \ker \bigl( \tau'(y'\tau')^{2k}y'\tau' \bigr) + s .
\end{equation}
Finally note that the Young tableaux $\RSl(\tau')$ and $\RSl(\sigma)$ are of the same shape, because we have $\tau'=w_I\sigma w_J^{-1}$ where $w_I,w_J$ are increasing bijections. Then, part {\rm (3)} of Theorem \ref{T5} follows from (\ref{24}), (\ref{25}), and (\ref{26}).
\end{proof}


\section{Proof of Theorem \ref{T1}}\label{section-11}

In this section we focus on the images of the conormal variety
$\mathcal{Y}$ by the maps
$\pi_\mathfrak{k}:\mathcal{Y}\to\mathcal{N}_\mathfrak{k}$ and
$\pi_\mathfrak{s}:\mathcal{Y}\to\mathcal{N}_\mathfrak{s}$. 
Our goal is to prove Theorem \ref{T1}, which describes the irreducible components of 
the nilpotent
varieties
$\mathcal{N}_{\mathfrak{X},\mathfrak{k}}:=\overline{\pi_\mathfrak{k}(\mathcal{Y})}$
and
$\mathcal{N}_{\mathfrak{X},\mathfrak{s}}:=\overline{\pi_\mathfrak{s}(\mathcal{Y})}$.

Recall that the conormal variety can be described as
\[\mathcal{Y}=\{(\mathfrak{b}'_K,\mathfrak{p}',x)\in K/B_K\times G/P_\mathrm{S}\times \lie{g}:x\in \mathfrak{nil}(\mathfrak{p}'),\ x^\theta\in \mathfrak{nil}(\mathfrak{b}'_K)\}.\]
Here $K/B_K$ (resp. $G/P_\mathrm{S}$) is identified with the set of
Borel subalgebras $\mathfrak{b}'_K\subset\mathfrak{k}$ (resp.
parabolic subalgebras $\mathfrak{p}'\subset\lie{g}$ conjugate
to $\mathfrak{p}_\mathrm{S}$). Note that
$\bigcup_{\mathfrak{b}'_K\in K/B_K}\mathfrak{nil}(\mathfrak{b}'_K)$
coincides with the nilpotent cone $ \nilpotentsof{\lie{k}} $ of $\mathfrak{k}$, 
while $\bigcup_{\mathfrak{p}'\in
G/P_\mathrm{S}}\mathfrak{nil}(\mathfrak{p}')=G\cdot\mathfrak{nil}(\mathfrak{p}_\mathrm{S})\subset\lie{g}=\Mat_{2n}$
is the subset of nilpotent matrices of square zero (the closure of the Richardson orbit corresponding to $ \lie{p}_{\mathrm{S}} $). 
This implies
that the images of $\mathcal{Y}$ by the maps
$\pi_\mathfrak{k}:(\mathfrak{b}'_K,\mathfrak{p}',x)\mapsto x^\theta$
and $\pi_\mathfrak{s}:(\mathfrak{b}'_K,\mathfrak{p}',x)\mapsto
x^{-\theta}$ can be described as
\begin{eqnarray}
\label{7.1a} & \displaystyle \pi_{\mathfrak{k}}(\mathcal{Y})=\left\{
\left(\begin{array}{cc} a & 0 \\ 0 & b
\end{array}\right)\in\mathcal{N}_\mathfrak{k}:\exists y,z\in\MatnC\ \mbox{such that }\left(\begin{array}{cc}
a & y \\ z & b
\end{array}\right)^2=0 \right\}; \\
\label{7.1b} \quad & \displaystyle
\ \ \pi_{\mathfrak{s}}(\mathcal{Y})=\left\{ \left(\begin{array}{cc} 0 &
y \\ z & 0
\end{array}\right)\in\mathfrak{s}:\exists a,b\in\MatnC\ \mbox{nilpotent,}\ \mbox{such that }\left(\begin{array}{cc}
a & y \\ z & b
\end{array}\right)^2=0 \right\}.
\end{eqnarray}
Given $a,b,y,z\in\MatnC$, note that the equality
$\left(\begin{array}{cc} a & y \\ z & b
\end{array}\right)^2=0$ is equivalent to the following condition:
\begin{equation}
\label{7.1} a^2+yz=b^2+zy=ay+yb=za+bz=0.
\end{equation}
Since $\pi_\mathfrak{k}(\mathcal{Y})$ (resp.
$\mathcal{N}_{\mathfrak{X},\mathfrak{k}}$) is a $K$-stable subset of
$\mathcal{N}_\mathfrak{k}$, it is a union of nilpotent $K$-orbits of
the form $\mathcal{O}_\lambda\times\mathcal{O}_\mu$ for pairs of
partitions $(\lambda,\mu)\in\partitionsof{n}{\times}\partitionsof{n}$.
Similarly, since $\pi_\mathfrak{s}(\mathcal{Y})$ (resp.
$\mathcal{N}_{\mathfrak{X},\mathfrak{s}}$) is a $K$-stable subset of
$\mathcal{N}_\mathfrak{s}$, it is a union of $K$-orbits of the form
$\mathfrak{O}_\Lambda$ corresponding to certain signed Young
diagrams $\Lambda\in\signpartitionsof{2n}$.

\begin{lemma}\label{L7.1}
\begin{thmenumeralph}
\item
$\mathcal{O}_\lambda\times\mathcal{O}_\mu\subset\pi_\mathfrak{k}(\mathcal{Y})$
$\implies$
$\mathcal{O}_\mu\times\mathcal{O}_\lambda\subset\pi_\mathfrak{k}(\mathcal{Y})$;
\item
$\mathfrak{O}_\Lambda\subset\pi_\mathfrak{s}(\mathcal{Y})$
$\implies$
$\mathfrak{O}_{\overline\Lambda}\subset\pi_\mathfrak{s}(\mathcal{Y})$,
where we denote by $\overline{\Lambda}$ the signed Young diagram
obtained from $\Lambda$ by switching the $+$'s and the $-$'s.
\end{thmenumeralph}
\end{lemma}

\begin{proof}
The property follows by observing that the quadruple $(a,b,y,z)$
satisfies (\ref{7.1}) if and only if $(b,a,z,y)$ satisfies
(\ref{7.1}), and then by invoking (\ref{7.1a})--(\ref{7.1b}).
\end{proof}

\begin{notation}
{\rm (a)} Let $\Lambda\in\signpartitionsof{2n}$ be a
signed Young diagrams with $k$ rows. Let $\Lambda_i[+)$ (resp.
$\Lambda_i[-)$) be the number of $+$'s (resp. $-$'s) contained in
the $i$-th row of $\Lambda$ but not in the rightmost box of the row.
Let $\Lambda[+)$ and $\Lambda[-)$ be the partitions corresponding to
the lists of numbers $(\Lambda_1[+),\ldots,\Lambda_k[+))$ and
$(\Lambda_1[-),\ldots,\Lambda_k[-))$ 
after rearranging the 
terms in nonincreasing order and erasing the terms equal to zero if necessary.

For instance,
\[\Lambda=\tableau{+-+-+-,+-+-+,+-+-,-+-+,-+,-}\quad\implies\quad\Lambda[+)=\diagram{3,2,2,1}\quad\mbox{and}\quad\Lambda[-)=\diagram{2,2,2,1,1}\,.\]

{\rm (b)} 
We consider partitions which satisfy the following
condition:
\begin{equation}
\label{7.4} \mbox{the partition
$\lambda=(\lambda_1,\ldots,\lambda_k)$ satisfies
$\left\{\begin{array}{ll} \lambda_{2i-1}-\lambda_{2i}\in\{0,1\}$
 $\forall i\in\{1,\ldots,\lfloor\frac{k}{2}\rfloor\}, \\[1mm]
\mbox{if $k$ is odd then $\lambda_k=1$.}
\end{array}\right.$}
\end{equation}
If $x$ is a nilpotent matrix whose Jordan normal form is encoded by
the partition $\mu=(\mu_1^{\alpha_1},\ldots,\mu_\ell^{\alpha_\ell})$
with numbers $\mu_1>\ldots>\mu_\ell$ and multiplicities
$\alpha_1,\ldots,\alpha_\ell\geq 1$, then the Jordan normal form of
$x^2$ is encoded by the partition
$(\lceil\frac{\mu_1}{2}\rceil^{\alpha_1},\lfloor\frac{\mu_1}{2}\rfloor^{\alpha_1},\ldots,
\lceil\frac{\mu_\ell}{2}\rceil^{\alpha_\ell},\lfloor\frac{\mu_\ell}{2}\rfloor^{\alpha_\ell})$.
This readily implies that
\begin{equation}\label{7.5} 
\text{
\begin{minipage}{.85\textwidth}
a partition $\lambda\in\partitionsof{n}$ encodes the Jordan normal form of a
nilpotent matrix of the form $x^2$ if and only if $\lambda$
satisfies (\ref{7.4}).
\end{minipage}
}
\end{equation}
\end{notation}

\begin{lemma}\label{L7.2} 
Let $\Lambda\in\signpartitionsof{2n}$ be a
signed Young diagram.
\begin{thmenumeralph}
\item 
If the $K$-orbit
$\mathfrak{O}_\Lambda$ is contained in
$\pi_\mathfrak{s}(\mathcal{Y})$, then the partitions $\Lambda[+)$
and $\Lambda[-)$ satisfy condition (\ref{7.4}).
\item
Any 
$x\in\pi_\mathfrak{s}(\mathcal{Y})$ satisfies $x^n=0$ if $n$ is
even and $x^{n+1}=0$ if $n$ is odd. Thus, if
$\mathfrak{O}_\Lambda$ is contained in
$\pi_\mathfrak{s}(\mathcal{Y})$, then $\Lambda$ has at most $n$
(resp. $n+1$) columns if $ n $ is even (resp. odd).  
\end{thmenumeralph}
\end{lemma}

\begin{proof} {\rm (a)} Assume that
$\mathfrak{O}_\Lambda\subset\pi_\mathfrak{s}(\mathcal{Y})$. Take
$\left(\begin{array}{cc} 0 & y \\ z & 0
\end{array}\right)\in\mathfrak{O}_\Lambda$. By (\ref{7.1b}), there exist nilpotent
matrices $a,b\in\MatnC$ such that the relations in
(\ref{7.1}) hold. The subspace $\Im z$ is stabilized by the
matrix $zy$. The last equality in (\ref{7.1}) implies that
$\Im z$ is also stabilized by $b$, and the equality
$b^2+zy=0$ yields $zy|_{\Im z}=-(b|_{\Im z})^2$.
Note that the Jordan normal form of the nilpotent endomorphism
$zy|_{\Im z}:\Im z\to\Im z$ corresponds
to the partition $\Lambda[-)$. It therefore follows from (\ref{7.5})
that the partition $\Lambda[-)$ satisfies (\ref{7.4}). A similar
argument (or Lemma \ref{L7.1}) implies that $\Lambda[+)$ also
satisfies (\ref{7.4}).

{\rm (b)} Any element $x\in\pi_\mathfrak{s}(\mathcal{Y})$ is of
the form $x=\left(\begin{array}{cc} 0 & y \\ z & 0
\end{array}\right)$ with $y,z\in\MatnC$
and such that there exist nilpotent matrices
$a,b\in\MatnC$ satisfying (\ref{7.1}). Let $m$ be
any even number such that $m\geq n$. Then
\[x^m=\left(\begin{array}{cc} 0 & y \\ z & 0
\end{array}\right)^m=\left(\begin{array}{cc} (yz)^{\frac{m}{2}} & 0 \\ 0 &
(zy)^{\frac{m}{2}}
\end{array}\right)=(-1)^{\frac{m}{2}}\left(\begin{array}{cc} a^m & 0 \\ 0 &
b^m
\end{array}\right)=0.\]
The proof is complete.
\end{proof}

\begin{proof}[Proof of Theorem \ref{T1}]
Recall that the $K$-orbits of $\mathfrak{X}$ are parametrized by the
elements $\omega\in(\mathfrak{T}_n^2)'$, and each orbit
$\orbitofX_\omega$ gives rise to a conormal bundle
$T^*_{\orbitofX_\omega}\mathfrak{X}\subset \mathcal{Y}$. Thus
\[\overline{\pi_\mathfrak{k}(T^*_{\orbitofX_\omega}\mathfrak{X})}\subset\mathcal{N}_{\mathfrak{X},\mathfrak{k}}\quad\mbox{and}\quad
\overline{\pi_\mathfrak{s}(T^*_{\orbitofX_\omega}\mathfrak{X})}\subset\mathcal{N}_{\mathfrak{X},\mathfrak{s}}.\]
For a matrix $\omega=\left(\begin{array}{c} \tau
\\ 1_n \end{array}\right)$ corresponding to a partial permutation
$\tau\in\mathfrak{T}_n$, Theorems \ref{T4} and \ref{T5} describe the
$K$-orbits which are dense in
$\overline{\pi_\mathfrak{k}(T^*_{\orbitofX_\omega}\mathfrak{X})}$
and
$\overline{\pi_\mathfrak{s}(T^*_{\orbitofX_\omega}\mathfrak{X})}$.
Choosing $\tau=1_n$, we get (by Theorem \ref{T4})
\[\Xik(\omega)=(\mathrm{shape}(\RSl(1_n)),\mathrm{shape}(\RSr(1_n)))=((n),(n)),\]
hence
\[\overline{\mathcal{O}_{(n)}\times\mathcal{O}_{(n)}}=\overline{\pi_{\mathfrak{k}}(T^*_{\orbitofX_\omega}\mathfrak{X})}\subset\mathcal{N}_{\mathfrak{X},\mathfrak{k}}.\]
Since the $K$-orbit $\mathcal{O}_{(n)}\times\mathcal{O}_{(n)}$ is
dense in $\mathcal{N}_{\mathfrak{k}}$, we already obtain
$\mathcal{N}_{\mathfrak{X},\mathfrak{k}}=\mathcal{N}_{\mathfrak{k}}$.

It remains to consider $\mathcal{N}_{\mathfrak{X},\mathfrak{s}}$.
For $n=1$, the equality $\mathcal{N}_{\mathfrak{X},\mathfrak{s}}=\mathcal{N}_\mathfrak{s}$ easily follows from (\ref{7.1b}). Hereafter we assume that $n\geq 2$.
Choosing $\tau=1_n$, we get (by Theorem \ref{T5}, and in view
of Example \ref{E6.1}\,{\rm (a)})
\[\Xis(\omega)=\tableau{+-+-\cdots,-+-+\cdots}=\Lambda_0\mbox{ (as in Theorem \ref{T1})}\quad\mbox{hence}\quad \overline{\mathfrak{O}_{\Lambda_0}}=\overline{\pi_\mathfrak{s}(T^*_{\orbitofX_\omega}\mathfrak{X})}\subset\mathcal{N}_{\mathfrak{X},\mathfrak{s}}.\]
Choosing $\tau$ as in Example \ref{E6.1}\,{\rm (b)}, we have
$\Xis(\omega)=\Lambda_-$ (the signed Young diagram of
Theorem \ref{T1}\,{\rm (a)}--{\rm (b)}) hence
$\overline{\mathfrak{O}_{\Lambda_-}}\subset\mathcal{N}_{\mathfrak{X},\mathfrak{s}}$.
From Lemma \ref{L7.1}\,{\rm (b)}, we deduce
$\overline{\mathfrak{O}_{\Lambda_+}}\subset\mathcal{N}_{\mathfrak{X},\mathfrak{s}}$
with $\Lambda_+$ as in Theorem \ref{T1}\,{\rm (a)}--{\rm (b)}.
Altogether this yields
\[\overline{\mathfrak{O}_{\Lambda_+}}\cup\overline{\mathfrak{O}_{\Lambda_0}}\cup\overline{\mathfrak{O}_{\Lambda_-}}\subset\mathcal{N}_{\mathfrak{X},\mathfrak{s}}.\]
It remains to show the reversed inclusion.

First assume that $n$ is even. In this case, the signed Young
diagrams $\Lambda_0$, $\Lambda_+$, and $\Lambda_-$ are described in
Theorem \ref{T1}\,{\rm (a)}. For any $K$-orbit
$\mathfrak{O}_\Lambda\subset\pi_\mathfrak{s}(\mathcal{Y})$, the
corresponding signed Young diagram $\Lambda$ has at most $n$ columns
by Lemma \ref{L7.2}\,{\rm (b)}, hence we have
$\Lambda\preceq\Lambda_0$, $\Lambda\preceq\Lambda_+$, or
$\Lambda\preceq\Lambda_-$ (see Remark \ref{R2.1}\,{\rm (b)}), and we get
$\mathfrak{O}_\Lambda\subset\overline{\mathfrak{O}_{\Lambda_+}}\cup\overline{\mathfrak{O}_{\Lambda_0}}\cup\overline{\mathfrak{O}_{\Lambda_-}}$.
We conclude that the inclusion
$\mathcal{N}_{\mathfrak{X},\mathfrak{s}}\subset\overline{\mathfrak{O}_{\Lambda_+}}\cup\overline{\mathfrak{O}_{\Lambda_0}}\cup\overline{\mathfrak{O}_{\Lambda_-}}$
holds in this case.

Finally assume that $n$ is odd. Then, the signed Young diagrams
$\Lambda_0$, $\Lambda_+$, and $\Lambda_-$ are described in Theorem
\ref{T1}\,{\rm (b)}. Let $\mathfrak{O}_\Lambda$ be a $K$-orbit
contained in $\pi_\mathfrak{s}(\mathcal{Y})$. By Lemma
\ref{L7.2}\,{\rm (b)}, the signed Young diagram $\Lambda$ has at
most $n+1$ columns. If $\Lambda$ has at most $n$ columns, then
$\mathfrak{O}_\Lambda\subset\overline{\mathfrak{O}_{\Lambda_+}}\cup\overline{\mathfrak{O}_{\Lambda_0}}\cup\overline{\mathfrak{O}_{\Lambda_-}}$ (see Remark \ref{R2.1}\,{\rm (b)}).
It remains to consider the case where $\Lambda$ has $n+1$ columns,
i.e., the first row of $\Lambda$ has length $n+1$. Say that the last
box of this row contains the symbol $+$ (the other case is similar),
thus $\Lambda_1[-)=\frac{n+1}{2}$. If $\Lambda$ is not $\Lambda_-$, then the second row of $\Lambda$ has
length $<n-1$ or ends with the symbol $-$; in both cases we get
$\Lambda_2[-)<\frac{n-1}{2}$, hence the signed
Young diagram $\Lambda$ does not satisfy (\ref{7.4}), so that Lemma
\ref{L7.2}\,{\rm (a)} yields a contradiction. Therefore $\Lambda_+$
and $\Lambda_-$ are the only signed Young diagrams with exactly
$n+1$ columns whose corresponding $K$-orbits are contained in
$\pi_\mathfrak{s}(\mathcal{Y})$. Altogether, we obtain the desired
inclusion
$\mathcal{N}_{\mathfrak{X},\mathfrak{s}}=\overline{\pi_\mathfrak{s}(\mathcal{Y})}\subset\overline{\mathfrak{O}_{\Lambda_+}}\cup\overline{\mathfrak{O}_{\Lambda_0}}\cup\overline{\mathfrak{O}_{\Lambda_-}}$.
The proof of the theorem is complete.
\end{proof}

\begin{remark}
\label{R7.1}
{\rm (a)} Theorem \ref{T1} shows that the set $\pi_\mathfrak{k}(\mathcal{Y})$ is dense in $\mathcal{N}_\mathfrak{k}$, however this set is not closed (thus the map $\pi_\mathfrak{k}:\mathcal{Y}\to\mathcal{N}_\mathfrak{k}$ is not surjective) unless $n\leq 3$.
(For $n\leq 3$, it is straightforward to see that $\pi_\mathfrak{k}(\mathcal{Y})=\mathcal{N}_\mathfrak{k}$.)

For $n\geq 4$, let us see that a $K$-orbit of the form
$\mathcal{O}_\lambda\times\mathcal{O}_{(1^n)}$ (for
$\lambda
\in\partitionsof{n}$) is not
contained in $\pi_\mathfrak{k}(\mathcal{Y})$  whenever $\lambda_1>
3$. Indeed, take $a\in\mathcal{O}_\lambda$ and
$b=0 \in\mathcal{O}_{(1^n)} $. Assume that
$\mathcal{O}_\lambda\times
\mathcal{O}_{(1^n)}\subset\pi_{\mathfrak{k}}(\mathcal{Y})$. Then, in
view of (\ref{7.1a}), there are matrices
$y,z\in\MatnC$ satisfying (\ref{7.1}).   
Whence
\[a^3=-a(yz)=(yb)z=0,\]
so that $\lambda_1\leq 3$.

{\rm (b)} The image of the map $\pi_\mathfrak{s}:\mathcal{Y}\to
\mathfrak{s}$ is not closed, unless $n\leq 2$. Indeed, for
$n\geq 3$, Lemma \ref{L7.2} implies that the orbit
$\mathfrak{O}_\Lambda$ corresponding to the signed Young diagram
\[\Lambda=\tableaul{+ & - & + & - \\ + \\ { }^{\vdots} \\ - \\ { }^{\vdots}}\in\signpartitionsof{2n}\]
is not contained in $\pi_\mathfrak{s}(\mathcal{Y})$, whereas
Theorem \ref{T1} shows that
$\mathfrak{O}_\Lambda\subset\overline{\pi_\mathfrak{s}(\mathcal{Y})}$.
\end{remark}

\section*{Index of notation}

\begin{itemize}
\item[\ref{section-1.1}] $\mathcal{B}$,
$\mathcal{N}=\mathcal{N}_{\lie{g}}$, $\Stmap $
\item[\ref{section-1.2}] $K$, $\mathfrak{k}$, $\mathfrak{s}$, $\mathcal{N}_\mathfrak{k}$,
$\mathcal{N}_\mathfrak{s}$, $\mathfrak{X}$, $\mathcal{Y}$,
$\pi_\mathfrak{k}$, $\pi_\mathfrak{s}$
\item[\ref{section-2.2.1}] $\partitionsof{n}$, $\mathrm{shape}(T)$
\item[\ref{section-2.2.2}] $(T\leftarrow a)$, $(a\to T)$,
$\rinsert$, $\cinsert$
\item[\ref{section-2.2.3}] $\RSl(w)$, $\RSr(w)$
\item[\ref{section-2.2.4}] $T*S$
\item[\ref{section-2.1}] $\mathcal{Y}_X$, $\mu_X$
\item[\ref{section-2.1.1}] $\mathfrak{nil}$
\item[\S\ref{section-Steinberg}] $\Stmap $ (type A),
$\mathcal{O}_\lambda$
\item[\ref{section-2.4.1}] $G$, $K$, $V^+$, $V^-$, $\mathfrak{k}$, $\mathfrak{s}$, $x^\theta$, $x^{-\theta}$
\item[\ref{section-2.4.2}] $\mathcal{N}_\mathfrak{k}$, $\mathcal{N}_\mathfrak{s}$,
$\numberofboxes{\lambda}{k}$, $\preceq$, $\signpartitionsof{2n}$,
$\numberofsigns{\Lambda}{k}{+}$, $\numberofsigns{\Lambda}{k}{-}$, $\mathfrak{O}_\Lambda$, $B_K$, $P_\mathrm{S}$ (type AIII)
\item[\S\ref{section-orbits}] $\Mat_n$,
$\mathcal{Y}_{\Mat_n}$, $\mathfrak{T}_n$, $\mathbb{O}_\tau$,
$\mathcal{Y}_\tau$
\item[\S\ref{section-Phi}] $B$, $\lie{n}$, $\Phi=(\Phi_1,\Phi_2)$
\item[\S\ref{section-8}] $(\mathfrak{T}_n^2)'$
\item[\S\ref{section-9}] $\orbitofX_\omega$,
$\Xik$ 
\item[\S\ref{section-10}]
$\Xis$, $\bigtriangleup$
\item[\ref{section-6.2}] $\Lambda[2\lambda]$
\end{itemize}


\end{document}

%% file: def_exotic_nullcone.tex
\theoremstyle{plain}
\newtheorem{theorem}{Theorem}[section]
\newtheorem{lemma}[theorem]{Lemma}
\newtheorem{proposition}[theorem]{Proposition}
\newtheorem{corollary}[theorem]{Corollary}
\newtheorem{assumption}[theorem]{Assumption}
\newtheorem{conjecture}[theorem]{Conjecture}

\theoremstyle{definition}
\newtheorem{example}[theorem]{Example}
\newtheorem{definition}[theorem]{Definition}
\newtheorem{remark}[theorem]{Remark}

\numberwithin{equation}{section}

\newcommand{\C}{\mathbb{C}}

\newcommand{\lie}[1]{\mathfrak{#1}}

\newcounter{thmenum}
\newenvironment{thmenumerate}{%
\begin{list}{$(\thethmenum)$}{%
\usecounter{thmenum}
\setlength{\labelsep}{.5em}
\setlength{\labelwidth}{-7pt}
\setlength{\topsep}{0pt}
\setlength{\partopsep}{0pt}
\setlength{\parsep}{0pt}
\setlength{\leftmargin}{3pt}
\setlength{\rightmargin}{0pt}
\setlength{\itemindent}{\leftmargin}
\setlength{\itemsep}{0pt}
}}
{\end{list}}

\newenvironment{thmenumeralph}{%
\begin{list}{{\rm (\alph{thmenum})}}{%
\usecounter{thmenum}
\setlength{\labelsep}{.5em}
\setlength{\labelwidth}{-7pt}
\setlength{\topsep}{0pt}
\setlength{\partopsep}{0pt}
\setlength{\parsep}{0pt}
\setlength{\leftmargin}{3pt}
\setlength{\rightmargin}{0pt}
\setlength{\itemindent}{\leftmargin}
\setlength{\itemsep}{0pt}
}}
{\end{list}}

\newcommand{\mycomment}[1]{} 

\newcommand{\Stab}{\qopname\relax o{Stab}}
\newcommand{\Ind}{\qopname\relax o{Ind}}

\newcommand{\Aut}{\qopname\relax o{Aut}}

\newcommand{\sgn}{\qopname\relax o{sgn}}

\renewcommand{\Im}{\qopname\relax o{Im}}

\newcommand{\closure}[1]{\overline{#1}}

\newcommand{\calorbit}{\mathcal{O}}

\newcommand{\GL}{\mathrm{GL}}

\newcommand{\Sp}{\mathrm{Sp}}

\newcommand{\Mat}{\mathrm{M}}

\newcommand{\gl}{\lie{gl}}

\newcommand{\Grass}{\qopname\relax o{Gr}}

\newcommand{\mattwo}[4]{\Bigl(\begin{array}{@{\,}c@{\;\;}c@{\,}}{#1} & {#2} \\ {#3} & {#4} \end{array}\Bigr)}

\newcommand{\STab}{\mathrm{STab}}

\newcommand{\nilpotents}{\mathcal{N}}

\newcommand{\dblFV}{\mathfrak{X}}

\newcommand{\bborbit}{\mathbb{O}}

\newcommand{\conormalvar}{\mathcal{Y}}

\newcommand{\skipover}[1]{}

\newcommand{\shape}[1]{{\qopname\relax o{shape}}(#1)}

\newcommand{\Flag}{\mathscr{F}}

\newcommand{\Stmap}{\qopname\relax o{St}}
\newcommand{\Partition}{\mathscr{P}}
\newcommand{\partitionsof}[1]{\Partition({#1})}
\newcommand{\signpartitionsof}[1]{\Partition^{\pm}({#1})}
\newcommand{\permutationsof}[1]{\mathfrak{S}_{#1}}
\newcommand{\ppermutationsof}[1]{\mathfrak{T}_{#1}}

\newcommand{\numberofboxes}[2]{\# #1_{\leq #2}}
\newcommand{\numberofsigns}[3]{\# #1_{\leq #2}(#3)}

\newcommand{\gerX}{\mathfrak{X}}
\newcommand{\frakorbit}{\mathfrak{O}}

\newcommand{\nilpotentsg}{\nilpotents_{\lie{g}}}
\newcommand{\nilpotentsof}[1]{\nilpotents_{\lie{#1}}}

\newlength{\dhatheight}

\newcommand{\GLnC}{\GL_n}
\newcommand{\GLnnC}{\GL_{2n}}
\newcommand{\MatnC}{\Mat_n}
\newcommand{\Xik}{\Xi_{\lie{k}}}
\newcommand{\Xis}{\Xi_{\lie{s}}}

\newcommand{\Triplets}[1]{\Upsilon_{#1}}
\newcommand{\wconjn}[2]{{}^{#1}\lie{#2}}
\newcommand{\conormbImk}[1]{\mathfrak{V}({#1})}
\newcommand{\conormbIms}[1]{\mathfrak{W}({#1})}